\documentclass[times,3p,nopreprintline]{elsarticle}

\usepackage{framed,multirow}

\usepackage{amssymb}
\usepackage{latexsym}
\usepackage{amsmath}
\usepackage{amsthm}
\usepackage[bookmarks=false]{hyperref}

\usepackage{url}
\usepackage{xcolor}
\definecolor{newcolor}{rgb}{.8,.349,.1}
\definecolor{green}{rgb}{0.463,0.725,0.0}

\usepackage[normalem]{ulem}

\usepackage{subcaption}
\usepackage{yhmath}
\usepackage{siunitx}

\newcommand{\bx}{\boldsymbol{x}}
\newcommand{\by}{\boldsymbol{y}}
\newcommand{\bz}{\boldsymbol{z}}
\newcommand{\bv}{\boldsymbol{v}}

\newtheorem{theorem}{Theorem}[section]
\newtheorem{prop}{Proposition}[section]

\begin{document}

\begin{frontmatter}

\title{Deep‑Learning Solvers and Surrogates for Infinity and p-Laplace Problems}%

\author[2]{Tak Shing Au Yeung}
\author[2]{Ka Chun Cheung}
\author[1]{Hannah Potgieter\textsuperscript{*}} 
\ead{hannahpotgieter1@gmail.com}
\author[1]{Steven J. Ruuth}
\address[1]{Department of Mathematics, Simon Fraser University, Burnaby, British Columbia, Canada}
\author[3]{Simon See}
\address[2]{NVIDIA AI Technology Center, NVIDIA, Hong Kong, China}
\address[3]{NVIDIA AI Technology Center, NVIDIA, Singapore, Singapore}


\begin{abstract}
We investigate the use of neural network solvers for infinity and $p$-Laplace problems, which are fundamental in nonlinear analysis and have practical applications. Our approach employs Physics-Informed Neural Networks (PINNs) and Deep Operator Networks (DeepONets) to address computational challenges associated with large $p$ values, ranging from $2$ to $1000$, on various 2D and 3D domains. Our method offers advantages over traditional physics-based solvers,  especially in three dimensions where mesh-based solvers become very costly for these problems. We also establish conditional convergence results for PINN approximations of both problems and a universal approximation result for DeepONet on the parametric $p$-Poisson problem. We demonstrate the effectiveness of these neural network solvers through numerical experiments and compare their performance with conventional methods.
\end{abstract}

 \begin{keyword}
 $p$-Laplacian \sep Infinity Laplacian \sep Physics-informed neural networks \sep Deep operator networks \sep Operator learning \sep Large-$p$ limit
\end{keyword}

\end{frontmatter}

\section{Introduction} 
\label{section:intro}

Infinity and $p$-Laplace problems are of significant interest in both mathematical analysis and for their practical applications.
The $p$-Laplace operator lies at the heart of nonlinear analysis~\cite{Drbek2007The}.
The infinity Laplace operator is closely related and may be formally derived from taking the $p \to \infty$ limit of the homogeneous $p$-Laplace equation with Dirichlet boundary data~\cite{bhattacharya, jensen1993uniqueness}. This limiting process is rigorously understood in a viscosity solution framework. Thus, it is natural to consider infinity and $p$-Laplace problems jointly.
Moreover, many applications call for the use of large $p$ values such as image processing~\cite{Blomgren,Chen2006}, distance function approximation~\cite{fayolle2018p,Potgieter}, and optimal transport~\cite{fayolle2018p}.

Neural network solvers for partial differential equations (PDEs) have gained considerable attention, see for example~\cite{lu2021learning,lu2021physics,RAISSI2019686}.
Traditional physics-based solvers for $p$-Laplace problems present unique challenges, and in particular, large $p$ values
often give rise to ill-conditioned systems.
The infinity Laplacian lacks divergence structure, so accurate numerical computation with nonsmooth data is computationally
expensive and slow to converge~\cite{OBERMANsecond}.
Moreover, popular ``traditional'' solvers such as Oberman's method~\cite{OBERMANsecond} become very expensive
beyond two-dimensional problems.
These difficulties motivate the exploration of alternative numerical approaches that may offer improved scalability or flexibility, particularly in three dimensions, where mesh-based solvers for the $p$-Laplacian become very costly.
We seek to assess the performance of neural network solvers for infinity and $p$-Laplace problems and determine their
potential advantages and limitations relative to established numerical methods.

Physics-Informed Neural Networks (PINNs) were first introduced in~\cite{RAISSI2019686,raissi2017physics} to approximate solutions of nonlinear partial differential equations by training neural
networks that incorporate physical laws through the loss function.
The PINN framework embeds the governing PDE as a soft constraint, serving as a regularization mechanism that restricts the
admissible solution space.
This physics-informed loss can reduce the need for large amounts of training data.
Since their introduction, PINNs have been applied to a wide range of problems, including Poisson equations,
advection-diffusion systems, high-speed flows, and heat transfer~\cite{cai2021physics,kharazmi2021hp,lu2021deepxde,mao2020physics,wang2021understanding}.

Existing convergence theory for PINNs is primarily restricted to linear elliptic and parabolic equations,
with extensions to certain nonlinear systems under strong regularity and stability assumptions~\cite{Shin2020PINNConvergence, DeRyck2024PINNSError}. Degenerate elliptic operators such as the $p$-Laplacian and the infinity Laplacian generally fall outside the scope of these results. In this work, we extend the theoretical framework by establishing convergence results for PINN approximations for $p$-Laplace and the infinity Laplace problems under some assumptions that we will state precisely later. We also assess accuracy and robustness empirically through comparison with reference solutions.

Our focus is on the $p$-Poisson problem, which generalizes the classical Poisson problem recovered when $p=2$.
In~\cite{lu2021physics}, a variant known as physics-informed neural networks with hard constraints (hPINNs) was introduced to
address limitations arising from the soft enforcement of constraints in standard PINNs, meaning the solutions may not strictly satisfy the constraints. The authors proposed penalty and augmented Lagrangian strategies in which penalty parameters are increased iteratively,
using the solution from the previous stage as an initial guess for subsequent training. These continuation-based ideas inspire the training strategies adopted in the present work, particularly
for stabilizing optimization in the large-$p$ regime, although we do not claim theoretical convergence guarantees in this
setting.

Additionally, but with a different focus, we employ neural networks to learn the $p$-Laplace operator. Operator learning networks map between infinite-dimensional function spaces, taking input functions such as initial conditions, boundary conditions, or PDE coefficients, and producing the corresponding (approximate) PDE solution as the output function. Several operator learning architectures have been proposed, including the Fourier neural operator (FNO)~\cite{li2020fourier}, convolutional neural operators (CNO)~\cite{raonic2024convolutional}, graph neural operators~\cite{li2020neural}, and deep operator networks (DeepONet)~\cite{lu2021learning}. FNO parameterizes the operator using spectral layers acting on Fourier modes, enabling efficient learning of mappings between function spaces. CNO is a convolutional architecture designed to preserve certain structural properties across discretizations, facilitating the learning of continuous operators. Graph neural operators employ graph-based architectures and message passing mechanisms to learn operators on discretized domains. Finally, DeepONet is motivated by a universal approximation theorem for nonlinear operators and has demonstrated strong performance in learning operators for a range of parameter-dependent PDEs from relatively small datasets~\cite{lu2021learning, Chen1995OperatorUniversal}. In this work, DeepONet is employed as a surrogate model trained on solution data generated by conventional numerical solvers or PINNs for fixed values of $p$. The resulting network is used to interpolate solutions across parameter values, rather than as a standalone PDE solver. Based on the generalization behavior of DeepONets observed in~\cite{zhu2023reliable}, we investigate their ability to approximate $p$-Laplace solutions for large $p$ and to capture the $p\to\infty$ limiting behavior.

The main contributions of this work are as follows. We provide the first systematic PINN and DeepONet study of infinity and $p$-Laplace problems over the range $p=2$ to $1000$, together with conditional convergence results for PINN approximations of both problems. We introduce stabilization and continuation strategies for training in the large-$p$ regime, including $\eta$-normalization or regularization, residual clipping, top-$2\%$ outlier removal, gradient-norm clamping, and continuation in $p$, and quantify the effects of the stabilization components through ablation studies in~\ref{app:ablation} and~\ref{app:ablation-plaplace}. We also investigate DeepONet generalization across parameter values and a family of geometries, which provides flexibility over traditional mesh-based solvers.

In Section~\ref{section:PDE}, we introduce the infinity and $p$-Laplace operators and the associated boundary value problems
under consideration.
Section~\ref{section:data} describes the generation of reference solutions using a Newton-based solver with a finite element method.
Section~\ref{section:PINNS} presents the PINN formulation together with a convergence analysis under appropriate structural assumptions, while Section~\ref{section:DeepONet} describes the DeepONet approach along with analytical support. 
Numerical experiments are performed in Section~\ref{section:expts}. Finally, we provide a brief concluding discussion in Section~\ref{section:conclusion}.

\section{Model problems and boundary value formulations}
\label{section:PDE}

We consider the Dirichlet problem associated with the infinity Laplace operator.
The homogeneous case is connected to absolutely minimizing Lipschitz extensions, which are useful
for inpainting processes~\cite{Chen2006}.
We often approximate problems involving the infinity Laplacian via the $p$-Laplacian with large
$p$ values.
This approach is motivated by the variational structure and divergence form of the
$p$-Laplacian, which enable the use of finite element numerical discretizations for finite $p$.
Some applications include geodesic distance approximation~\cite{Potgieter}, optimal transport~\cite{fayolle2018p}, and shape optimization~\cite{shapeOpt}.

The $p$-Laplace operator is defined as
\begin{align*}
\Delta_{p} u_p = \nabla \cdot \left( \left|\nabla u_p\right|^{p-2}\nabla u_p\right).
\end{align*}
If we consider Dirichlet boundary conditions, we arrive at
\begin{align*}
-\Delta_{p} u_p &=  f,  \qquad\text{in }\Omega,\\
u_p &= g,  \qquad\text{on } \partial\Omega,
\end{align*}
which we refer to as a $p$-Poisson problem.

The infinity Laplacian arises formally in the limit as $p\to\infty$ of the (homogeneous)
$p$-Laplacian.
We have
\begin{align*}
\Delta_{\infty} u_{\infty} = (\nabla u_{\infty})^T D^2u_{\infty} \, \nabla u_{\infty},
\end{align*}
where $\nabla u$ denotes the gradient of $u$ and $D^2u$ denotes the Hessian matrix of $u$.
We emphasize that this operator does not admit a divergence form and is interpreted
in the viscosity sense. Moreover, the infinity Laplacian appears as a limiting equation only under
specific assumptions on the source term and boundary conditions.

We also consider the application of distance function approximation via the $p$-Laplacian. Distance function approximation is an important step for a number of applications including
surface reconstruction~\cite{Calakli}, medical imaging~\cite{jones20063d}, and computational mechanics~\cite{BISWAS2004215}. Further, there are many test problems which may be constructed for distance function approximation. The nonsmooth nature of distance functions, for example the distance to the origin in one dimension given by $|x|$ which is Lipschitz but not differentiable at the origin, also provides insight into the neural network solver robustness against limited regularity.

The mixed boundary value problem (BVP) for the $p$-Laplacian is given by
\begin{subequations}
\begin{alignat}{3}
\label{eq:1a}  -\Delta_p u_p &= 1, \quad \, &\text{ in } \Omega,  \\[3pt]
\label{eq:1b} u_p &= 0, \qquad &\text{ on } \Gamma_1 \subseteq \partial \Omega,  \\[3pt]
\label{eq:1c}   \frac{\partial u_p}{\partial n} & = 0, \qquad &\text{ on } \Gamma_2 \subseteq \partial \Omega.
\end{alignat}
\end{subequations}
where $\Omega$ is an open set, $\Gamma_1$ denotes the `feature' set to which we want to compute the geodesic distance, and $\Gamma_2$  represents the remaining unconstrained portion of the boundary. Figure~\ref{fig:dist_visual} visualizes an example of a suitable domain.

\begin{figure}
    \centering
    \includegraphics[width=0.5\linewidth]{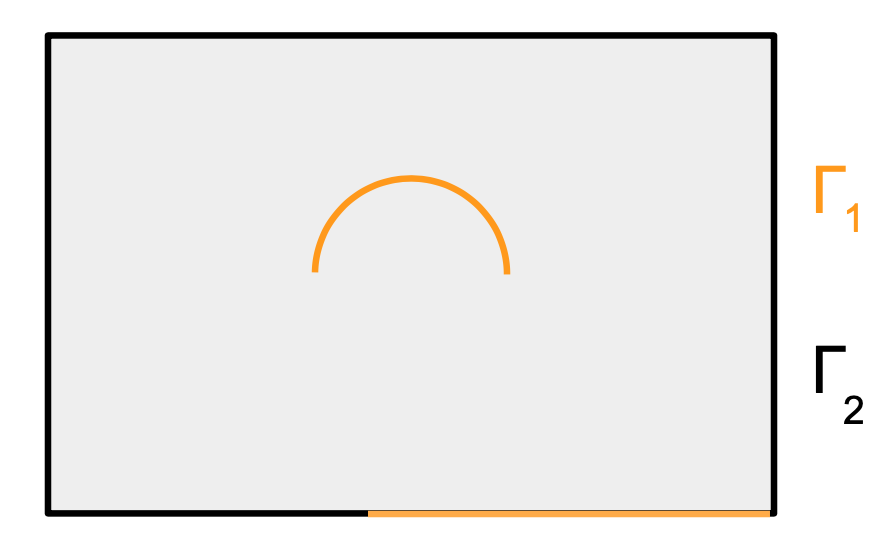}
    \caption{Example of a suitable 2D domain: $\Omega$ (grey) is an open subset of a rectangle, $\Gamma_1$ (orange) is the `feature' set where the Dirichlet condition is imposed, and  $\Gamma_2$ (black) is where the Neumann condition is imposed. } 
    \label{fig:dist_visual}
\end{figure}

We know that
\begin{equation*}
     \lim_{p \to \infty} u_p(\bx) = \mathrm{dist} (\bx, {\Gamma_1}),
\end{equation*}
where the limit is understood in an appropriate weak or viscosity sense~\cite{bhattacharya, Azorero_etal2009-mlimits}. Hence, large $p$ values are of particular interest for distance approximation. Potgieter et al.~\cite{Potgieter} report numerical experiments suggesting that $u_p$ approaches
$\mathrm{dist}(\bx,\Gamma_1)$ at an empirical rate of $\mathcal{O}(1/p)$ for the problems considered.
Distance approximations obtained from the $p$-Laplacian exhibit desirable properties such as
robustness to geometric noise and adherence to the triangle inequality~\cite{Potgieter}.

However, large $p$ values are highly susceptible to initial guess sensitivity and ill-conditioning in systems arising from iterative solvers, making them
challenging to solve using traditional numerical methods~\cite{fayolle2018p,caliari2017quasi,toulopoulos2017numerical}. Motivated by these challenges, we investigate neural network based numerical approaches for approximating solutions of $p$-Laplace and infinity Laplace problems. In this work, neural networks are employed as numerical solvers and surrogate models whose performance is assessed through systematic numerical experiments. In addition, we provide a convergence analysis for the PINN approximations under structural assumptions that are natural from an analytical perspective, but are not directly enforced in the practical implementation. The practical performance of the methods is therefore also evaluated empirically.

In particular, we employ PINNs to approximate
solutions of fixed PDE instances, and DeepONets to learn parametric mappings from problem data and parameters (such as $p$ or
domain descriptors) to solution fields. Operator learning theory provides universal approximation results for nonlinear operators~\cite{lu2021learning,Chen1995OperatorUniversal}. In Proposition~\ref{prop:deeponet_parametric_p_poisson}, we specialize these results to the parametric $p$-Poisson problem, showing that, under a continuity assumption on the solution map, a DeepONet can approximate the associated operator on compact parameter sets. However, rigorous convergence analyses for degenerate elliptic equations such as the $p$-Laplacian and the infinity Laplacian remain limited, particularly in the regimes considered here. 
The primary focus of this work is on numerical evaluation.
Newton-based FEM solutions are used to generate training data for DeepONet and to provide numerical comparison baselines. When a known infinity-Laplace solution is available, the neural-networks are evaluated against this limiting solution. We study the resulting behavior across a range of $p$ values and geometries, including large $p$ values relevant to distance approximation.

\section{Data generation}
\label{section:data}

Problems involving the infinity Laplacian are often numerically approximated by using the $p$-Laplacian with large $p$ values. This is largely due to the divergence structure of $\Delta_p$ enabling us to implement a variational finite element solver. Pryer~\cite{Pryer2018} shows that finite element approximation of the homogeneous Dirichlet $p$-Laplace problem converges to the homogeneous Dirichlet infinity Laplace problem as $p \to \infty$. Accordingly, we employ a $p$-Laplacian finite element solver to generate reference solutions for
both finite $p$ and limiting $p \to \infty$ cases.

Using a finite element method (FEM) is advantageous over finite differences for nonsmooth problems. Accordingly, we use a $p$-Laplacian FEM solver for both $p$-Laplace and homogeneous infinity Laplace problems. This approach is widely regarded as a standard one for solving nonlinear elliptic PDEs such as those involving the $p$-Laplacian in the variational regime (see, e.g.,~\cite{Pryer2018, BarrettLiu, Huang2007PreconditionedDA}). Another numerical approach is Oberman's~\cite{OBERMANsecond} monotone scheme which is shown to converge to the unique weak ($p$-Laplacian) or viscosity (infinity Laplacian) solution. While Oberman's~\cite{OBERMANsecond} method may be advantageous over Newton-based FEM when a mesh is not available, the convergence rate is dependent on angular resolution of the computational grid stencil making it prohibitively expensive at the angular resolutions required in 3D.

Let $\Omega_h$ be a discrete representation of the domain $\Omega \subset \mathbb{R}^d$, with associated mesh $\mathcal{T}_h$ and mesh size $h := \max_{K \in \mathcal{T}_h} \operatorname{diam}(K)$. We define the finite element space
\[
V_h := \{ v_h \in C^0(\Omega_h) \mid v_h|_K \in \mathbb{Q}_1(K), \; K \in \mathcal{T}_h \},
\]
where $\mathbb{Q}_1(K)$ denotes bilinear elements in two dimensions and trilinear elements in three dimensions. When Dirichlet boundary conditions are present, the finite element space is understood to incorporate these in the usual way. We restrict our attention to first-order elements, as solutions of the $p$-Laplace equation and related problems are generally not $C^2$ for $p>2$. 

Since our $p$-Laplace problem is nonlinear, we need an iterative method for approximating solutions. We use a damped Newton iteration in order to solve the nonlinear $p$-Poisson problem. Our implementation is very similar to those shown in~\cite{BarrettLiu},~\cite{Huang2007PreconditionedDA}, and~\cite{AlvarezFlores}. We implement the code in C++ using the \texttt{deal.II}~\cite{dealII96} (version 9.6.1) finite element library and generate quadrilateral meshes (in $2$D) or hexahedral (in $3$D) elements using the built-in \texttt{GridGenerator} utilities. 

For the iteration updates, we seek solutions of 
\begin{equation*}    
 F(u) := \nabla \cdot \left(\gamma(u; \eta) \nabla u\right) + f = 0,
\end{equation*}
where the coefficient with small regularization constant $\eta > 0$ is given by
\begin{equation}
\label{eqn:gamma4}
     \gamma(u; \eta) := (\eta^2+|\nabla u|^2)^{\frac{p-2}{2}} > 0.
\end{equation}
We solve for the update at Newton iteration $n+1$ (denoted by superscripts) via
\begin{align*}
    F'(u^n)[\delta u^n] & = - F(u^n) \label{NewtonIt14}, \\
    u^{n+1} & = u^n + \beta^n \delta u^n,
\end{align*}
where $F'(u)[\delta u]$ is the derivative of $F$ in direction of $\delta u$.
Here, $\eta$ is a small regularization constant which ensures positive-definiteness of the generalized stiffness matrix appearing in the resulting linearized problem, allowing a conjugate gradient (CG) solver to be used.
In our experiments we take $\eta = 10^{-5}$. The damping parameter, $0< \beta^n \leq 1$, is chosen at each iteration via a quadratic interpolation line search with a sufficient decrease condition~\cite{nocedal2006numerical}. 
That is, we find $0< \beta^n \leq 1$ such that
\begin{equation}
\label{eqn:line-search4}
    E_p(u^n + \beta^n \delta u^n; \Omega_h) \leq E_p(u^n; \Omega_h) + c_1 \beta^n E_p'(u^n; \Omega_h) [\delta u^n], 
\end{equation}
where $c_1 = 10^{-3}$ is the prescribed tolerance, $E_p$ is the discrete $p$-Poisson energy given by 
\[E_p(u; \Omega_h) = \frac1p \int_{\Omega_h} |\nabla u|^p d x_h - \int_{\Omega_h} f u \, d  x_h,\] 
and $E_p'(u) [\delta u]$ is the derivative of $E_p$ in the direction of $\delta u$. 

This gives the Newton iteration equations as 
\begin{subequations} 
\begin{alignat}{2}
       \nabla \cdot \bigg( \gamma(u^n; \eta) \nabla \delta u^n+ (p-2)\gamma(u^n; \eta) & \frac{\nabla u^n \cdot \nabla \delta u^n}{\eta^2+|\nabla u^n|^2}  \nabla u^n\bigg)
     \label{eq:pLapNewtonIt14} 
     &= - \nabla \cdot \bigg( \gamma(u^n; \eta) \nabla u^n\bigg) -f  , \\
     u^{n+1} &= u^n + \beta^n \delta u^n. 
    \end{alignat}   
\end{subequations}

The weak formulation of \eqref{eq:pLapNewtonIt14} is obtained by multiplying with a test function $\phi$ and integrating by parts on both sides. Reducing  to a finite dimensional space with basis $\{\phi_h^0, ..., \phi_h^{N-1} \} \in V_h$, we can write the update $\delta u^n \in H^1(\Omega)$, as 
\begin{equation*}
    \delta u^n = \sum_{j = 0}^{N-1} \delta U^n_j \phi_h^j.
\end{equation*}

Similarly to~\cite{BarrettLiu}, we find the resulting linear system for the update coefficients $\delta U^n$  vector at each iteration to be
\begin{equation}
     K^n  \delta U^n = b^n,
    \label{eq:linearSystem4}
\end{equation}
with the entries of the generalized stiffness matrix $K^n$ given by
\begin{equation*}
    K^n_{ij} :=\int_{\Omega_h} \nabla \phi_h^i \cdot  
     \gamma(u^n; \eta) \left ( I + \frac{(p-2)}{\eta^2 + |\nabla u^n|^2}   [\nabla u^n] \otimes [\nabla u^n] \right )\nabla \phi_h^j dx_h,
\end{equation*}
and the entries of the right hand side $b^n$ given by
\begin{equation*}
    b^n_{i} := \int_{\Omega_h} \left (f \phi_h^i - \nabla \phi_h^i \cdot \left( \gamma(u^n; \eta) \nabla u^n \right) \right)  dx_h.
\end{equation*}

The regularization parameter $\eta > 0$ ensures that the linearized operator at each Newton step is uniformly elliptic. In particular, the generalized stiffness matrix $K^n$ is symmetric positive definite, which allows the use of a conjugate gradient solver for the linear system \eqref{eq:linearSystem4}. In practice, we accelerate the conjugate gradient solver using a symmetric successive over-relaxation (SSOR) preconditioner as implemented in the \texttt{deal.II}~\cite{dealII96} finite element library. Specifically, we use the relaxation parameter $\omega = 1.2$, initialized at each Newton step with the current system matrix $K^n$. 

Continuation in $p$, together with line search damping, significantly improve robustness of the solver, particularly for large values of $p$. Starting from $p_0 = 2$, we solve the corresponding problem using Newton's method with initial
guess $u_{p_0,h}^0 = 1$.
Once convergence is achieved, the solution is used as the initial guess for the next value of
$p$.
That is, for a sequence $\{p_i\}$, we perform:
\begin{itemize}
    \item Solve for $u_{p_i,h}^{N_i}$ such that
    \[
    \frac{\|u_{p_i,h}^{N_i} - u_{p_i,h}^{N_i-1}\|_2}
         {\|u_{p_i,h}^{N_i}\|_2}
    \le \mathrm{tol}.
    \]
    \item Increment $p$ via $p_{i+1} = p_i + (\delta p)_{i+1}$.
    \item Initialize the next solve with $u_{p_{i+1},h}^0 = u_{p_i,h}^{N_i}$.
\end{itemize}
In all experiments, we take $p_0 = 2$, $(\delta p)_i = 1$, and $\mathrm{tol} = 10^{-8}$.
We will continue incrementing in $p$ until the algorithm fails to converge, usually due to failure to find a suitable descent direction according to the decrease condition \eqref{eqn:line-search4}. Typically we can take $p$ up to approximately $200$ before the latter occurs. 

The approximate solutions computed using this method serve as high-fidelity reference data for the neural network experiments presented in subsequent sections.
These solutions provide numerical comparison baselines for the PINN experiments and training data for the DeepONet models.

\section{Physics-Informed Neural Networks}
\label{section:PINNS}
We employ physics-informed neural networks (PINNs) with a penalty method~\cite{lu2021physics} as numerical solvers for fixed instances of $p$-Laplace and infinity Laplace type boundary value problems. The foundational convergence analysis for PINNs was established by Shin et al.~\cite{Shin2020PINNConvergence}, who proved consistency and convergence results for linear second-order elliptic and parabolic PDEs under suitable regularity assumptions. These results provide a theoretical framework demonstrating that, with sufficiently expressive neural network architectures and appropriate sampling strategies, PINNs can approximate solutions of classical PDEs. Extensions of PINN analysis to nonlinear PDEs remain significantly more challenging. For example, De Ryck et al.~\cite{DeRyck2024PINNSError} derived rigorous a priori error bounds for PINNs applied to the Navier-Stokes equations. 

The $p$-Laplacian and infinity Laplacian considered in this work fall outside the scope of these existing theoretical results. In particular, these operators are nonlinear, degenerate elliptic, and solutions may exhibit limited regularity, especially in the large-$p$ regime. In this work, we provide a convergence analysis for the PINN approximations in the $p$-Laplace setting under assumptions motivated by the analytical structure of the problem, though these are not explicitly enforced in the practical implementation. Nevertheless, the theory provides important motivation for using PINNs as numerical solvers and suggests that, with sufficiently rich training data and network capacity, accurate approximations may be obtained in practice.

\subsection{Network architecture}

The network architecture is illustrated in Figure~\ref{fig:PINNS}.
It is defined by a series of linear blocks; the linear blocks consists of four linear layers and activation functions with $128$ neurons and the final output layers gives one prediction, $u_\theta$. We use $\tanh$ as the activation function, so that the network outputs are smooth. However, the target solutions of the $p$-Laplace and infinity Laplace problems may exhibit limited regularity, and thus the network approximates potentially nonsmooth data using smooth functions.

\begin{figure}[h]
\centering
\includegraphics[width=0.7\textwidth]{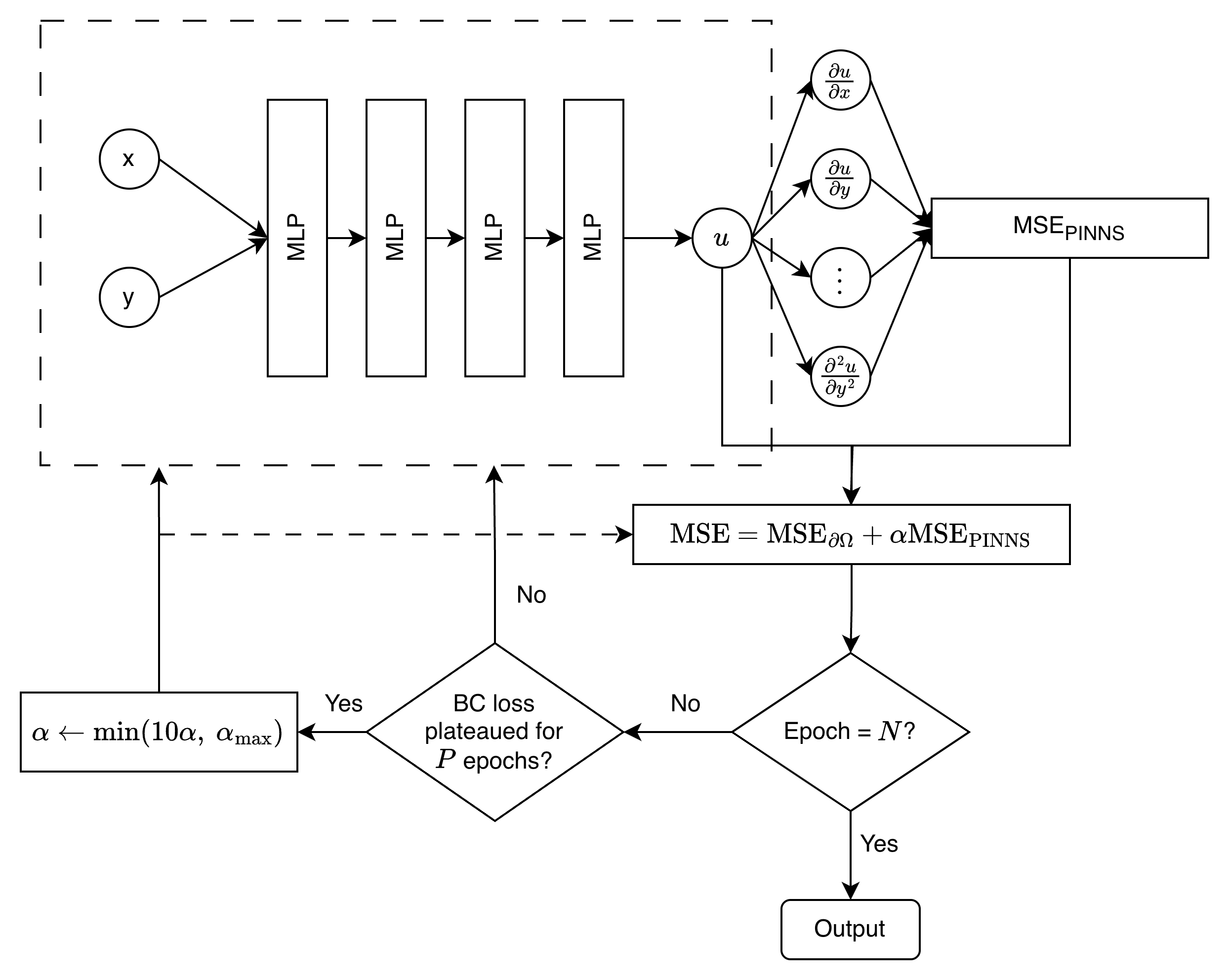}
\caption{PINN architecture and training logic for the two dimensional setting. The three dimensional case is analogous.}
\label{fig:PINNS}
\end{figure}

\subsection{Collocation sets and loss terms}

The PINNs are trained using two types of collocation points:
boundary points $\bx_i \in \partial\Omega$ and interior points $\bx_j \in \Omega$.
Let $N_{\partial\Omega}$ denote the number of boundary samples and $N_{\mathrm{int}}$ the number
of interior samples.

\paragraph{Boundary loss}
For Dirichlet conditions $u=g$ on $\partial\Omega$, we define
\begin{equation}
\label{eq:MSE_Dirichlet}
\mathrm{MSE}_{\partial\Omega}
= \frac{1}{N_{\partial\Omega}} \sum_{i=1}^{N_{\partial\Omega}}
\bigl(u_\theta(\bx_i) - g(\bx_i)\bigr)^2 .
\end{equation}
For the mixed problem \eqref{eq:1a}-\eqref{eq:1c} with $u=0$ on $\Gamma_1$ and
$\partial u/\partial n = 0$ on $\Gamma_2$, we use
\begin{equation}
\label{eq:MSE_mixed}
\mathrm{MSE}_{\partial\Omega,\mathrm{mixed}}
=
\frac{1}{N_{\Gamma_1}} \sum_{i=1}^{N_{\Gamma_1}} \bigl(u_\theta(\bx_i)\bigr)^2
+
\frac{1}{N_{\Gamma_2}} \sum_{i=1}^{N_{\Gamma_2}}
\left(\frac{\partial u_\theta}{\partial n}(\bx_i)\right)^2 .
\end{equation}
All boundary derivatives are computed using
\texttt{PyTorch} automatic differentiation~\cite{paszke2017automatic}.

\paragraph{PINNs loss}
To enforce the governing equation at interior collocation points, we define the interior PDE residual loss
\begin{equation}
\label{eq:MSE_PINNS}
\mathrm{MSE}_{\mathrm{PINNS}}
=
\frac{1}{N_{\mathrm{int}}} \sum_{j=1}^{N_{\mathrm{int}}}
\left(\mathcal{N}_p[u_\theta](\bx_j) - f(\bx_j)\right)^2,
\end{equation}
where $\mathcal{N}_p$ denotes the differential operator.
For finite $p$, we take $\mathcal{N}_p[u] := -\Delta_p u$, while for the infinity Laplacian or limiting type
tests (see Section~\ref{section:PDE}) we use the appropriate operator evaluated through automatic
differentiation.

\paragraph{Total loss}
The total training objective is
\begin{equation}
\label{eq:MSE_total}
\mathrm{MSE}
=
\mathrm{MSE}_{\partial\Omega}
+
\alpha\,\mathrm{MSE}_{\mathrm{PINNS}},
\end{equation}
(or $\mathrm{MSE}_{\partial\Omega,\mathrm{mixed}}$ in the mixed-boundary case),
where $\alpha>0$ balances boundary enforcement and PDE residual minimization.
The total loss~\eqref{eq:MSE_total} is minimized using the Adam optimizer~\cite{KingmaBa2015} with an initial learning rate of $10^{-3}$. 

For large $p$, the $\left|\nabla u\right|^{p-2}$ term in the PDE loss will be prone to blow up. When $p$ tends to infinity, this term will blow up to infinity if $| \nabla u |$ is bigger than $1$.
To make the training more stable, $\left|\nabla u\right|$ is clamped to at most~$1$ in the PDE loss for all values of~$p$.
In the case of the distance-to-boundary problem, where the limiting equation is the Eikonal equation $|\nabla u| = 1$, this strategy is particularly natural, as the true solution has unit gradient norm almost everywhere.

Another key ingredient for stable training of the $p$-Laplace operator is an adaptive schedule for the loss-balancing weight~$\alpha$.
In the simple (single-$p$) training and the first band of the iterative training, 
we employ the same adaptive plateau scheduler used for the infinity Laplacian experiments (Section~\ref{subsec:2Dinflap-dbc}): 
$\alpha$ is initialized to~$10^{-5}$ and multiplied by~$10$ whenever the boundary loss plateaus, up to a maximum of~$10^{-1}$.
This staged strategy ensures that the boundary condition is satisfied first,
after which the PDE residual is progressively enforced with increasing weight.

\subsection{PINNs convergence under structural assumptions}
\label{sec:pinns_conv}

We now present conditional convergence results for sequences of PINN approximations for the Dirichlet $p$-Laplace and (homogeneous) infinity Laplace problems.
These results provide analytical support for the continuum residual formulation,
but should not be interpreted as convergence guarantees for the discrete optimization algorithm used in the numerical experiments.
In particular, the theorems presented here are based on assumptions inspired by the universal approximation and 
PINN convergence frameworks of~\cite{Shin2020PINNConvergence,mishra2022estimates}, adapted here to the nonlinear setting.

The assumptions are natural from the perspective of PDE analysis. The use of (at least) $C^2$ activation functions is consistent with the PINN architecture considered above, since smooth activations such as $\tanh$ yield sufficiently regular network outputs for evaluation of the differential operator. Additionally, the requirement that the PDE residual and boundary mismatch vanish reflects the intended objective of PINN training. That is, the network should approximately satisfy the governing equation in the interior $\Omega$ and the boundary condition on $\partial\Omega$. These assumptions may therefore be viewed as idealized expressions of successful training.

However, these hypotheses are not enforced directly in the loss function in the exact form appearing in the theorem. In practice, the PINN loss is minimized only over finitely many collocation and boundary points, and the training process does not explicitly guarantee convergence of the residual in $W^{-1,p'}(\Omega)$ or convergence of the trace in $W^{1-1/p,p}(\partial\Omega)$. Thus, the theorem is more appropriately understood as an analytical statement concerning sequences of PINN approximations whose residual and boundary errors converge to zero asymptotically in suitable norms, rather than as a direct implication of the specific numerical optimization procedure employed. Nonetheless, these assumptions are closely aligned with the structure of the method and help explain why the PINN formulation is a reasonable approximation framework for $p$-Laplace problems. We now present the convergence result for the Dirichlet $p$-Laplace problem.

\begin{theorem}[Convergence of PINNs for $p$-Laplace Problems]
\label{thm:pinn_p_laplacian}
Let $\Omega \subset \mathbb{R}^d$ be a bounded domain with Lipschitz boundary. Consider the boundary value problem
\begin{subequations}
\begin{alignat}{2}
-\Delta_p u &= f \quad \text{in } \Omega, \\
u &= g \quad \text{on } \partial\Omega,
\end{alignat}
\label{eqn:pinns_p_bvp}
\end{subequations}
where $2 \le p < \infty$. Assume the source term satisfies $f \in L^{p'}(\Omega)$, with $p' = p/(p-1)$, and the boundary data satisfies $g \in W^{1-1/p,p}(\partial\Omega)$. Let $u^* \in W^{1,p}(\Omega)$ denote the unique weak solution of~\eqref{eqn:pinns_p_bvp} and 
let $\{u_n\}_{n=1}^{\infty}$ be a sequence of PINN approximations. Assume the activation functions are of class $C^2$, imposed only so that $-\Delta_p u_n$ is defined pointwise. Additionally, assume that 
\[
R_n := -\Delta_p u_n - f
\]
satisfies
\[
\|R_n\|_{W^{-1,p'}(\Omega)} \to 0
\quad \text{as } n \to \infty,
\]
and that
\[
\mathcal{T}(u_n)-g \to 0
\quad \text{in } W^{1-1/p,p}(\partial\Omega),
\]
where $
\mathcal{T} : W^{1,p}(\Omega) \to W^{1-1/p,p}(\partial\Omega)
$ denotes the trace operator.
Under the assumptions provided above, the following hold:
\begin{enumerate}
\item The sequence $\{u_n\}$ converges weakly to $u^*$ in $W^{1,p}(\Omega)$.
\item If, in addition, the boundary condition is enforced exactly, i.e.,
\[
\mathcal{T}(u_n)=g \quad \text{on } \partial\Omega \text{ for all } n,
\]
then the convergence is strong in $W^{1,p}(\Omega)$.
\end{enumerate}
\end{theorem}
\begin{proof} 
We proceed in several steps. First, we establish uniform boundedness of the sequence $\{u_n\}$ in $W^{1,p}(\Omega)$ and extract a weakly convergent subsequence using reflexivity. Next, we show that the weak limit satisfies the boundary condition and identify it as the unique weak solution via a monotonicity argument. Finally, we prove strong convergence under exact boundary enforcement.
The argument relies on the weak formulation of the $p$-Laplace equation, properties of reflexive Banach spaces, and the theory of monotone operators. We denote by $\langle \cdot,\cdot\rangle$ the duality pairing between $W^{-1,p'}(\Omega)$ and $W^{1,p}_0(\Omega)$.

Since $u^* \in W^{1,p}(\Omega)$ denotes the unique weak solution of~\eqref{eqn:pinns_p_bvp}, it satisfies the boundary condition
\[
\mathcal{T}(u^*) = g \quad \text{on } \partial\Omega,
\]
and the weak formulation
\begin{equation}
\int_\Omega |\nabla u^*|^{p-2}\nabla u^* \cdot \nabla v \, dx
=
\int_\Omega f v \, dx,
\qquad
\forall v \in W_0^{1,p}(\Omega).
\label{eq:weak_form}
\end{equation}

\bigskip

\noindent \emph{1. Weak convergence.}
\smallskip

\noindent We first consider the general setting where boundary conditions are enforced through a penalty term in the loss function. Our initial objective is to establish uniform boundedness of the sequence $\{u_n\}$ in $W^{1,p}(\Omega)$.

To handle the approximate boundary condition, we introduce a correction term that removes the boundary mismatch.
To correct the boundary mismatch, we use a bounded right inverse of the trace operator. Let
\[
\mathcal{E}: W^{1-1/p,p}(\partial\Omega) \to W^{1,p}(\Omega)
\]
be a bounded right inverse of the trace operator satisfying $\mathcal{T}(\mathcal{E}\psi)=\psi$ on $\partial\Omega$ (see \cite[Sec.~2.5.7, Thm.~5.7]{Necas2012}). Define
\[
E_n := \mathcal{E}(\mathcal{T}(u_n)-g).
\]
Then
\[
\mathcal{T}(E_n)=\mathcal{T}(u_n)-g
\]
and, by continuity of $\mathcal{E}$, there exists $C>0$ such that
\[
\|E_n\|_{W^{1,p}(\Omega)}
\le
C\|\mathcal{T}(u_n)-g\|_{W^{1-1/p,p}(\partial\Omega)}
\to 0 .
\]

Let $\bar g \in W^{1,p}(\Omega)$ be a fixed extension of $g$, and define
\[
v_n = u_n - E_n - \bar g .
\]
By construction, $\mathcal{T}(v_n)=0$, hence $v_n \in W_0^{1,p}(\Omega)$ and may be used as a test function for the residual.

Using the definition of the residual $R_n$, we obtain
\begin{align*}
\langle R_n, v_n \rangle
&=
\int_\Omega |\nabla u_n|^{p-2}\nabla u_n \cdot \nabla v_n \, dx
-
\int_\Omega f v_n \, dx \\
&=
\int_\Omega |\nabla u_n|^p \, dx
-
\int_\Omega |\nabla u_n|^{p-2}\nabla u_n \cdot \nabla(E_n+\bar g)\,dx
-
\int_\Omega f v_n \, dx.
\end{align*}
Rearranging then yields
\begin{align}
\|\nabla u_n\|_{L^p(\Omega)}^p
&=
\langle R_n,v_n\rangle
+
\int_\Omega |\nabla u_n|^{p-2}\nabla u_n \cdot \nabla(E_n+\bar g)\,dx
+
\int_\Omega f v_n\,dx. 
\label{eq:coercivity_expansion}
\end{align}
We now estimate each term in~\eqref{eq:coercivity_expansion} to obtain a uniform $L^p$ bound on the gradients.
Using Poincaré's inequality for $v_n \in W_0^{1,p}(\Omega)$, there exists constant $C_1 > 0$ such that 
\[
\|v_n\|_{W^{1,p}(\Omega)}
\le
C_1\|\nabla v_n\|_{L^p(\Omega)}
\le
C_1\bigl(\|\nabla u_n\|_{L^p(\Omega)}+\|\nabla E_n\|_{L^p(\Omega)}+\|\nabla\bar g\|_{L^p(\Omega)}\bigr),
\]
and so first term in~\eqref{eq:coercivity_expansion} satisfies
\begin{align*}
|\langle R_n,v_n\rangle|
&\le
\|R_n\|_{W^{-1,p'}(\Omega)}\|v_n\|_{W^{1,p}(\Omega)}\\
& \le C_1 \|R_n\|_{W^{-1,p'}(\Omega)}\bigl(\|\nabla u_n\|_{L^p(\Omega)}+\|\nabla E_n\|_{L^p(\Omega)}+\|\nabla\bar g\|_{L^p(\Omega)}\bigr),
\end{align*}
where for the first inequality we have used the definition of the dual norm.
Similarly, for the third term in~\eqref{eq:coercivity_expansion}, Hölder's inequality and Poincaré's inequality give
\begin{align*}
\left|\int_\Omega f v_n\,dx\right|
&\le
\|f\|_{L^{p'}(\Omega)}\|v_n\|_{L^p(\Omega)} \\
&\le
C_2\|f\|_{L^{p'}(\Omega)}\|\nabla v_n\|_{L^p(\Omega)}
 \\
&\le
C_2\|f\|_{L^{p'}(\Omega)}\bigl(\|\nabla u_n\|_{L^p(\Omega)}+\|\nabla E_n\|_{L^p(\Omega)}+\|\nabla\bar g\|_{L^p(\Omega)}\bigr),
\end{align*}
for some constant $C_2 >0$. 
For the middle term in~\eqref{eq:coercivity_expansion}, Hölder's inequality gives
\begin{align*}
\int_\Omega |\nabla u_n|^{p-2}\nabla u_n \cdot \nabla(E_n+\bar g)\,dx
&\le
\int_\Omega |\nabla u_n|^{p-1} |\nabla(E_n+\bar g)|\,dx \\
&\le
\||\nabla u_n|^{p-1}\|_{L^{p'}(\Omega)}
\|\nabla(E_n+\bar g)\|_{L^p(\Omega)} \\
&=
\|\nabla u_n\|_{L^p(\Omega)}^{p-1}\|\nabla(E_n+\bar g)\|_{L^p(\Omega)}.
\end{align*}
Substituting these bounds back into~\eqref{eq:coercivity_expansion} gives 
\begin{equation}
\begin{split}
\|\nabla u_n\|_{L^p(\Omega)}^p
&=
\langle R_n,v_n\rangle
+
\int_\Omega |\nabla u_n|^{p-2}\nabla u_n \cdot \nabla(E_n+\bar g)\,dx
+
\int_\Omega f v_n\,dx \\
&\le C \bigl( \|R_n\|_{W^{-1,p'}(\Omega)} +\|f\|_{L^{p'}(\Omega)} \bigr) \bigl(\|\nabla u_n\|_{L^p(\Omega)}+\|\nabla E_n\|_{L^p(\Omega)}+\|\nabla\bar g\|_{L^p(\Omega)}\bigr) 
\\& \qquad\qquad
+ \|\nabla u_n\|_{L^p(\Omega)}^{p-1}\|\nabla(E_n+\bar g)\|_{L^p(\Omega)},
\end{split}
\label{eqn:gradu-to-apply-Young}
\end{equation}
where $C = \max\{C_1, C_2\}$. 

Let $\epsilon>0$. By $\epsilon$-Young's inequality, there exist constants
$C_\epsilon, C'_\epsilon>0$ such that
\[
ab \le \epsilon a^p + C_\epsilon b^{p'},
\qquad
ab \le \epsilon a^{p'} + C'_\epsilon b^p,
\]
for all $a,b\ge0$. Applying these inequalities to the terms involving
$\|\nabla u_n\|_{L^p(\Omega)}$ on the right-hand side of
\eqref{eqn:gradu-to-apply-Young}, we obtain
\begin{equation}
\begin{split}
\|\nabla u_n\|_{L^p(\Omega)}^p
&\le C \bigl( \|R_n\|_{W^{-1,p'}(\Omega)} +\|f\|_{L^{p'}(\Omega)} \bigr)
\bigl(\|\nabla E_n\|_{L^p(\Omega)}+\|\nabla\bar g\|_{L^p(\Omega)}\bigr) \\
&\qquad
+ \epsilon\|\nabla u_n\|_{L^p(\Omega)}^p
+ C_\epsilon C^{p'}
\bigl( \|R_n\|_{W^{-1,p'}(\Omega)} +\|f\|_{L^{p'}(\Omega)} \bigr)^{p'} \\
&\qquad
+ \epsilon\|\nabla u_n\|_{L^p(\Omega)}^p
+ C'_\epsilon\|\nabla(E_n+\bar g)\|_{L^p(\Omega)}^p .
\end{split}
\label{eqn:applied-Young}
\end{equation}
Choosing $\epsilon<1/2$ and absorbing the terms
$\epsilon\|\nabla u_n\|_{L^p(\Omega)}^p$ into the left-hand side yields
\begin{equation*}
    \begin{split}
        (1-2\epsilon)\|\nabla u_n\|_{L^p(\Omega)}^p
&\le 
C \bigl( \|R_n\|_{W^{-1,p'}(\Omega)} +\|f\|_{L^{p'}(\Omega)} \bigr)
\bigl(\|\nabla E_n\|_{L^p(\Omega)}+\|\nabla\bar g\|_{L^p(\Omega)}\bigr) \\
&\qquad
+ C_\epsilon C^{p'}
\bigl( \|R_n\|_{W^{-1,p'}(\Omega)} +\|f\|_{L^{p'}(\Omega)} \bigr)^{p'}
+ C'_\epsilon\|\nabla(E_n+\bar g)\|_{L^p(\Omega)}^p .
    \end{split}
\end{equation*}
Since $\|R_n\|_{W^{-1,p'}(\Omega)} \to 0$,
$\|f\|_{L^{p'}(\Omega)}$ is fixed, $\|E_n\|_{W^{1,p}(\Omega)} \to 0$ implies $\|\nabla E_n\|_{L^p(\Omega)}$ is bounded, and $\bar g$ is fixed in $W^{1,p}(\Omega)$, the right-hand side is bounded independently of $n$. Therefore, there exists a constant $K>0$ such that
\[
(1-2\epsilon)\|\nabla u_n\|_{L^p(\Omega)}^p \le K
\qquad \text{for all } n.
\]
Because $1-2\epsilon>0$, it follows that
\[
\|\nabla u_n\|_{L^p(\Omega)}^p \le \frac{K}{1-2\epsilon},
\]
and so $\|\nabla u_n\|_{L^p(\Omega)}$ is uniformly bounded.

We now extend this to a full $W^{1,p}(\Omega)$ bound on $\{u_n\}$.
Recall that
$
v_n = u_n - E_n - \bar g,
$
with $v_n \in W_0^{1,p}(\Omega)$. By Poincaré's inequality, for some constant $\tilde C>0$ we have
\[
\|v_n\|_{L^p(\Omega)} \le \tilde C \|\nabla v_n\|_{L^p(\Omega)}.
\]
Moreover,
\[
\nabla v_n = \nabla u_n - \nabla E_n - \nabla \bar g,
\]
so
\[
\|\nabla v_n\|_{L^p(\Omega)}
\le
\|\nabla u_n\|_{L^p(\Omega)}
+
\|\nabla E_n\|_{L^p(\Omega)}
+
\|\nabla \bar g\|_{L^p(\Omega)}.
\]
Since $\|\nabla u_n\|_{L^p(\Omega)}$ is uniformly bounded, $\|E_n\|_{W^{1,p}(\Omega)}\to0$, and $\bar g$ is fixed, it follows that $\|\nabla v_n\|_{L^p(\Omega)}$ is uniformly bounded. Therefore, by Poincaré's inequality, $\|v_n\|_{L^p(\Omega)}$ is also uniformly bounded, and therefore $\{v_n\}$ is uniformly bounded in $W^{1,p}(\Omega)$.
Thus, from
\[
u_n = v_n + E_n + \bar g,
\]
and the triangle inequality, we conclude that $\{u_n\}$ is uniformly bounded in $W^{1,p}(\Omega)$. 

Having established uniform boundedness, the bounded sequence $\{u_n\}$ in $W^{1,p}(\Omega)$ admits a subsequence, which we denote by $\{u_{n_k}\}$, and a function $\tilde u \in W^{1,p}(\Omega)$ such that $u_{n_k}$ converges weakly to $\tilde u$ in $W^{1,p}(\Omega)$, since $W^{1,p}(\Omega)$ is reflexive for $1<p<\infty$.

We now show that the weak limit $\tilde u$ satisfies the boundary condition.
The trace operator is continuous, and therefore $\mathcal{T}(u_{n_k})$ converges weakly to $\mathcal{T}(\tilde u)$ in $W^{1-1/p,p}(\partial\Omega)$. Also, by assumption, the boundary loss implies that
\[
\mathcal{T}(u_{n_k}) \to g
\quad \text{strongly in } W^{1-1/p,p}(\partial\Omega).
\]
Since weak and strong limits must coincide, it follows that
\[
\mathcal{T}(\tilde u) = g.
\]
In particular, $\tilde u$ satisfies the boundary condition in the trace sense and we have  $\tilde{u} - \bar g \in W_0^{1,p}(\Omega)$. 

It remains to verify that the weak limit $\tilde u$ solves the $p$-Laplace problem. To achieve this we will use a Minty argument, which utilizes the framework of monotone operators~\cite{showalter1997monotone, brezis2011functional}.
Let $v \in W^{1,p}(\Omega)$ be fixed and satisfy $v-\bar g \in W_0^{1,p}(\Omega)$.  By monotonicity of the map $\xi \mapsto |\xi|^{p-2}\xi$, we have
\begin{equation}
\int_\Omega \left( |\nabla u_{n_k}|^{p-2}\nabla u_{n_k} - |\nabla v|^{p-2}\nabla v \right)\cdot \nabla(u_{n_k}-v)\,dx \ge 0.
\label{eq:minty_mono}
\end{equation}

Write the decomposition
\[
u_{n_k}-v = (u_{n_k}-v-E_{n_k}) + E_{n_k},
\]
and define
\[
w_{n_k} := u_{n_k}-v-E_{n_k}.
\]
Since
$
\mathcal{T}(E_{n_k})=\mathcal{T}(u_{n_k})-g
$
and $\mathcal{T}(v)=g$, we obtain
\[
\mathcal T(w_{n_k})
=
\mathcal T(u_{n_k})-\mathcal T(v)-\mathcal T(E_{n_k})
=
\mathcal T(u_{n_k})-g-(\mathcal T(u_{n_k})-g)
=
0.
\]
Thus, we get $w_{n_k}\in W_0^{1,p}(\Omega)$ which permits its use as a test function. Then
\begin{align*}
\int_\Omega |\nabla u_{n_k}|^{p-2}\nabla u_{n_k}\cdot \nabla(u_{n_k}-v)\,dx
&=
\int_\Omega |\nabla u_{n_k}|^{p-2}\nabla u_{n_k}\cdot \nabla(w_{n_k}+E_{n_k})\,dx \\
&=
\langle -\Delta_p u_{n_k}, w_{n_k}\rangle
+
\int_\Omega |\nabla u_{n_k}|^{p-2}\nabla u_{n_k}\cdot \nabla E_{n_k}\,dx \\
&=
\langle f+R_{n_k}, w_{n_k}\rangle
+
\int_\Omega |\nabla u_{n_k}|^{p-2}\nabla u_{n_k}\cdot \nabla E_{n_k}\,dx \\
&=
\langle f+R_{n_k}, u_{n_k}-v-E_{n_k}\rangle
+
\int_\Omega |\nabla u_{n_k}|^{p-2}\nabla u_{n_k}\cdot \nabla E_{n_k}\,dx .
\end{align*}
Substituting this into~\eqref{eq:minty_mono}, we get
\begin{align}
\langle f+R_{n_k}, u_{n_k}-v-E_{n_k}\rangle
+
\int_\Omega |\nabla u_{n_k}|^{p-2}\nabla u_{n_k}\cdot \nabla E_{n_k}\,dx
\ge
\int_\Omega |\nabla v|^{p-2}\nabla v \cdot \nabla(u_{n_k}-v)\,dx .
\label{eq:minty_pre_limit}
\end{align}
Since $\|\nabla u_{n_k}\|_{L^p}$ is uniformly bounded and $\|E_{n_k}\|_{W^{1,p}(\Omega)}\to0$, for the second term on the left-hand side of~\eqref{eq:minty_pre_limit} it follows from Hölder's inequality that  
\[
\left|\int_\Omega |\nabla u_{n_k}|^{p-2}\nabla u_{n_k}\cdot \nabla E_{n_k}\,dx\right|
\le
\|\nabla u_{n_k}\|_{L^p(\Omega)}^{p-1}\|\nabla E_{n_k}\|_{L^p(\Omega)}
\to 0.
\]
Next, we split
\[
\langle f+R_{n_k}, u_{n_k}-v-E_{n_k}\rangle
=
\langle f, u_{n_k}-v-E_{n_k}\rangle
+
\langle R_{n_k}, u_{n_k}-v-E_{n_k}\rangle.
\]
Since $\{u_{n_k}\}$ is bounded in $W^{1,p}(\Omega)$, $v$ is fixed, and $E_{n_k}\to0$ in $W^{1,p}(\Omega)$, it follows that $\{u_{n_k}-v-E_{n_k}\}$ is bounded in $W^{1,p}(\Omega)$. Further, since  $R_{n_k}\to0$ in $W^{-1,p'}(\Omega)$, we have
\[
\langle R_{n_k}, u_{n_k}-v-E_{n_k}\rangle \to 0.
\]
Moreover, $u_{n_k}$ converges weakly to $\tilde u$ in $W^{1,p}(\Omega)$, hence also weakly in $L^p(\Omega)$ by continuity of the embedding $W^{1,p}(\Omega) \hookrightarrow L^p(\Omega)$, while $E_{n_k}\to0$ in $L^p(\Omega)$. Therefore
\[
\langle f, u_{n_k}-v-E_{n_k}\rangle \to \langle f,\tilde u-v\rangle.
\]
Finally, since $v$ is fixed and
$|\nabla v|^{p-2}\nabla v \in L^{p'}(\Omega)$, weak convergence of
$\nabla u_{n_k}$ in $L^p(\Omega)$ yields
\[
\int_\Omega |\nabla v|^{p-2}\nabla v \cdot \nabla(u_{n_k}-v)\,dx
\to
\int_\Omega |\nabla v|^{p-2}\nabla v \cdot \nabla(\tilde u-v)\,dx.
\]
Passing to the limit in~\eqref{eq:minty_pre_limit} therefore gives
\begin{equation}
    \langle f,\tilde u-v\rangle
\ge
\int_\Omega |\nabla v|^{p-2}\nabla v \cdot \nabla(\tilde u-v)\,dx
=
\langle -\Delta_p v,\tilde u-v\rangle.
\label{eqn:minty_post_limit}
\end{equation}

Now pick $v=\tilde u-t\phi$, where $t>0$ and $\phi\in W_0^{1,p}(\Omega)$. Then
\[
\tilde u-t\phi-\bar g = (\tilde u-\bar g)-t\phi \in W_0^{1,p}(\Omega),
\]
so substituting into~\eqref{eqn:minty_post_limit} gives
\[
t\langle -\Delta_p(\tilde u-t\phi),\phi\rangle
\le
t\langle f,\phi\rangle.
\]
Dividing by $t>0$, we obtain
\begin{equation}
    \langle -\Delta_p(\tilde u-t\phi),\phi\rangle
\le
\langle f,\phi\rangle.
\label{eqn:pre-hemi-lim}
\end{equation}
We will now use the hemicontinuity of the operator $-\Delta_p : W^{1,p}(\Omega) \to W^{-1,p'}(\Omega)$. The operator $-\Delta_p$ is monotone and hemicontinuous; see \cite[Ch.~II]{showalter1997monotone}.
In particular, hemicontinuity ensures that
\[
t \mapsto \langle -\Delta_p(\tilde u - t\phi), \phi \rangle
\]
is continuous. Thus, letting $t\to0^+$ in~\eqref{eqn:pre-hemi-lim}, 
it follows that
\[
\langle -\Delta_p\tilde u,\phi\rangle
\le
\langle f,\phi\rangle.
\]

Replacing $\phi$ by $-\phi$ yields the reverse inequality, and therefore
\[
\langle -\Delta_p\tilde u,\phi\rangle
=
\langle f,\phi\rangle
\qquad
\forall \phi \in W_0^{1,p}(\Omega).
\]
Thus, $\tilde u$ is a weak solution. By uniqueness of the Dirichlet $p$-Laplace problem~\cite{LindqvistpLap}, $\tilde u=u^*$.
Since every weakly convergent subsequence has the same limit, the entire sequence $\{u_n\}$ converges weakly to $u^*$ in $W^{1,p}(\Omega)$. This concludes the proof of the first statement.\\

\noindent \emph{2. Strong convergence.}
\smallskip

\noindent For the second statement, it is part of our hypotheses that 
\[
\mathcal{T}(u_n)=g \quad \text{on } \partial\Omega \text{ for all } n.
\]
Also, $\mathcal{T}(u^*)=g$ because $u^*$ is the weak solution of~\eqref{eqn:pinns_p_bvp}, linearity of the trace operator gives
\[
\mathcal{T}(u_n-u^*)=\mathcal{T}(u_n)-\mathcal{T}(u^*)=g-g=0.
\]
Therefore, we have $u_n-u^* \in W_0^{1,p}(\Omega)$ meaning the difference $ u_n-u^*$ is a valid test function.

Using the uniform monotonicity of the map $\xi \mapsto |\xi|^{p-2}\xi$,
\[
\int_\Omega
\bigl(|\nabla u_n|^{p-2}\nabla u_n-|\nabla u^*|^{p-2}\nabla u^*\bigr)
\cdot \nabla(u_n-u^*)\,dx
\ge
c_p\|\nabla(u_n-u^*)\|_{L^p}^p .
\]
Therefore,
\begin{equation}
    \begin{split}
        \|\nabla(u_n-u^*)\|_{L^p(\Omega)}^p
&\le
\frac{1}{c_p}
\int_\Omega
\bigl(|\nabla u_n|^{p-2}\nabla u_n-|\nabla u^*|^{p-2}\nabla u^*\bigr)
\cdot \nabla(u_n-u^*)\,dx \\
&=
\frac{1}{c_p}
\Bigl(
\langle -\Delta_p u_n, u_n-u^*\rangle
-
\langle -\Delta_p u^*, u_n-u^*\rangle
\Bigr) \\
&=
\frac{1}{c_p}\langle R_n, u_n-u^*\rangle,
    \end{split}
    \label{eqn:Lp_grad_simon}
\end{equation}
where we have used that $u_n-u^*\in W_0^{1,p}(\Omega)$.

Poincaré's inequality then gives
\[
\|u_n-u^*\|_{W^{1,p}(\Omega)}^p
\le
(1+C_p^p)\|\nabla(u_n-u^*)\|_{L^p(\Omega)}^p,
\]
for some constant $C_p>0$. 
Substituting this bound into~\eqref{eqn:Lp_grad_simon} and using Hölder's inequality, we arrive at
\[ 
\|u_n-u^*\|_{W^{1,p}(\Omega)}^p
\le
\frac{(1+C_p^p)}{c_p}\|R_n\|_{W^{-1,p'}(\Omega)}\|u_n-u^*\|_{W^{1,p}(\Omega)},
\]
and therefore
\[
\|u_n-u^*\|_{W^{1,p}(\Omega)}^{p-1}
\le
\frac{(1+C_p^p)}{c_p}\|R_n\|_{W^{-1,p'}(\Omega)}.
\]
Since $\|R_n\|_{W^{-1,p'}(\Omega)}\to0$ and $p>1$, we conclude that
\[
\|u_n-u^*\|_{W^{1,p}(\Omega)}\to0.
\]
Thus the convergence is strong when the boundary condition is enforced exactly.\\
\end{proof}
 
The above result establishes conditional convergence for sequences satisfying the stated continuum residual and trace assumptions.
We now turn to the infinity Laplace problem, where solutions are understood in the viscosity sense as opposed to the weak sense.

\begin{theorem}[Convergence of PINNs for infinity Laplace Problems]
\label{thm:pinn_inf_laplacian}
Let $\Omega \subset \mathbb{R}^d$ be a bounded domain with Lipschitz boundary. Consider the boundary value problem
\begin{subequations}
\begin{alignat}{2}
-\Delta_\infty u &= 0 \quad \text{in } \Omega, \\
u &= g \quad \text{on } \partial\Omega,
\end{alignat}
\label{eqn:pinns_inf_bvp}
\end{subequations}

\noindent where $g \in \mathrm{Lip}(\partial\Omega)$. Let $u^* \in W^{1,\infty}(\Omega) \cap C(\bar{\Omega})$ denote the unique viscosity solution of~\eqref{eqn:pinns_inf_bvp} and
let $\{u_n\}_{n=1}^{\infty}$ be a sequence of PINN approximations. Assume the activation functions are of class $C^2$. Additionally, assume that
\[
\|-\Delta_\infty u_n\|_{L^\infty(\Omega)} \to 0
\quad \text{as } n \to \infty,
\]
and that
\[
\|u_n - g\|_{L^\infty(\partial\Omega)} \to 0.
\]
Further, assume that the sequence $\{u_n\}$ is uniformly Lipschitz in $\Omega$.
Then, the sequence $\{u_n\}$ converges uniformly to $u^*$ in $C(\bar{\Omega})$, that is,
\[
\|u_n - u^*\|_{C(\bar{\Omega})} \to 0.
\]
\end{theorem}

\begin{proof} 
We proceed in several steps. First, we establish uniform boundedness and equicontinuity of the sequence $\{u_n\}$ and invoke Arzelà-Ascoli to extract a uniformly convergent subsequence. Next, we show that the limit function satisfies the boundary condition and is a viscosity solution of the infinity Laplace equation. Finally, we use uniqueness to conclude convergence of the full sequence.

We first establish a uniform bound on the sequence $\{u_n\}$. By assumption, there exists a constant $C>0$ such that each $u_n$ is Lipschitz continuous on $\Omega$ with Lipschitz constant at most $C$. Therefore, for any $\bx_0 \in \Omega$ and any $\bx \in \partial\Omega$, we have
\begin{equation}
\begin{split}
    |u_n(\bx_0)|
&\le |u_n(\bx_0) - u_n(\bx)| + |u_n(\bx)| \\
&\le C |\bx_0 - \bx| + \|u_n\|_{L^\infty(\partial\Omega)} \\
&\le C \operatorname{diam}(\Omega) + \|u_n\|_{L^\infty(\partial\Omega)}.
\end{split}
\label{eqn:un_ineq}
\end{equation}

By the convergence of the boundary loss,
$
\|u_n - g\|_{L^\infty(\partial\Omega)} \to 0,
$
there exists $N$ such that for all $n \ge N$,
$
\|u_n - g\|_{L^\infty(\partial\Omega)} \le 1.
$
Therefore, for all $n \ge N$,
\begin{align*}
\|u_n\|_{L^\infty(\partial\Omega)}
&\le \|u_n - g\|_{L^\infty(\partial\Omega)} + \|g\|_{L^\infty(\partial\Omega)} \\
&\le 1 + \|g\|_{C(\partial\Omega)}.
\end{align*}
Substituting this back into~\eqref{eqn:un_ineq}, we obtain
\begin{align*}
\|u_n\|_{L^\infty(\Omega)} \leq C \operatorname{diam}(\Omega) + 1+\|g\|_{C(\partial\Omega)}.
\end{align*}
This shows that $\{u_n\}$ is uniformly bounded in $L^\infty(\Omega)$.

Furthermore, the uniform Lipschitz bound implies that $\{u_n\}$ is equicontinuous on $\bar{\Omega}$.
Combined with the uniform $L^\infty(\Omega)$ bound established above, the sequence $\{u_n\}$ is uniformly bounded and equicontinuous on the compact set $\bar{\Omega}$. Therefore, by the Arzelà-Ascoli theorem, there exist a subsequence $\{u_{n_k}\}$ and a function $\tilde{u} \in C(\bar{\Omega})$ such that
\[
u_{n_k} \to \tilde{u}
\quad \text{uniformly on } \bar{\Omega}.
\]

We now show that $\tilde{u}$ is Lipschitz continuous. For any $\bx, \by \in \bar{\Omega}$, each $u_{n_k}$ satisfies
\[
|u_{n_k}(\bx) - u_{n_k}(\by)| \le C\,|\bx - \by|.
\]
Passing to the limit as $k \to \infty$ and using the uniform convergence of $u_{n_k}$ to $\tilde u$ on $\bar\Omega$, we obtain
\[
|\tilde{u}(\bx) - \tilde{u}(\by)|
= \lim_{k \to \infty} |u_{n_k}(\bx) - u_{n_k}(\by)|
\le C\,|\bx - \by|.
\]
Therefore $\tilde{u}$ is Lipschitz continuous on $\bar{\Omega}$. In particular, $\tilde{u} \in W^{1,\infty}(\Omega)$ and
\[
\|\nabla \tilde{u}\|_{L^\infty(\Omega)} \le C.
\]

Next, we verify that the limit function satisfies the boundary condition. For any $\bx \in \partial\Omega$, the triangle inequality gives
\begin{align*}
|\tilde{u}(\bx) - g(\bx)|
&\leq |\tilde{u}(\bx) - u_{n_k}(\bx)| + |u_{n_k}(\bx) - g(\bx)| \\
&\leq \|u_{n_k} - \tilde{u}\|_{C(\bar{\Omega})}
+ \|u_{n_k} - g\|_{L^\infty(\partial\Omega)}.
\end{align*}
As $k \to \infty$, both terms on the right-hand side tend to $0$, which implies
\[
\tilde{u}(\bx) = g(\bx)
\quad \text{for all } \bx \in \partial\Omega.
\]
Therefore, $\tilde{u} \in W^{1,\infty}(\Omega) \cap C(\bar{\Omega})$ and satisfies the boundary condition $\tilde{u} = g$ on $\partial\Omega$.

We will now identify $\tilde{u}$ as a viscosity solution of $-\Delta_\infty \tilde{u} = 0$. We first show that $\tilde{u}$ is a viscosity subsolution. Let $\phi \in C^2(\Omega)$ be a test function such that $\tilde{u} - \phi$ attains a strict local maximum at a point $\bx_0 \in \Omega$. 

Since $u_{n_k} \to \tilde{u}$ uniformly on $\bar{\Omega}$, there exist points $\bx_k \to \bx_0$ such that $u_{n_k} - \phi$ attains a local maximum at $\bx_k$ for each $k$. We know $u_{n_k} - \phi$ attains a local maximum at $\bx_k$, and since the activation functions are of class $C^2$, each $u_{n_k}$ is $C^2$ and the classical optimality conditions apply. Therefore, we have $\nabla u_{n_k}(\bx_k) = \nabla \phi(\bx_k)$ and $D^2 u_{n_k}(\bx_k) \le D^2 \phi(\bx_k)$,
meaning that
$\bv^T D^2 u_{n_k}(\bx_k) \bv \le \bv^T D^2 \phi(\bx_k) \bv$
for all $\bv \in \mathbb{R}^d$.

Using the definition of the infinity Laplacian and the optimality conditions,
\begin{align}
-\Delta_\infty \phi(\bx_k)
&= -\nabla \phi(\bx_k)^T\, D^2 \phi(\bx_k)\, \nabla \phi(\bx_k) \nonumber\\
&\le -\nabla \phi(\bx_k)^T\, D^2 u_{n_k}(\bx_k)\, \nabla \phi(\bx_k) \nonumber\\
&= -\nabla u_{n_k}(\bx_k)^T\, D^2 u_{n_k}(\bx_k)\, \nabla u_{n_k}(\bx_k) \nonumber\\
&= -\Delta_\infty u_{n_k}(\bx_k).
\label{eq:inftydelta}
\end{align}
Since the PDE residual satisfies
\[
\|-\Delta_\infty u_{n_k}\|_{L^\infty(\Omega)} \to 0,
\]
we have, in particular,
\[
|-\Delta_\infty u_{n_k}(\bx_k)|
\le \|-\Delta_\infty u_{n_k}\|_{L^\infty(\Omega)}
\to 0
\quad \text{as } k \to \infty.
\]
Since $\phi \in C^2(\Omega)$ and $\bx_k \to \bx_0$, it follows that
\[
-\Delta_\infty \phi(\bx_k) \to -\Delta_\infty \phi(\bx_0)
\quad \text{as } k \to \infty.
\]
Passing to the limit in~\eqref{eq:inftydelta}, we conclude that
\[
-\Delta_\infty \phi(\bx_0) \le 0,
\]
showing that $\tilde{u}$ is a viscosity subsolution.

Similarly, let $\psi \in C^2(\Omega)$ be a test function such that $\tilde{u} - \psi$ attains a strict local minimum at a point $\by_0 \in \Omega$. Since $u_{n_k} \to \tilde{u}$ uniformly on $\bar{\Omega}$, there exist points $\by_k \to \by_0$ such that $u_{n_k} - \psi$ attains a local minimum at $\by_k$ for each $k$.

At the point $\by_k$, the function $\psi$ touches $u_{n_k}$ from below, so it follows that
\begin{align}
-\Delta_\infty \psi(\by_k)
&= -\nabla \psi(\by_k)^T D^2 \psi(\by_k)\, \nabla \psi(\by_k) \nonumber\\
&\ge -\Delta_\infty u_{n_k}(\by_k).
\label{eq:infty_supersol}
\end{align}
Since $\|-\Delta_\infty u_{n_k}\|_{L^\infty(\Omega)} \to 0$, we have
\[
-\Delta_\infty u_{n_k}(\by_k)
\ge -\|-\Delta_\infty u_{n_k}\|_{L^\infty(\Omega)} \to 0
\quad \text{as } k \to \infty.
\]
Moreover, since $\psi \in C^2(\Omega)$ and $\by_k \to \by_0$, we have
\[
-\Delta_\infty \psi(\by_k) \to -\Delta_\infty \psi(\by_0)
\quad \text{as } k \to \infty.
\]
Passing to the limit yields
\[
-\Delta_\infty \psi(\by_0) \ge 0,
\]
showing that $\tilde{u}$ is a viscosity supersolution. Therefore, $\tilde{u}$ is a viscosity solution. By Jensen's theorem~\cite{jensen1993uniqueness}, the Dirichlet problem for $-\Delta_\infty u = 0$ admits a unique viscosity solution. Hence, we may conclude $\tilde{u} = u^*$.

Finally, we show convergence of the entire sequence. We have shown that every subsequence of $\{u_n\}$ admits a further subsequence that converges uniformly on $\bar{\Omega}$ to the same limit $u^*$. It follows that the whole sequence $\{u_n\}$ converges uniformly to $u^*$ in $C(\bar{\Omega})$, that is,
\[
\|u_n - u^*\|_{C(\bar{\Omega})} \to 0 \quad \text{as } n \to \infty.
\]
This completes the proof.
\end{proof}

These results provide analytical support for the PINN framework at the level of continuum residuals and boundary errors. 
In each case, convergence is conditional on assumptions that reflect the intended outcome of training, 
but which are not directly enforced by the finite-sample, 
clipped, and outlier-filtered loss used in the implementation.

\section{DeepONet}
\label{section:DeepONet}

In addition to PINNs, we investigate Deep Operator Networks (DeepONet)~\cite{lu2021learning} as a surrogate modeling approach for families of $p$-Laplace and
infinity Laplace type problems.
Unlike PINNs, which are trained to approximate the solution of a single PDE instance, DeepONet
aims to learn a parametric mapping from problem inputs to solution fields.
Once trained, a single DeepONet model can be asked to produce approximate solutions for a range
of parameter values without retraining.

In the present work, DeepONet is used to learn the dependence of the solution on the $p$ value
and, when applicable, on geometric parameters describing the computational domain.
This allows us to efficiently explore the $p$ parameter landscape in the $p$-Laplace solutions, including the
large $p$ regime associated with distance function approximation. We emphasize that DeepONet is not used here as a direct PDE solver, but rather as a surrogate model trained on high-fidelity numerical data.
The ability of DeepONet to generalize across parameter values is therefore governed by the
coverage and quality of the training dataset, as well as properties of the operator itself.

Universal approximation results for nonlinear operators provide theoretical justification
for the expressive power of DeepONet~\cite{Chen1995OperatorUniversal, lu2021learning}. In Section~\ref{subsec:deeponet-universal-plap}, we specialize these results to the
parametric $p$-Poisson problem, showing that, under a continuity assumption on the solution
operator, DeepONet can approximate the associated mapping uniformly on compact parameter sets.
However, quantitative error estimates and rigorous convergence guarantees for degenerate elliptic
operators such as the $p$-Laplacian and the infinity Laplacian are generally unavailable.
Accordingly, the performance of the DeepONet models is evaluated empirically in
Section~\ref{section:expts} by comparison with reference solutions and by examining their behavior
for large $p$ values.

\subsection{DeepONet architecture}

DeepONet consists of a trunk network and a branch network.
The trunk network, denoted by $f_t(\bx;\theta_t)$, encodes the spatial dependence of the solution,
while the branch network, denoted by $f_b(\boldsymbol{\mu};\theta_b)$, encodes problem parameters.
Here, $\theta_t$ and $\theta_b$ denote the trainable parameters of the trunk and branch networks,
respectively, and $\boldsymbol{\mu}$ represents the input parameter vector.
In our setting, the trunk network takes spatial coordinates $\bx \in \Omega$ as input and outputs a $K_l$-dimensional vector
\[
{\bf t} = (t_1,\dots,t_{K_l}).
\]
The branch network takes as input $p$ and, when relevant, additional geometric
parameters $\boldsymbol{a}_\Omega$ describing the domain (e.g., ellipse semi-axis lengths), and outputs a $K_l$-dimensional vector
\[
{\bf b} = (b_1,\dots,b_{K_l}).
\] 
The network output is given by
\begin{equation*}
u_\theta(\bx;\boldsymbol{\mu}) = \sum_{k=1}^{K_l} w_k \, b_k t_k,
\end{equation*}
where $K_l$ is the latent dimension and $w_k$ are trainable weights.

The architecture is illustrated in Figure~\ref{fig:Deeponet}.
In our experiments, the trunk network consists of three hidden fully connected layers with $512$ neurons each, followed by an output layer of dimension $K_l=128$.
The branch network consists of three hidden fully connected layers with $128$ neurons each, also followed by an output layer of dimension $K_l=128$. Both networks use $\tanh$ activation functions.

\begin{figure*}[h]
\centering
\includegraphics[width=0.6\textwidth]{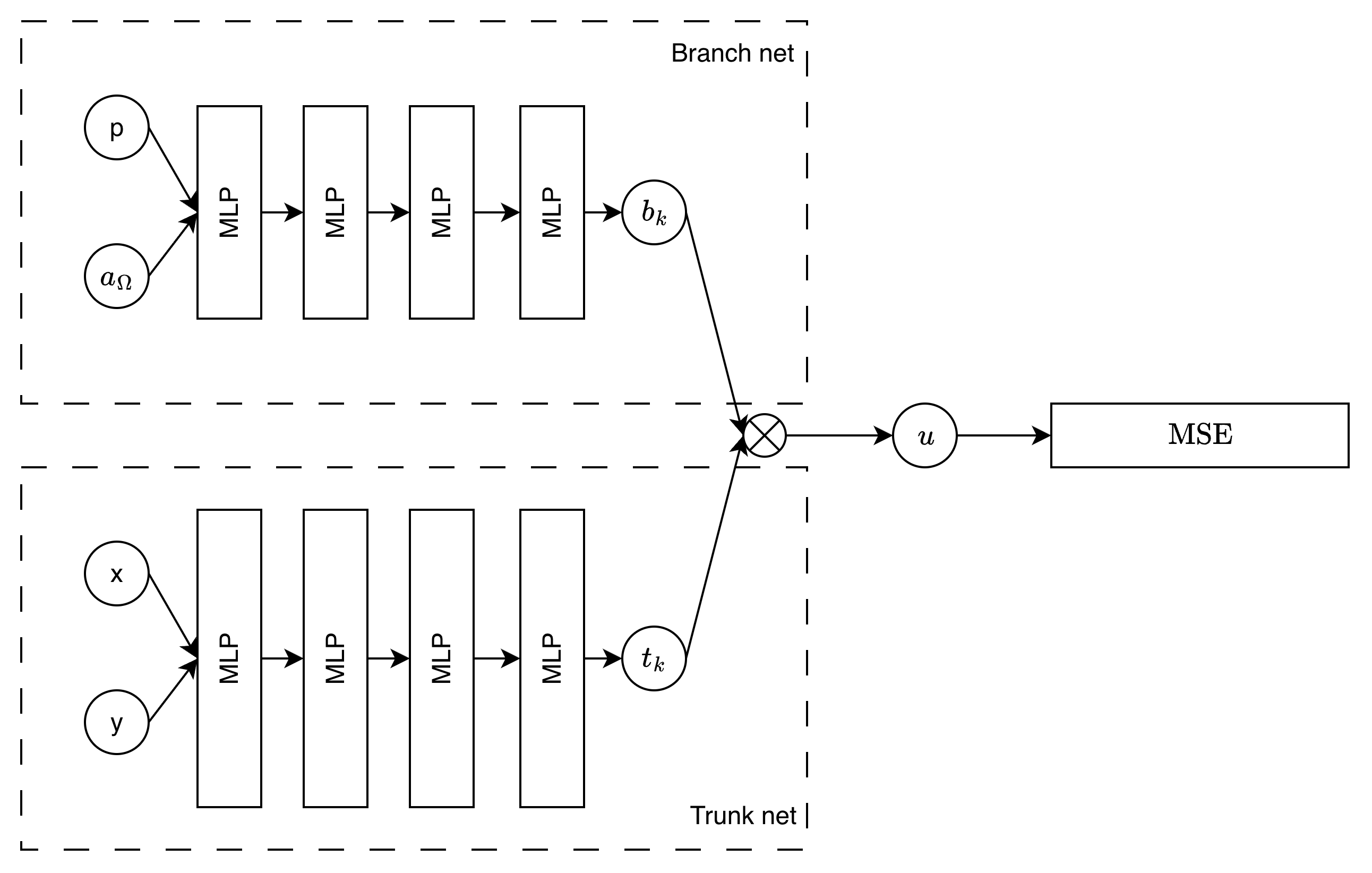}
\caption{DeepONet architecture used in this work.}
\label{fig:Deeponet}
\end{figure*}

\subsection{Training data and loss function}

DeepONet is trained using solution data generated by the numerical method described in
Section~\ref{section:data}.
Specifically, for each selected value of $p$, we generate reference solutions using the
Newton-based finite element solver.
For problems involving the infinity Laplacian limit, additional training data is obtained from
PINNs approximations, which offer a practical way of generating solutions where
traditional solvers become prohibitive.

Let $N_t$ denote the number of training samples associated with a fixed parameter instance $\boldsymbol{\mu}$,
and let $\{u^{\mathrm{ref}}_{\boldsymbol{\mu},i}\}_{i=1}^{N_t}$ denote the corresponding reference solution
values evaluated at spatial points $\{\bx_i\}$.
For a fixed parameter instance $\boldsymbol{\mu}$, the DeepONet is trained by minimizing the mean squared error
\begin{equation}
\label{eq:deeponet_loss}
\mathrm{MSE}_{\mathrm{DeepONet}}(\boldsymbol{\mu})
=
\frac{1}{N_t}
\sum_{i=1}^{N_t}
\left(
u_\theta(\bx_i; \boldsymbol{\mu}) - u^{\mathrm{ref}}_{\boldsymbol{\mu},i}
\right)^2.
\end{equation}
In practice, training data is assembled across multiple parameter instances $\boldsymbol{\mu}$, enabling the model to learn the parametric dependence of the solution on both $p$ and the domain geometry.
Once trained, the DeepONet can be called for previously unseen values of $p$. 
The loss~\eqref{eq:deeponet_loss} is minimized using the Adam optimizer~\cite{KingmaBa2015} with an initial learning rate of $10^{-4}$ and cosine-annealed to $10^{-6}$ over $20$ epochs.

\subsection{Universal approximation for $p$-Laplace Problems}
\label{subsec:deeponet-universal-plap}

\begin{prop}[Universal approximation of DeepONet for parametric $p$-Poisson problems]
\label{prop:deeponet_parametric_p_poisson}
Let $\sigma:\mathbb{R}\to\mathbb{R}$ be a continuous non-polynomial activation function, and let
\[
Z:=K\times[2,P]\subset\mathbb{R}^{m+1},
\]
where $K\subset\mathbb{R}^m$ is compact.
For each $\bz=(\boldsymbol{\mu},p)\in Z$, consider the boundary value problem
\begin{alignat*}{2}
-\Delta_p u_{\boldsymbol{\mu},p} &= f_{\boldsymbol{\mu}}, \quad &&\text{in } \Omega_{\boldsymbol{\mu}}, \\
u_{\boldsymbol{\mu},p} &= g_{\boldsymbol{\mu}}, \quad &&\text{on } \partial\Omega_{\boldsymbol{\mu}},
\end{alignat*}
where $\Omega_{\boldsymbol{\mu}}\subset\mathbb{R}^d$ is a bounded domain with $\partial\Omega_{\boldsymbol{\mu}}\in C^{1,\beta}$ for some $\beta>0$ and $2\le p\le P<\infty$.

Assume that for each $(\boldsymbol{\mu},p)\in Z$, the problem admits a unique weak solution
$
u_{\boldsymbol{\mu},p}\in W^{1,p}(\Omega_{\boldsymbol{\mu}})\cap C(\overline{\Omega_{\boldsymbol{\mu}}}).
$
Assume further that there exists a fixed bounded reference domain $\widehat{\Omega}\subset\mathbb{R}^d$ such that, for each $\boldsymbol{\mu}\in K$, there is a bijective $C^1$ map
\[
T_{\boldsymbol{\mu}}:\overline{\widehat{\Omega}}\to\overline{\Omega_{\boldsymbol{\mu}}}.
\]
Define the pullback
\[
\hat u_{\boldsymbol{\mu},p}(\bx):=u_{\boldsymbol{\mu},p}(T_{\boldsymbol{\mu}}(\bx)),
\qquad \bx\in\overline{\widehat{\Omega}},
\]
and the solution operator
\[
\mathcal{G}:Z\to C(\overline{\widehat{\Omega}}),
\qquad
\mathcal{G}(\bz)=\hat u_{\boldsymbol{\mu},p}.
\]
Assume that $\mathcal{G}$ is continuous.
Then, for every $\epsilon>0$, there exists a DeepONet
\[
\mathcal{G}_\theta: Z\times\overline{\widehat{\Omega}}\to\mathbb{R},
\]
with branch input $\bz=(\boldsymbol{\mu},p)$ and trunk input $\bx\in\overline{\widehat{\Omega}}$, such that
\[
\sup_{\bz\in Z}
\|\mathcal{G}(\bz)-\mathcal{G}_\theta(\bz,\cdot)\|_{C(\overline{\widehat{\Omega}})}<\epsilon.
\]
\end{prop}

\begin{proof}
Since $K$ is compact and $[2,P]$ is compact, the parameter space
\[
Z = K \times [2,P] \subset \mathbb{R}^{m+1}
\]
is compact. By assumption, for each $\bz=(\boldsymbol{\mu},p)\in Z$, the pullback $\hat u_{\boldsymbol{\mu},p}$ is a continuous function on $\overline{\widehat{\Omega}}$, so the operator
\[
\mathcal{G}: Z \to C(\overline{\widehat{\Omega}}), 
\qquad
\mathcal{G}(\bz) = \hat u_{\boldsymbol{\mu},p},
\]
is well-defined and continuous. Then $Z$ is a compact subset of the Banach space $\mathbb{R}^{m+1}$. We set
\[
K_1 := Z, \qquad K_2 := \overline{\widehat{\Omega}}.
\]
Then $\mathcal{G}$ is a continuous operator from the compact set $K_1$ into $C(K_2)$. Hence, by the universal approximation theorem for nonlinear operators~\cite{Chen1995OperatorUniversal}, and its branch-trunk realization in DeepONet form as stated in~\cite{lu2021learning}, for any $\epsilon > 0$ there exists a DeepONet $\mathcal{G}_\theta(\bz,\bx)$ such that
\[
\sup_{\bz \in Z} \sup_{\bx \in \overline{\widehat{\Omega}}}
\bigl| \mathcal{G}(\bz)(\bx) - \mathcal{G}_\theta(\bz,\bx) \bigr| < \epsilon.
\]
Taking the supremum over $\bx \in \overline{\widehat{\Omega}}$ yields
\[
\sup_{\bz \in Z}
\| \mathcal{G}(\bz) - \mathcal{G}_\theta(\bz,\cdot) \|_{C(\overline{\widehat{\Omega}})} < \epsilon,
\]
which establishes the desired result.
\end{proof}

\section{Numerical experiments}
\label{section:expts}

In this section, we present numerical experiments for the $p$-Laplacian and related
infinity Laplacian problems using PINNs and DeepONets. The primary objective of these
experiments is to assess the empirical performance of these neural network approaches
in approximating solutions of nonlinear and degenerate elliptic problems, particularly
in regimes where traditional numerical solvers become challenging. In particular, we
investigate robustness under reduced regularity, which is characteristic of both
$p$-Laplace and infinity Laplace problems, where solutions often exhibit
$C^{1,\alpha}$ regularity for finite $p$ with $\alpha \to 0$ as $p \to \infty$, and
degenerate to only Lipschitz continuous solutions in the limit.

To this end, we deliberately probe the large-$p$ regime. We use continuation in $p$
and partial training strategies, allowing us to examine how well these methods capture limiting behavior.
The results presented are representative of the observed behavior across our experiments, 
although quantitative performance may vary depending on initialization, sampling, and hyperparameter choices.

Except for the direct finite-$p$ unit-disc validations reported in Section~\ref{subsec:pinn-finite-p-disc} and Table~\ref{tbl:deeponet-finite-p-disc}, the solution error is measured against the known $p\to\infty$ limiting solution $u_\infty$, rather than against the corresponding finite-$p$ solution $u_p$. These errors therefore measure recovery of the limiting profile, not the accuracy of the neural network as a solver or surrogate for the finite-$p$ problem. For the unit-disc problem, where the exact finite-$p$ solution is available, we additionally report the direct error $\mathrm{MSE}_p$ for both the iterative PINN and DeepONet approximations. The FEM errors reported alongside them are convergence-to-limit baselines computed using the same reference solution $u_\infty$. For the remaining DeepONet experiments, finite-$p$ FEM solutions are still used as training data, but the reported testing error is evaluated against $u_\infty$. Training losses and PDE residuals are reported separately. In all tables, ``$-$" denotes failure to converge or compute, typically due to numerical instability or loss blow-up for large values of $p$.

Table~\ref{tbl:meshes} summarizes the mesh statistics used to generate the FEM reference solutions. Mesh construction and refinement are carried out using the \texttt{deal.II}
library’s built-in \texttt{GridGenerator} namespace functions~\cite{dealII96}.

\begin{table}[!h] 
\begin{center}
\scalebox{0.9}{
\begin{tabular}{c|c|ccc|ccc}
\hline
\hline
 \multirow{2}{*}{Example} &  \multirow{2}{*}{Equation of domain}  & \multicolumn{3}{c|}{Cell edge length} &  \multicolumn{3}{c}{Mesh size} \\
 & &\text{mean}&\text{max} &\text{min}&\text{\# cells}&\text{\# vertices}&\text{\# faces}\\
\hline
2D disc & $x^2+y^2\leq 1$ & $2.5132$e-02 & $4.9083$e-02 & $1.8043$e-02 & $6825$ & $5185$ & $13776$ \\
\hline
Ellipse 1 & $x^2+16y^2 \leq 1$ & $1.1645$e-02 & $1.5783$e-02 & $6.4314$e-03 & $7644$ & $6025$ & $15638$ \\
Ellipse 2 & $8.5x^2+8.5y^2-15xy \leq 1$ & $1.1645$e-02 & $1.5783$e-02 & $6.4314$e-03 & $7644$ & $6025$ & $15638$\\
Ellipse 3 & $4x^2+y^2 \leq 1$ & $6.6203$e-03 & $1.8764$e-02 & $1.4471$e-03 & $46704$ & $35785$ &  $93758$\\
\hline
3D ball & $x^2+y^2+z^2 \leq 1$ & $2.0717$e-02 & $4.9105$e-02 & $9.1529$e-03 & $224694$ & $201857$ & $680907$\\
3D cylinder & $y^2+z^2 \leq 1, -1 \leq x \leq 1$ & $3.7588$e-02 & $6.2500$e-02 & $1.8043$e-02 & $374490$ & $337025$ & $1135755$\\
3D torus & \begin{tabular}{@{}c@{}}$(2-\sqrt{x^2+z^2})^2+y^2\leq 1,$\\
$-3 \leq x,z \leq 3, -1 \leq y \leq 1$\end{tabular} & $7.7294$e-02 & $1.9631$e-01 & $3.6612$e-02 & $140430$ & $126048$ & $425382$ \\
\hline
\hline
\end{tabular}
}
\end{center}
\caption{ 
Summary of mesh statistics for training data used in the FEM scheme. All meshes are composed of quadrilateral or hexahedral elements and were generated with \texttt{deal.II}~\cite{dealII96}. We report the average, maximum, and minimum cell edge lengths, as well as the total number of cells, vertices, and faces for each domain geometry.
}
\label{tbl:meshes}
\end{table}

\subsection{Testing error}
To assess approximation accuracy, we evaluate all learned models on a fixed set of test points. 
For the 2D PINN experiments, these consist of a uniform $100 \times 100$ grid over the computational domain, restricted to  points in $\Omega$ together with boundary points on $\partial \Omega$. 
For the 3D Eikonal PINN experiments used to generate limiting DeepONet data, the test set is constructed from a uniform $50\times50\times50$ grid and restricted to the computational domain.

Let $N_{\mathrm{test}}$ denote the total number of testing points.
For a reference solution $u_{\mathrm{ref}}$, we define its normalized
discrete $L^2$ norm on the test set by
\begin{equation}
\label{eq:normalized_discrete_l2}
\|u_{\mathrm{ref}}\|_{2,N_{\mathrm{test}}}
=
\left(
\frac{1}{N_{\mathrm{test}}}
\sum_{k=1}^{N_{\mathrm{test}}}
|u_{\mathrm{ref}}(\bx_k)|^2
\right)^{1/2}.
\end{equation}

For a learned approximation $u_\theta(\bx)$, the test error with respect to the exact infinity Laplace 
solution $u_\infty$ is defined by
\begin{equation}
\label{eq:MSE_test}
\mathrm{MSE}_\infty
=
\frac{1}{N_{\mathrm{test}}}
\sum_{k=1}^{N_{\mathrm{test}}}
\bigl(u_\theta(\bx_k) - u_\infty(\bx_k)\bigr)^2.
\end{equation}
For finite $p$ examples, we measure the discrepancy with respect to the limiting infinity solution via
\begin{equation}
\label{eq:MSE_test_pinf}
\mathrm{MSE}_{\infty, p}
= 
\frac{1}{N_{\mathrm{test}}}
\sum_{k=1}^{N_{\mathrm{test}}}
\bigl(u_\theta(\bx_k; p) - u_\infty(\bx_k)\bigr)^2,
\end{equation}
which quantifies the discrepancy between the network prediction evaluated at parameter $p$ and the limiting solution $u_\infty$, not the finite-$p$ solution error. When an exact finite-$p$ solution is available, we also report the direct finite-$p$ solution error
\begin{equation} 
\label{eq:MSE_test_finite_p}
\mathrm{MSE}_{p} = \frac{1}{N_{\mathrm{test}}} \sum_{k=1}^{N_{\mathrm{test}}} \bigl(u_\theta(\bx_k;p)-u_p(\bx_k)\bigr)^2.
\end{equation}
Unlike $\mathrm{MSE}_{\infty,p}$, this quantity directly evaluates the learned
approximation against the solution of the corresponding finite-$p$ problem. Here, $u_\theta(\bx; p)$ denotes a learned approximation at parameter value $p$. 
In the PINN setting, a separate network is trained for each $p$, so the dependence 
on $p$ is implicit, whereas in the DeepONet setting the network explicitly takes $p$ 
as an input parameter. 
We additionally consider the residual error associated with the infinity Laplacian,
\begin{equation}
\label{eq:MSE_test_inflapres}
\mathrm{MSE}_{\Delta_\infty}
=
\frac{1}{N_{\mathrm{test}}}
\sum_{k=1}^{N_{\mathrm{test}}}
\bigl(\Delta_\infty u_\theta(\bx_k; p)\bigr)^2,
\end{equation}
which measures how well the learned solution satisfies the homogeneous infinity Laplace equation.

Table~\ref{tab:pinns_mse} summarizes the training and testing mean squared error 
quantities used throughout the PINN experiments. Similarly, 
Table~\ref{tab:deeponet_mse} summarizes the corresponding quantities for the DeepONet experiments.

\begin{table}[h!]
\centering
\begin{tabular}{|lll|}
\hline
\underline{\textbf{Training:}} & &\\

$\mathrm{MSE}_{\partial\Omega}$ 
&~\eqref{eq:MSE_Dirichlet}
& Enforces Dirichlet boundary data \\

$\mathrm{MSE}_{\partial\Omega,\mathrm{mixed}}$&~\eqref{eq:MSE_mixed} & Enforces mixed Dirichlet-Neumann boundary data\\

$\mathrm{MSE}_{\mathrm{PINNs}}$ 
&~\eqref{eq:MSE_PINNS}
& Enforces governing equation \\

$\mathrm{MSE}$ 
&~\eqref{eq:MSE_total}
& Total training objective \\

\hline 
\underline{\textbf{Testing:}} & & \\

$\mathrm{MSE}_{\infty}$ 
&~\eqref{eq:MSE_test}
& Infinity Laplacian: error vs infinity solution \\

$\mathrm{MSE}_{\infty, p}$ 
&~\eqref{eq:MSE_test_pinf} 
&  $p$-Laplacian: error vs limiting infinity solution \\

$\mathrm{MSE}_{p}$
&~\eqref{eq:MSE_test_finite_p}
& $p$-Laplacian: error vs exact finite-$p$ solution \\

$\mathrm{MSE}_{\Delta_\infty}$ 
&~\eqref{eq:MSE_test_inflapres} 
& Residual for infinity Laplacian \\
\hline
\end{tabular}
\caption{Mean squared error quantities used for PINN training and evaluation.}
\label{tab:pinns_mse}
\end{table}

\begin{table}[h!]
\centering
\begin{tabular}{|lll|}
\hline
\underline{\textbf{Training:}} & & \\

$\mathrm{MSE}_{\mathrm{DeepONet}}$ 
&~\eqref{eq:deeponet_loss}
& Data-driven training objective \\

\hline 
\underline{\textbf{Testing:}} & & \\

$\mathrm{MSE}_{\infty, p}$ 
&~\eqref{eq:MSE_test_pinf}
& $p$-Laplacian: error vs limiting infinity solution \\
\hline
\end{tabular}
\caption{Mean squared error quantities used for DeepONet training and evaluation.}
\label{tab:deeponet_mse}
\end{table}

\subsection{2D infinity Laplace equation with Dirichlet boundary conditions}
\label{subsec:2Dinflap-dbc}

The first set of experiments considers two-dimensional infinity Laplace boundary value problems with Dirichlet boundary conditions. 
For each problem, a separate PINN model is trained to assess its ability to approximate solutions of the infinity Laplacian. 
The training process employs interior points for the PDE residual loss and boundary points for the data-driven loss. 

For square domains, $10{,}000$ evenly spaced points are placed along each edge, giving $40{,}000$ boundary points in total. 
For the unit disc, $10{,}000$ points are placed uniformly in angle on the circle. Interior points (${\sim}\,10^6$) are resampled uniformly at random within the domain at each epoch. 
The models are trained for $100$ epochs. The training boundary loss (blue) corresponds to the mean squared error evaluated on boundary points, 
the training PINNs loss (orange) corresponds to the mean squared PDE residual evaluated at interior points, 
and the testing loss (green) is the mean squared error computed on a $100\times100$ uniform grid ($10{,}000$ points for the square domain, $7{,}668$ points after disc masking).

For the infinity Laplacian $\Delta_\infty u = (\nabla u)^T D^2 u\, \nabla u$, the PDE residual of a neural network approximation $u_\theta$ satisfies
\begin{equation}
\label{eq:residual_bound}
|r(\bx)| = |(\nabla u_\theta)^T D^2 u_\theta\, \nabla u_\theta|\leq \|\nabla u_\theta\|^2 \, \|D^2 u_\theta\|,
\end{equation}
in a pointwise sense. Convergence analysis requires the PDE residual to be bounded; however, the bound~\eqref{eq:residual_bound} scales quadratically in the gradient norm $\|\nabla u_\theta\|$, which can become large during training. To obtain a tighter bound, we employ the $\eta$-normalized formulation
\begin{equation*}
\frac{(\nabla u)^T D^2 u\, \nabla u}{\eta^2 + \|\nabla u\|^2} = 0,
\end{equation*}
with $\eta = 10^{-5}$.
Since the numerator vanishes for any solution of $\Delta_\infty u = 0$, this defines the same solution set. This normalized residual satisfies
\begin{equation}
\label{eq:normalized_bound}
\left|\frac{(\nabla u_\theta)^T D^2 u_\theta\, \nabla u_\theta}{\eta^2 + \|\nabla u_\theta\|^2}\right| \leq \|D^2 u_\theta\|,
\end{equation}
which removes the bound dependence on $\|\nabla u_\theta\|$ entirely, making it more stable during training.

To further enforce this bound during training, we apply outlier removal, excluding the top $2\%$ of extreme PDE residuals at each optimization step. 
This ensures that the empirical PDE loss $\mathrm{MSE}_{\mathrm{PINNS}}$ is not dominated by a small set of points 
where the network's approximation temporarily violates the regularity assumptions underlying \eqref{eq:normalized_bound}. 

In addition, the PDE residual is clamped to the interval $[-10,\,10]$ before computing the MSE loss, preventing isolated large residuals from destabilizing early training. 
Together, the $\eta$-normalization addresses the operator-level scaling of the residual, while residual clipping and outlier removal provide robustness at the optimization level.

The remaining training hyperparameters are chosen as follows and are shared across all three test cases. 
The PDE loss weight $\alpha$ is controlled by an adaptive plateau scheduler: $\alpha$ is initialized at $10^{-5}$ and multiplied by $10$ 
whenever the epoch-averaged boundary loss fails to improve by at least $20\%$ over $5$ consecutive epochs, up to a maximum of $10^{-1}$. 
This curriculum ensures that the network first fits the boundary data before gradually increasing the PDE residual's influence. 
A $20\%$ validation split with best-model checkpointing is used to select the model with the lowest boundary validation loss. 
See~\ref{app:hyperparams} for full training details and~\ref{app:ablation} for an ablation study of these choices.

\begin{table}[h]
\begin{center}
\scalebox{0.9}{
\begin{tabular}{c|cc|cc|cc}
\hline
\hline
\multirow{3}{*}{Example} &  \multirow{3}{*}{Exact solution}  & \multirow{3}{*}{$\|u_\infty\|_{2,N_{\mathrm{test}}}$} & \multicolumn{2}{c|}{MSE} & \multicolumn{2}{c}{Runtime}\\
 & &  & Our method&\multirow{2}{*}{Newton}& PINN inference &\multirow{2}{*}{Newton}\\
 &&&(mean $\pm$ s.d.)&&(+training)& \\
\hline
Arctan, Fig \ref{fig:training1.1} & $\text{arctan}(y/x)$  & $8.862$e-01 & $(3.118\pm2.327)$e-07 & $5.892$e-05 & $1.456$e-04s (+$1453$s) & $2776$s\\
Aronsson (square), Fig \ref{fig:training1.3} & $|x|^{\frac{4}{3}}-|y|^{\frac{4}{3}}$ & $4.277$e-01 & $(4.468\pm1.696)$e-07 & $9.511$e-05  & $7.719$e-05s (+$1064$s) & $6955$s\\
Aronsson (disc), Fig \ref{fig:training1.4} & $|x|^{\frac{4}{3}}-|y|^{\frac{4}{3}}$ & $4.284$e-01 & $(6.247\pm2.805)$e-07 & $2.522$e-04  & $7.608$e-05s (+$905$s) & $17447$s\\
\hline
\hline
\end{tabular}
}
\end{center}
\caption{Comparison of our PINN method and the Newton FEM method for solving
the two-dimensional infinity-Laplacian equation subject to Dirichlet boundary
conditions. The PINN MSE is reported as the mean $\pm$ sample standard
deviation over five runs with random seeds $1234$, $3456$, $5678$, $6789$,
and $8765$. The reported PINN timings and loss histories in Figures \ref{fig:training1.1}-\ref{fig:training1.4} correspond to the seed-$1234$
runs. The Newton FEM results are deterministic results computed on the fixed
meshes described in Table~\ref{tbl:meshes}.}
\label{tbl:mse1}
\end{table}
To examine seed sensitivity, we repeated each of the three PINN experiments using five fixed random seeds: $1234$, $3456$, $5678$, $6789$, and $8765$. The resulting means and sample standard deviations are reported in Table~\ref{tbl:mse1}. The mean test errors remain of order $10^{-7}$ for all three problems. The FEM errors reported are computed on the fixed meshes used in these experiments. The comparison represents accuracy at a representative reference resolution. The PINN computations were performed on an NVIDIA RTX 5880 Ada GPU, whereas the Newton FEM computations were performed on an Intel Xeon Gold 5420+ CPU, with deal.II configured to use up to 112 threads.
Table~\ref{tbl:mse1} summarizes the testing MSE values, training and inference times for the considered test cases.
Both the PINN and FEM errors are computed against the known exact solution. The FEM results provide a numerical baseline for the same test problems.
The PINN models provide fast inference once training is complete. 
Detailed results for each example are discussed in the subsections that follow.
Training error evolution is illustrated in the corresponding figures.

\subsubsection{Arctan example (homogeneous, Dirichlet)}
\label{subsec:arctan}

Consider the following boundary value problem:
\begin{align*}
-\Delta_{\infty} u &=  0, &\qquad \text{ in } \Omega,\\
u(x,y) &= \text{arctan}(y/x), &\qquad \text{ on } \partial\Omega.
\end{align*}
The computational domain is defined as $[0.01,1] \times [0.01,1]$. This example admits a smooth solution on the given computational domain, which excludes the coordinate singularity at the origin. It therefore provides a convenient test case for evaluating the accuracy of the PINNs approximations.

For the seed-$1234$ run, the achieved training boundary loss is $1.395\times10^{-7}$ and the training PINNs loss is $9.338\times10^{-6}$. 
Averaged over five random seeds, the test MSE is $(3.118\pm2.327)\times10^{-7}$, where the reported variation is the sample standard deviation.
Figure~\ref{fig:training1.1} shows that the PINN prediction closely matches the exact solution.

The total training time is approximately $1453$s for $100$ epochs. The inference time is $1.456\times10^{-4}$s for evaluation over the entire domain. In comparison, the Newton-based FEM solver requires approximately $2776$s to compute the reference solution. 
At the fixed FEM resolution used in this experiment, the mean PINN test error, $(3.118\pm2.327)\times10^{-7}$ over five random seeds, is more than two orders of magnitude below the FEM error of $5.892\times10^{-5}$.

\begin{figure*}[h]
\centering
\begin{subfigure}[b]{0.33\textwidth}
\centering
\includegraphics[width=\textwidth]{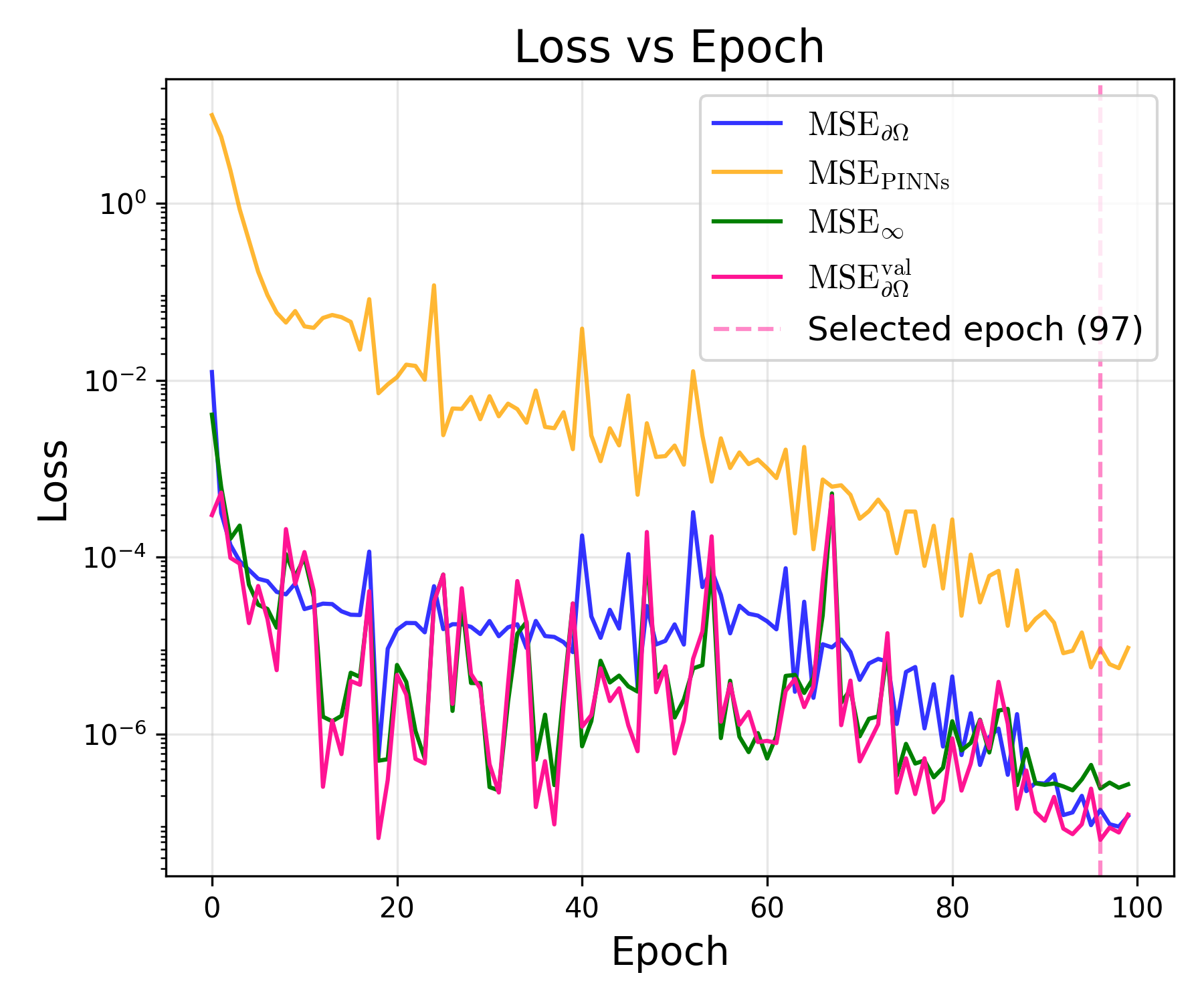}
\caption[]%
{{\small }}    
\end{subfigure}
\begin{subfigure}[b]{0.33\textwidth}
\centering
\includegraphics[width=\textwidth]{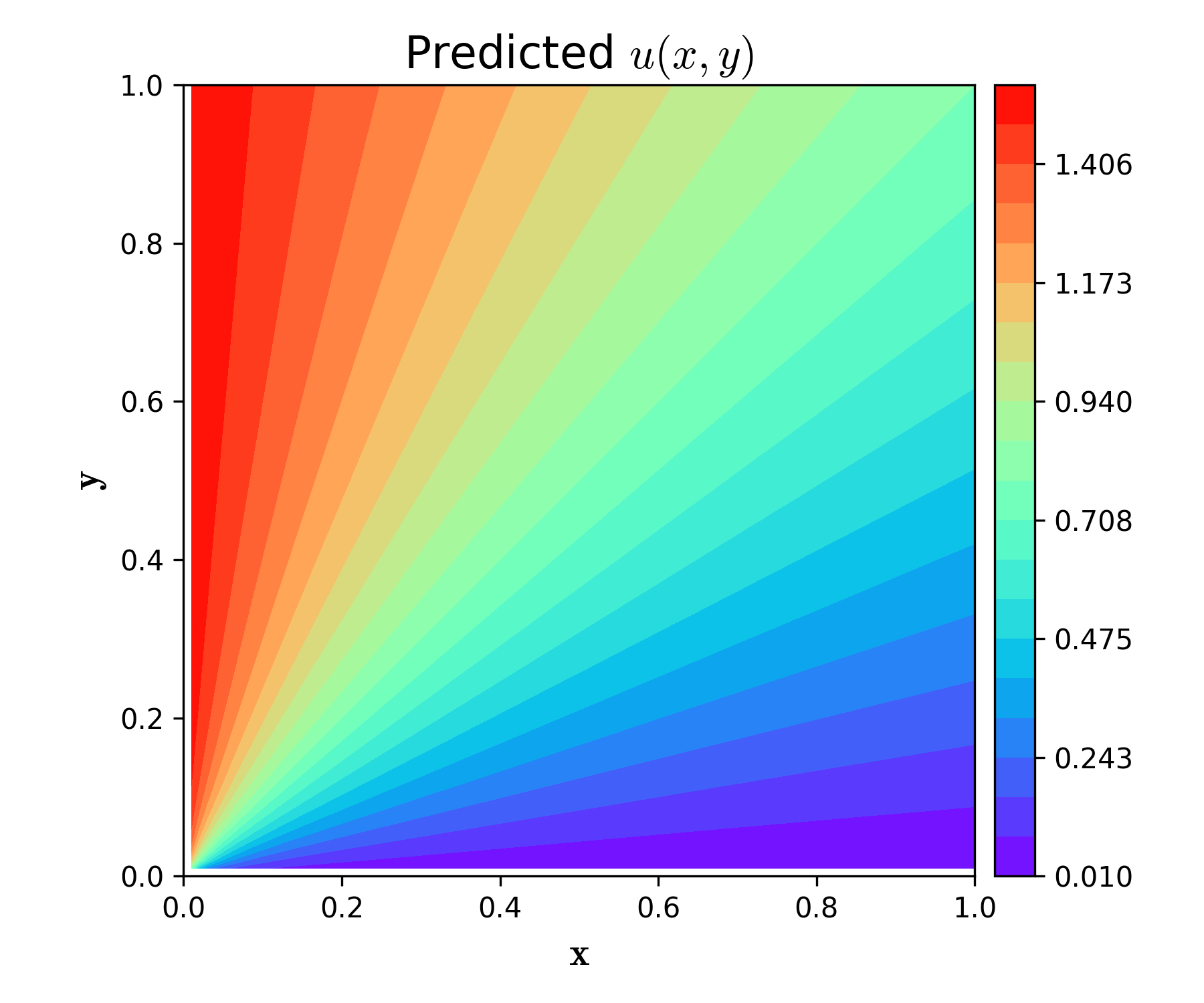}
\caption[]%
{{\small }}    
\end{subfigure}
\begin{subfigure}[b]{0.33\textwidth}  
\centering 
\includegraphics[width=\textwidth]{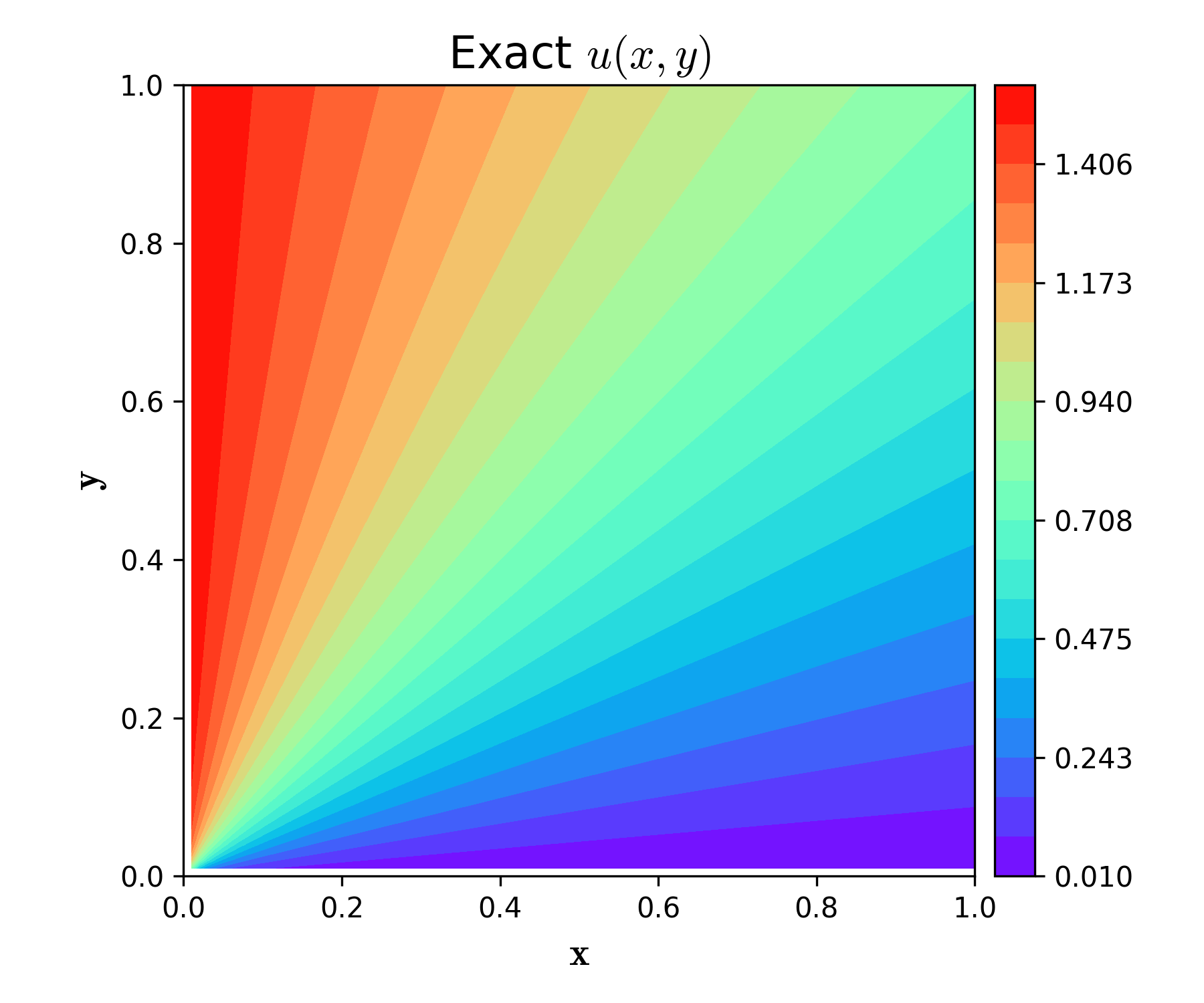}
\caption[]%
{{\small  }}    
\end{subfigure}
\caption[]%
{{\small Arctan example~\ref{subsec:arctan}. (a) Loss histories versus epoch: training boundary loss $\mathrm{MSE}_{\partial\Omega}$, training PINNs loss $\mathrm{MSE}_{\mathrm{PINNs}}$, testing loss $\mathrm{MSE}_{\infty}$, and validation boundary loss $\mathrm{MSE}_{\partial\Omega}^{\mathrm{val}}$, evaluated on the validation subset of boundary points. The dashed vertical line marks the checkpoint selected by the minimum validation boundary loss. (b) PINN prediction. (c) Exact solution.}  }
\label{fig:training1.1}
\end{figure*}

\subsubsection{Aronsson example (homogeneous, Dirichlet)}
\label{subsec:aronsson}

The Aronsson problem for the infinity Laplacian is given by
\begin{align*}
-\Delta_{\infty} u &=  0, &\qquad \text{ in } \Omega,\\
u(x,y) &= |x|^{\frac{4}{3}}-|y|^{\frac{4}{3}}, &\qquad \text{ on } \partial\Omega.
\end{align*}

This example is a canonical test case for the infinity Laplacian, as the solution is only
$C^{1,\frac13}$, with nonsmoothness occurring along coordinate axes. 

For this example, we evaluate the model on two different domains.
In the first square domain, $[-1,1]\times[-1,1]$, the seed-$1234$ run achieves a training boundary loss of $4.111\times10^{-7}$ and a training PINNs loss of $9.810\times10^{-5}$. 
Averaged over five random seeds, the test MSE is $(4.468\pm1.696)\times10^{-7}$, which remains more than two orders of magnitude below the FEM error of $9.511\times10^{-5}$ at the fixed resolution used here.

On the unit disc domain, $\{(x,y)\in\mathbb{R}^2 \mid x^2+y^2\le1\}$, the seed-$1234$ run achieves a training boundary loss of $4.150\times10^{-7}$ and a training PINNs loss of $8.966\times10^{-5}$. 
Averaged over five random seeds, the test MSE is $(6.247\pm2.805)\times10^{-7}$. This mean error remains more than two orders of magnitude below the FEM error of $2.522\times10^{-4}$ at the fixed resolution listed in Table~\ref{tbl:meshes}.
Figures~\ref{fig:training1.3} and~\ref{fig:training1.4} illustrate that the PINNs predictions capture the qualitative structure of the exact solution on both domains.

The training time is approximately $1064$s and $905$s for $100$ epochs on the square and unit disc
domains, respectively.
The inference time for evaluation over the entire domain is $7.719\times10^{-5}$s and
$7.608\times10^{-5}$s, respectively. Even including training, this is significantly faster than the Newton-based FEM solver, which requires approximately
$6955$s and $17447$s for the square and unit disc domains.

\begin{figure*}[h]
\centering
\begin{subfigure}[b]{0.33\textwidth}
\centering
\includegraphics[width=\textwidth]{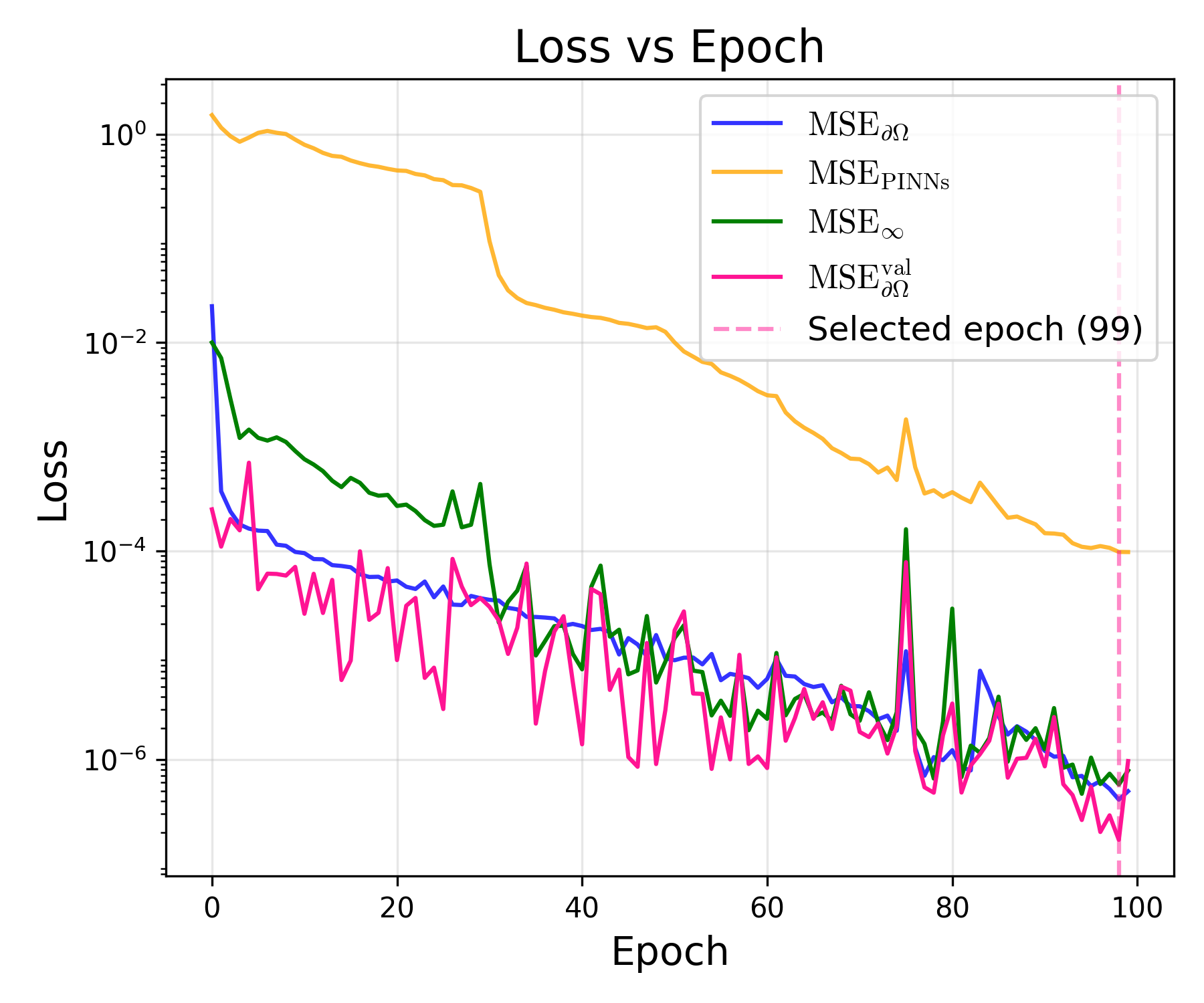}
\caption[]%
{{\small }}    
\end{subfigure}
\begin{subfigure}[b]{0.33\textwidth}
\centering
\includegraphics[width=\textwidth]{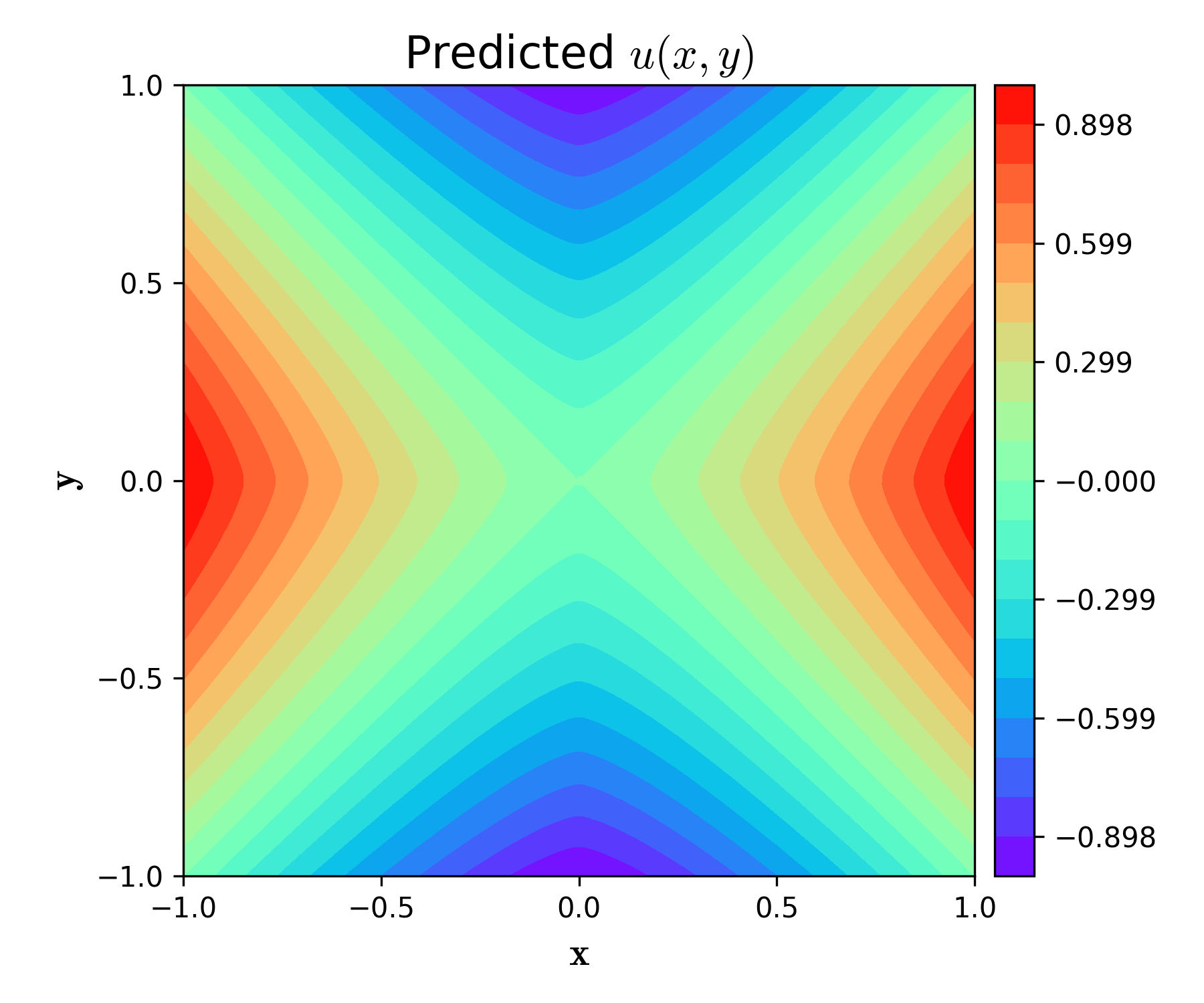}
\caption[]%
{{\small }}    
\end{subfigure}
\begin{subfigure}[b]{0.33\textwidth}  
\centering 
\includegraphics[width=\textwidth]{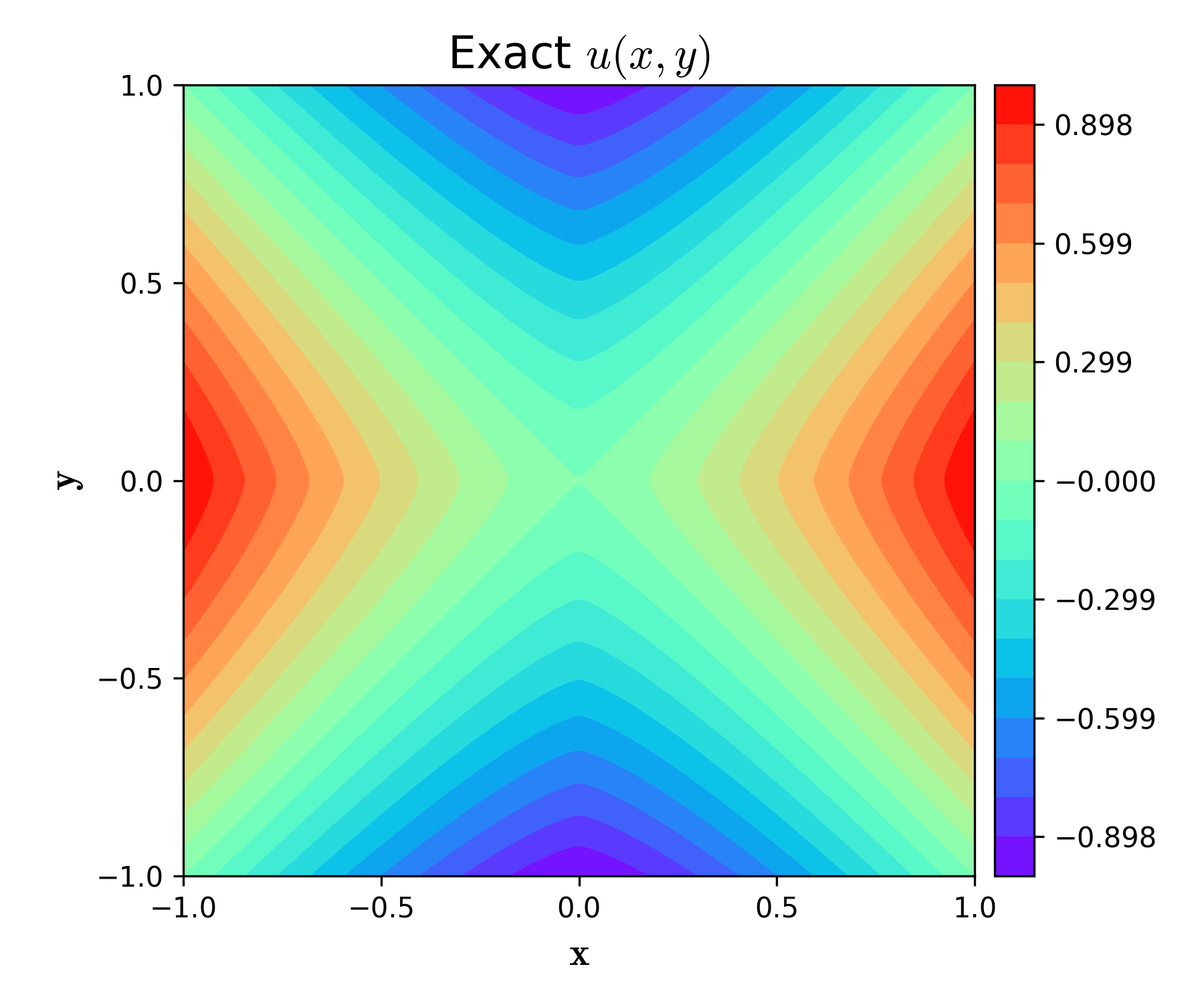}
\caption[]%
{{\small  }}    
\end{subfigure}
\caption[]%
{{\small Aronsson example~\ref{subsec:aronsson} on square domain. (a) Loss histories versus epoch: training boundary loss $\mathrm{MSE}_{\partial\Omega}$, training PINNs loss $\mathrm{MSE}_{\mathrm{PINNs}}$, testing loss $\mathrm{MSE}_{\infty}$, and validation boundary loss $\mathrm{MSE}_{\partial\Omega}^{\mathrm{val}}$, evaluated on the validation subset of boundary points. The dashed vertical line marks the checkpoint selected by the minimum validation boundary loss. (b) PINN prediction. (c) Exact solution.} }
\label{fig:training1.3}
\begin{subfigure}[b]{0.33\textwidth}
\centering
\includegraphics[width=\textwidth]{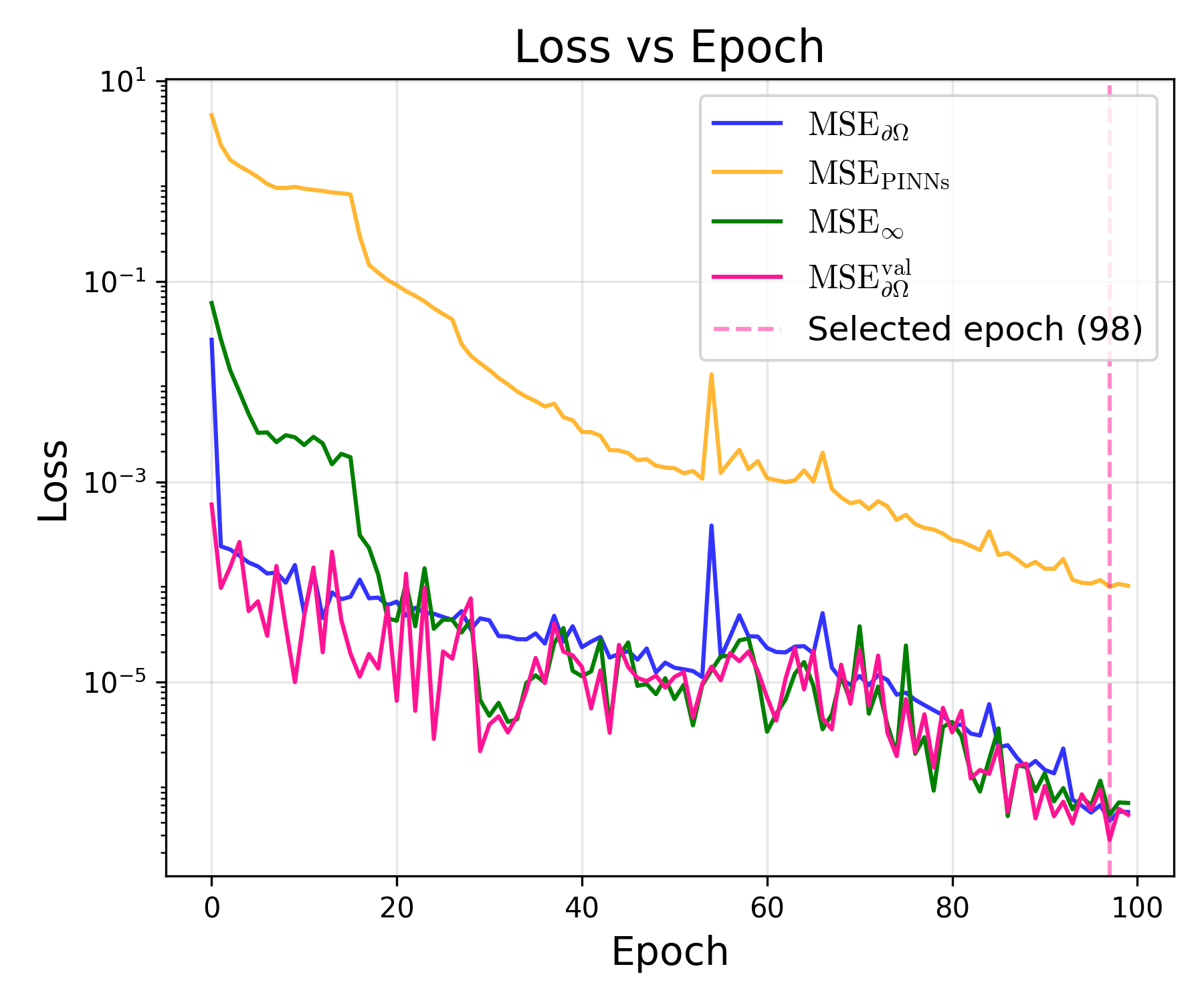}
\caption[]%
{{\small }}    
\end{subfigure}
\begin{subfigure}[b]{0.33\textwidth}
\centering
\includegraphics[width=\textwidth]{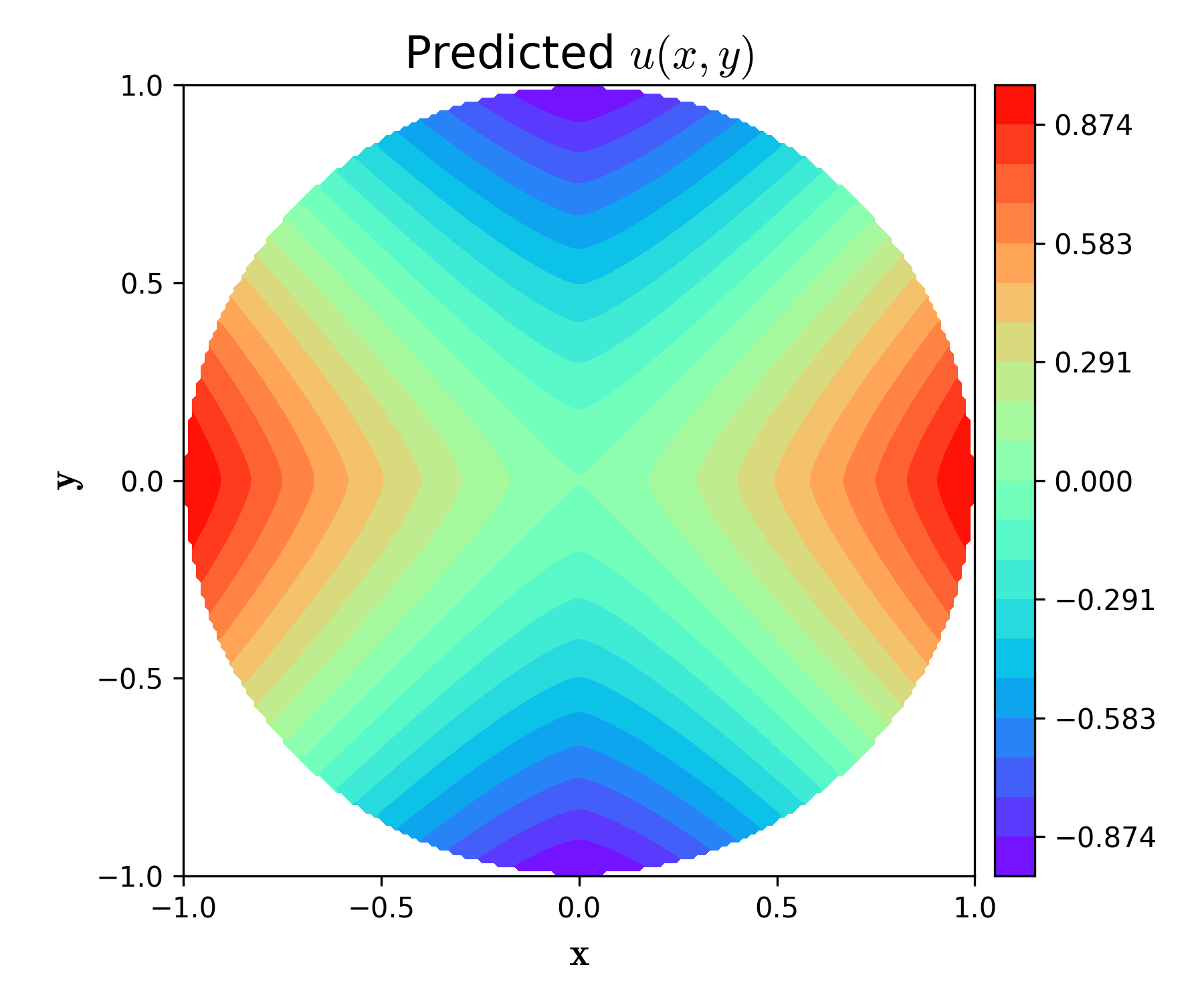}
\caption[]%
{{\small }}    
\end{subfigure}
\begin{subfigure}[b]{0.33\textwidth}  
\centering 
\includegraphics[width=\textwidth]{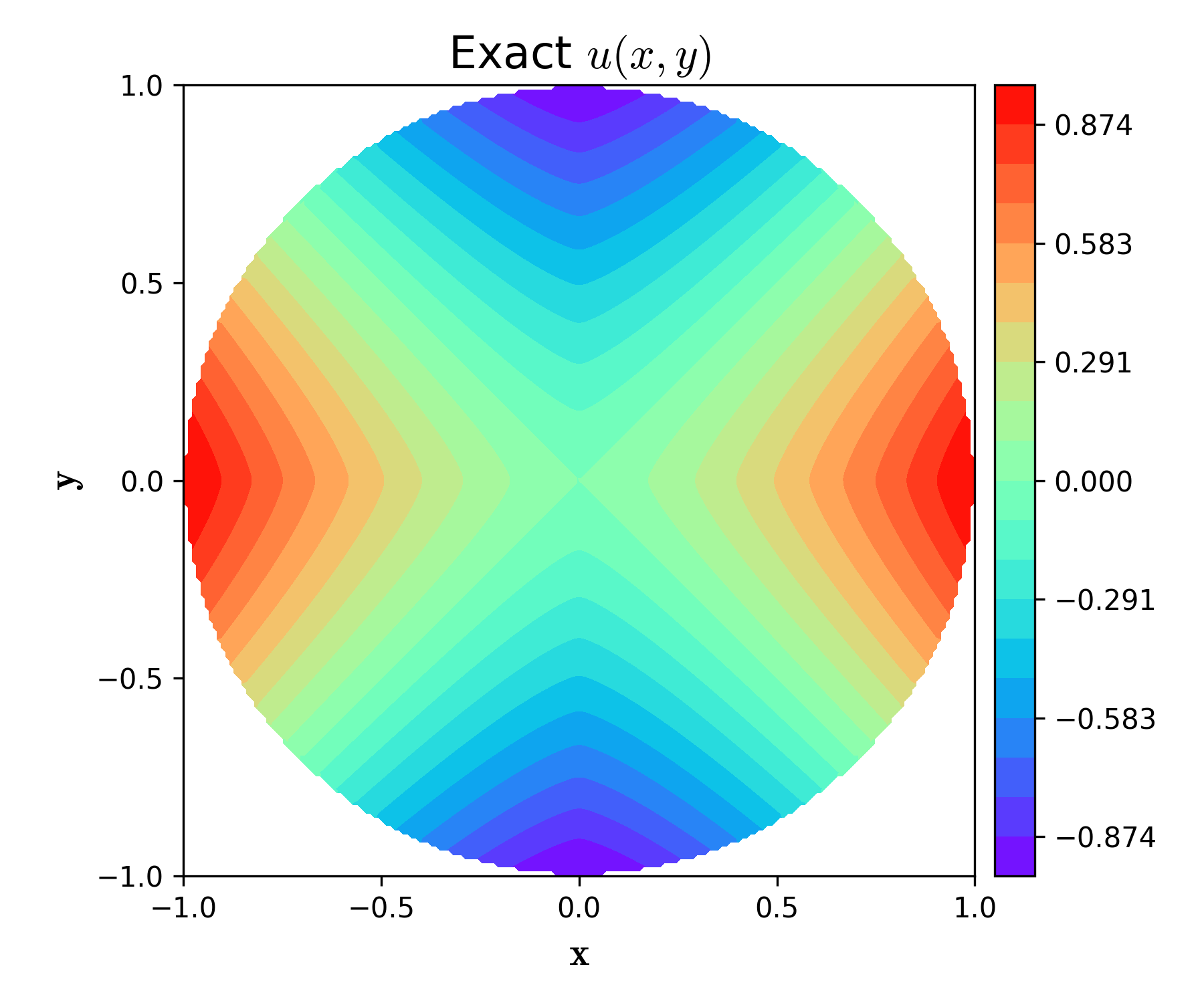}
\caption[]%
{{\small  }}    
\end{subfigure}
\caption[]%
{{\small Aronsson example~\ref{subsec:aronsson} on unit disc. (a) Loss histories versus epoch: training boundary loss $\mathrm{MSE}_{\partial\Omega}$, training PINNs loss $\mathrm{MSE}_{\mathrm{PINNs}}$, testing loss $\mathrm{MSE}_{\infty}$, and validation boundary loss $\mathrm{MSE}_{\partial\Omega}^{\mathrm{val}}$, evaluated on the validation subset of boundary points. The dashed vertical line marks the checkpoint selected by the minimum validation boundary loss. (b) PINN prediction. (c) Exact solution.} }
\label{fig:training1.4}
\end{figure*}

\subsection{$p$-Laplace equation with Dirichlet boundary conditions}
\label{subsec:plaplacian-dbc}

In the second set of experiments, we investigate the use of PINNs to approximate solutions of
the $p$-Laplacian with Dirichlet boundary conditions.
We seek to empirically assess the performance of PINNs for nonlinear elliptic problems
as the nonlinearity parameter $p$ increases.

We specifically train PINN models for $p$-Laplacian equations with varying values of $p$, ranging
from $2$ to $1000$.
Our primary interest is to examine whether the computed PINNs solutions exhibit a systematic
trend toward the corresponding infinity Laplacian solution as $p$ increases.
We also compare two different training strategies: simple training and iterative training.

In addition to the Aronsson example, we consider an inhomogeneous problem on the unit disc, where the exact solution is known for finite $p>1$. This provides a direct validation of the iterative PINN approximation against $u_p$, rather than only an assessment of convergence toward the limiting profile $u_\infty$.

In the simple training strategy, independent PINN models are trained from scratch for each
value of $p$, using only the corresponding PDE residual and boundary losses.
In contrast, the iterative training strategy involves initializing the PINN model for a given
value of $p$ using the pretrained network weights obtained for a smaller value of $p$, and then
fine-tuning the model for the next $p$ value.
This continuation-based approach is motivated by the strong dependence of the $p$-Laplace
operator on $p$ and mirrors continuation strategies commonly employed in traditional numerical
methods.

The network architecture, optimizer, learning rate, grid sizes, and batch sizes are the same as those used for the infinity-Laplacian experiments described in Section~\ref{subsec:2Dinflap-dbc}. 
The same stabilization framework is used, except that the regularized $p$-Laplacian residual is clipped to $[-100,100]$, whereas the infinity-Laplacian residual is clipped to $[-10,10]$. 
See~\ref{app:hyperparams} for the remaining hyperparameters.
In addition to the gradient-norm clamping described above, we regularize the singular prefactor in the expanded $p$-Laplacian.
Writing $\Delta_p u = |\nabla u|^{p-4}\bigl[|\nabla u|^{2}\,\Delta u
+ (p-2)\,\nabla u^{T} D^{2}u\,\nabla u\bigr]$, we replace only the
prefactor $|\nabla u|^{p-4}$ to obtain
\begin{align*}
\tilde{\Delta}_p u = (\eta^{2}+|\nabla u|^{2})^{(p-4)/2} \bigl[|\nabla u|^{2}\,\Delta u + (p-2)\,\nabla u^{T} D^{2}u\,\nabla u\bigr], \qquad \eta = 10^{-5},
\end{align*}
which avoids the singularity at $|\nabla u| = 0$ and numerical blow-up for large~$p$. This modification should be interpreted as a residual-reweighting device for the homogeneous equation, rather than as the flux-based regularization used in the Newton formulation of Section~\ref{section:data}. Away from points where $|\nabla u|=0$, it differs from the original homogeneous $p$-Laplacian residual only by a positive multiplicative factor and therefore preserves its zero set.

\subsubsection{$p$-Laplacian: Aronsson example
with simple PINNs training}

In this experiment, we utilize the Aronsson example to examine the behavior of PINNs when applied
to the $p$-Laplacian for increasing values of $p$.
The goal is to assess whether independently trained PINN models recover the expected
asymptotic behavior as $p \to \infty$.

The Aronsson problem is given by
\begin{align*}
-\Delta_{p} u &=  0, &\qquad \text{ in } \Omega,\\
u(x,y) &= |x|^{\frac{4}{3}}-|y|^{\frac{4}{3}}, &\qquad \text{ on } \partial\Omega.
\end{align*}
Note that $u(x,y)=|x|^{\frac{4}{3}}-|y|^{\frac{4}{3}}$ is an exact solution of the homogeneous
infinity Laplacian~\cite{Aronsson1984}.
This solution is nonsmooth, with reduced regularity along coordinate axes, which is known to cause
difficulties for many numerical methods, including standard finite difference schemes~\cite{OBERMANsecond}.

For this experiment, we employ the simple PINNs training strategy, in which independent models
are trained from scratch for each value of $p$.
The total training loss is defined as the weighted sum of the mean squared error evaluated on boundary
points (training boundary loss) and the PDE residual loss (PINNs loss).

We observe that the mean squared error with respect to the infinity Laplacian solution does not
decrease monotonically as $p$ increases.
Specifically, we measure the error using $\text{MSE}_{\infty}$ as defined in equation \eqref{eq:MSE_test}.

This behavior suggests that, when trained independently, the PINNs formulation for the
$p$-Laplacian may become unstable for large values of $p$.
In particular, the $p$-dependent nonlinearities in the PDE residual can lead to numerical
instabilities and loss imbalance when the network initialization is far from the true solution.

The training time is approximately $2{,}962$\,s for $100$ epochs, and the inference time for evaluation over the entire domain is $2.582\times10^{-4}$\,s.

\subsubsection{$p$-Laplacian: Aronsson example with iterative PINNs training}

In this experiment, we consider both the mean squared error
$\text{MSE}_\infty$ between the predicted solution and the infinity Laplacian limit,
and the residual error $\text{MSE}_{\Delta_{\infty}}$ associated with the homogeneous
infinity Laplace operator, defined in equation \eqref{eq:MSE_test_inflapres}. The purpose of this experiment is to study how these two error measures depend on the parameter
$p$ when PINNs are trained using an iterative, continuation-based strategy.

The iterative training is organized as a five-band continuation pipeline.
The first band trains the model from scratch for
$p\in\{2,3,\dots,10,15,20\}$ using the adaptive plateau scheduler and a
cosine-annealed learning rate, with early stopping 
in this first band only
when both the validation
boundary loss and PDE residual drop below prescribed thresholds.
Each subsequent band loads the best-model checkpoint from the previous band
and fine-tunes for the next group of $p$ values using a fixed $\alpha$ and a
fixed learning rate of $10^{-5}$. 
The value of $\alpha$ decreases across the successive bands to compensate for the increasing magnitude of the unweighted PDE residual at larger $p$ and to prevent the PDE term from dominating the total loss.
The full band-by-band schedule is listed in~\ref{app:iter-schedule}.

By training the PINN models iteratively from $p=2$ to $p=1000$,
we study whether the continuation procedure can track the limiting infinity-Laplacian profile in the large-$p$ regime.
The total training loss is defined as the weighted sum of the mean squared error evaluated on boundary
points (boundary training loss) and the PDE residual loss (PINNs loss).

When $p$ increases from $2$ to $1000$, 
the testing error $\text{MSE}_\infty$ initially decreases with increasing $p$.
For values of $p$ larger than approximately $100$, 
the testing error and PDE residual error $\text{MSE}_{\Delta_{\infty}}$ continue to decrease as $p$
increases, which is consistent with the expected convergence of the $p$-Laplacian operator toward
the infinity Laplacian in the $p\to\infty$ limit.
These results indicate that iterative PINNs training
improves stability as a large-$p$ continuation procedure compared to simple training.

The total training time for the five-band iterative pipeline is approximately $11{,}769$\,s ($1{,}169$ total epochs across all bands and $p$ values), 
and the inference time for evaluation over the entire domain is $7.651\times10^{-5}$\,s.

Table~\ref{tbl:mse-combined} compares the simple and iterative PINNs training strategies. Under simple training, $\text{MSE}_\infty$ decreases up to $p = 100$, but both the training loss and $\text{MSE}_\infty$ then grow rapidly, leading to failure for large values of $p$. In contrast, iterative training stabilizes the boundary loss and the error to the limiting solution across a wide range of $p$, 
allowing us to continue the network profile up to the diagnostic value $p=1000$.
We emphasize that the largest-$p$ entries should be interpreted as large-$p$ continuation diagnostics rather than fully verified finite-$p$ PINN solves. 
In the final continuation band, the PDE residual weight is reduced to $\alpha=10^{-5}$ for stability, see ~\ref{app:iter-schedule}. 
Thus, the total loss may remain small even when the unweighted $p$-Laplacian residual is large. 
The small $\mathrm{MSE}_{\infty}$ and $\mathrm{MSE}_{\Delta_\infty}$ values indicate accuracy with respect to the limiting infinity-Laplacian solution, 
not a separate verification of the finite-$p$ equation at $p=1000$.

Figure~\ref{fig:training3}(a) visualizes the dependence of the total training loss,
$\text{MSE}_\infty$, and $\text{MSE}_{\Delta_{\infty}}$ on $p$ for the Aronsson example~\ref{subsec:aronsson} on square.

\begin{table}[h]
\captionsetup{font=scriptsize}
\begin{center}
\scalebox{0.72}{
\begin{tabular}{c|cccc|ccccc|c}
\hline
\hline
 & \multicolumn{4}{c|}{Simple training} & \multicolumn{5}{c|}{Iterative training} & \\
\cline{2-10}
 $p$ & total loss & $\text{MSE}_{\partial \Omega}$ & PINNs loss & $\text{MSE}_\infty$ & total loss & $\text{MSE}_{\partial \Omega}$ & PINNs loss & $\text{MSE}_\infty$ & $\text{MSE}_{\Delta_{\infty}}$&$\|u_\infty\|_{2,N_{\mathrm{test}}}$\\
\hline
$2$    & $4.455$e-05 & $1.371$e-05 & $3.083$e-03 & $4.429$e-03 & $4.455$e-05 & $1.371$e-05 & $3.083$e-03 & $4.429$e-03 & $2.142$e+01 & $4.277$e-01 \\
$3$    & $4.561$e-05 & $1.192$e-05 & $3.369$e-03 & $1.499$e-03 & $3.818$e-05 & $9.955$e-06 & $2.823$e-03 & $1.466$e-03 & $3.455$e+00 & $4.277$e-01\\
$4$    & $5.200$e-07 & $1.685$e-07 & $3.516$e-03 & $7.125$e-04 & $4.501$e-05 & $8.129$e-06 & $3.688$e-03 & $7.756$e-04 & $1.294$e+00 & $4.277$e-01 \\
$5$    & $5.026$e-07 & $1.686$e-07 & $3.340$e-03 & $4.253$e-04 & $6.131$e-05 & $1.216$e-05 & $4.914$e-03 & $4.592$e-04 & $6.440$e-01 & $4.277$e-01\\
$10$   & $1.443$e-05 & $2.332$e-06 & $1.210$e-03 & $1.057$e-04 & $1.301$e-04 & $1.164$e-05 & $1.185$e-02 & $1.449$e-04 & $1.317$e-01 & $4.277$e-01\\
$50$   & $4.726$e-04 & $2.545$e-05 & $4.471$e-02 & $2.712$e-05 & $1.998$e-04 & $1.108$e-05 & $1.888$e-02 & $1.034$e-05 & $1.052$e-02 & $4.277$e-01\\
$100$  & $3.674$e-04 & $1.458$e-05 & $3.528$e-02 & $1.194$e-05 & $6.388$e-05 & $2.761$e-06 & $6.112$e-02 & $3.754$e-06 & $3.804$e-03 & $4.277$e-01\\
$200$  & $6.930$e-03 & $5.296$e-03 & $1.634$e-02 & $1.652$e-02 & $2.771$e-05 & $1.850$e-06 & $2.586$e-01 & $1.808$e-06 & $1.835$e-03 & $4.277$e-01\\
$1000$ & $4.287$e+01 & $5.172$e-04 & $4.287$e+02 & $1.085$e-01 & $6.812$e-05 & $1.762$e-06 & $6.635$e+00 & $1.056$e-06 & $7.236$e-04 & $4.277$e-01\\
\hline
\hline
\end{tabular}
}
\end{center}
\caption{Comparison of simple and iterative PINNs training for the Aronsson example on the square domain.
For both methods, we report the total training loss, boundary loss $\text{MSE}_{\partial \Omega}$, PINNs physical loss, and solution error $\text{MSE}_\infty$.
For iterative training, we additionally report the residual error $\text{MSE}_{\Delta_{\infty}}$ under the homogeneous infinity Laplacian operator.}\label{tbl:mse-combined}
\end{table}

\begin{figure*}[t]
\centering

\begin{subfigure}[t]{0.48\textwidth}
    \centering
    \includegraphics[width=\linewidth]
    {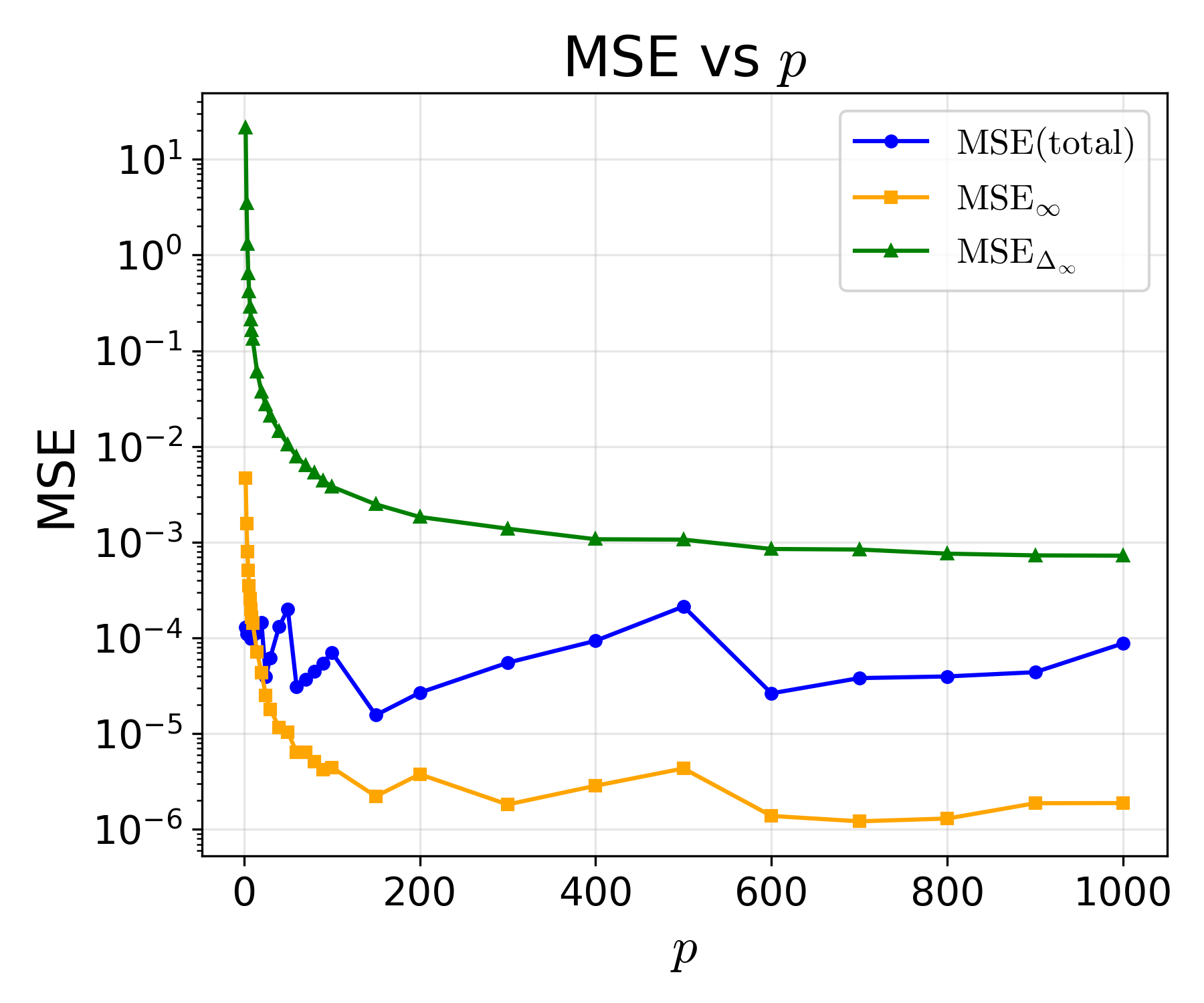}
    \caption{Aronsson example.}
    \label{fig:training3-aronsson}
\end{subfigure}
\hfill
\begin{subfigure}[t]{0.48\textwidth}
    \centering
    \includegraphics[width=\linewidth]
    {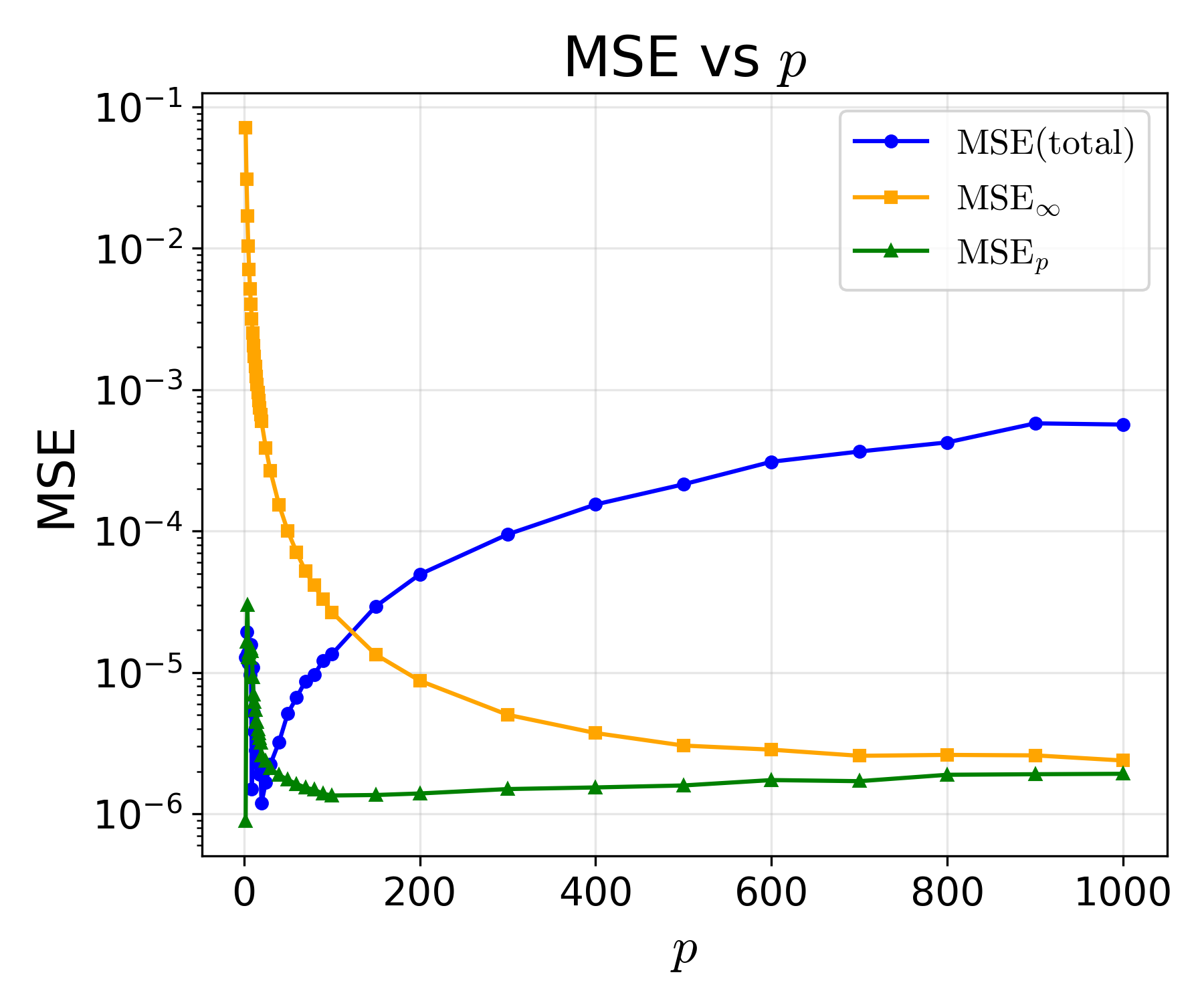}
    \caption{Direct finite-$p$ validation on the unit disc.}
    \label{fig:training3-disc}
\end{subfigure}

\caption{
Iterative PINN continuation results over increasing values of $p$.
In both panels, the blue curve represents the total training loss and
the orange curve represents $\mathrm{MSE}_{\infty}$.
In panel~(a), the green curve represents
$\mathrm{MSE}_{\Delta_\infty}$ for the Aronsson example.
In panel~(b), the green curve represents the direct finite-$p$ error
$\mathrm{MSE}_{p}$ against the exact solution $u_p$.
}
\label{fig:training3}
\end{figure*}

\subsubsection{Finite-$p$ validation on the unit disc}
\label{subsec:pinn-finite-p-disc}

To directly validate the iterative PINN method for finite values of $p$, let $\Omega\subset\mathbb{R}^2$ be the unit disc centered at the origin. We consider the $p$-Poisson problem
\begin{align*}
-\Delta_p u_p &= 1,
&& \text{in } \Omega,\\
u_p &= 0,
&& \text{on } \partial\Omega.
\end{align*}
A direct radial calculation gives the exact solution
\begin{align*}
u_p(\bx)
=
\frac{p-1}{p}\,2^{-1/(p-1)}
\left(
1-\|\bx\|_2^{p/(p-1)}
\right),
\qquad p>1.
\end{align*}
Moreover,
\begin{align*}
u_p(\bx)\longrightarrow
u_\infty(\bx)=1-\|\bx\|_2
\qquad\text{as }p\to\infty,
\end{align*}
where $u_\infty$ is the distance to the boundary of the unit disc.

The PINN is trained using continuation in $p$, with the network weights from
one value of $p$ used to initialize the next continuation stage. Unlike the
homogeneous Aronsson experiments, this problem has a nonzero right-hand side.
We therefore use the flux-based regularization
\begin{align*}
-\Delta_{p,\eta}u
=
-\nabla\cdot
\left[
\left(\eta^2+|\nabla u|^2\right)^{(p-2)/2}\nabla u
\right],
\qquad \eta=10^{-5},
\end{align*}
together with the gradient-norm and residual stabilization used in the
preceding PINN experiments. The exact comparison target $u_p$ remains the
solution of the original unregularized finite-$p$ problem. The complete continuation schedule and sampling configuration are reported in \ref{app:iter-schedule}.

For each saved continuation checkpoint, we report the total training loss
\begin{align*}
\mathrm{MSE}_{\mathrm{total}}
=\mathrm{MSE}_{\partial\Omega}
+\alpha\,\mathrm{MSE}_{\mathrm{PINNs}},
\end{align*}
where $\mathrm{MSE}_{\partial\Omega}$ is the training boundary loss and
$\mathrm{MSE}_{\mathrm{PINNs}}$ is the mean squared regularized $p$-Poisson
residual at the training collocation points. We additionally report the direct
finite-$p$ error $\mathrm{MSE}_{p}$ defined in
\eqref{eq:MSE_test_finite_p}, and the error $\mathrm{MSE}_{\infty}$ against the
limiting distance function. These two quantities are evaluated on the fixed
$100\times100$ test grid
restricted to the unit disc, containing $N_{\mathrm{test}}=7{,}668$ points.

\begin{table}[h]
\captionsetup{font=scriptsize}
\begin{center}
\scalebox{0.82}{
\begin{tabular}{c|ccccccc}
\hline
\hline
$p$
& total loss
& $\mathrm{MSE}_{\partial\Omega}$
& PINNs loss
& $\mathrm{MSE}_{\infty}$
&$\|u_\infty\|_{2,N_{\mathrm{test}}}$
& $\mathrm{MSE}_{p}$ 
&$\|u_p\|_{2,N_{\mathrm{test}}}$\\
\hline
$2$    & $1.287$e-05 & $1.203$e-05 & $8.403$e-04 & $7.145$e-02 & $4.090$e-01 & $8.861$e-07 & $1.446$e-01 \\
$3$    & $1.930$e-05 & $9.691$e-06 & $9.607$e-04 & $3.070$e-02 & $4.090$e-01 & $1.645$e-05 & $2.395$e-01\\
$4$    & $1.366$e-05 & $7.840$e-06 & $5.816$e-04 & $1.696$e-02 & $4.090$e-01 & $2.982$e-05 & $2.851$e-01\\
$5$    & $1.173$e-05 & $4.237$e-06 & $7.498$e-04 & $1.033$e-02 & $4.090$e-01 & $1.284$e-05 & $3.116$e-01\\
$10$   & $1.090$e-05 & $1.296$e-06 & $9.608$e-04 & $2.507$e-03 & $4.090$e-01 & $9.213$e-06 & $3.620$e-01\\
$50$   & $5.152$e-06 & $2.383$e-07 & $4.914$e-04 & $9.981$e-05 & $4.090$e-01 & $1.740$e-06 & $3.999$e-01\\
$100$  & $1.346$e-05 & $2.560$e-07 & $1.320$e-03 & $2.672$e-05 & $4.090$e-01 & $1.350$e-06 & $4.045$e-01\\
$200$  & $4.942$e-05 & $3.388$e-07 & $4.908$e-03 & $8.748$e-06 & $4.090$e-01 & $1.395$e-06 & $4.068$e-01\\
$1000$ & $5.674$e-04 & $8.495$e-07 & $5.665$e-02 & $2.385$e-06 & $4.090$e-01 & $1.922$e-06 & $4.086$e-01\\
\hline
\hline
\end{tabular}
}
\end{center}
\caption{Direct finite-$p$ validation of the iterative PINN method for the
distance-to-boundary problem on the unit disc. We report the total training
loss, the boundary loss $\mathrm{MSE}_{\partial\Omega}$, the PINNs physical
loss, the solution error $\mathrm{MSE}_{\infty}$ against the limiting distance
function $u_\infty$, and the direct solution error $\mathrm{MSE}_{p}$ against
the exact finite-$p$ solution $u_p$.}
\label{tbl:pinn-finite-p-disc}
\end{table}

Table~\ref{tbl:pinn-finite-p-disc} shows that the direct finite-$p$ error remains below $3\times10^{-5}$ for all reported values of $p$ and equals $1.922\times10^{-6}$ at $p=1000$. At the same time, $\mathrm{MSE}_{\infty}$ decreases from $7.145\times10^{-2}$ at $p=2$ to $2.385\times10^{-6}$ at $p=1000$, consistently capturing convergence toward the distance-to-boundary limit. The training PINNs loss becomes larger at the largest values of $p$, while the direct comparison with the known finite-$p$ solution remains accurate. Figure~\ref{fig:training3}(b) shows the corresponding continuation history over all trained values of $p$. $\mathrm{MSE}_{\infty}$ decreases as $p$ increases, while the direct finite-$p$ error $\mathrm{MSE}_{p}$ remains small despite the increase in the total training loss at large $p$.

\subsection{DeepONet example}

In the third set of experiments, we use DeepONet to learn from finite-$p$ solution data and 
evaluate how well the resulting predictions recover the corresponding $p\to\infty$ limiting solution.
In some cases, a PINN approximation of the limiting solution is included as an auxiliary input for $p=500$.
Furthermore, we extend the training data by including parameters describing the domain geometry, allowing the trained DeepONet to generalize across both the parameter $p$ and a family of domains.

As a reference baseline, we measure the convergence of the FEM solutions toward the infinity Laplacian limit using
\begin{align*}
\mathrm{MSE}_{\infty, p}^{\text{FEM}}
= 
\frac{1}{N_{\mathrm{test}}}
\sum_{k=1}^{N_{\mathrm{test}}}
\bigl(u_p^{\text{FEM}}(\bx_k) - u_\infty(\bx_k)\bigr)^2
\end{align*}
where $u_p^{\text{FEM}}$ is the Newton FEM solution at parameter~$p$ and
$u_\infty$ is the exact $p \to \infty$ limit. This resembles~\eqref{eq:MSE_test_pinf} but with the FEM solution rather than the network prediction being compared to $ u_\infty$.
This is a convergence-to-limit baseline for the FEM solution at finite $p$, rather than an estimate of FEM discretization error.

The trained model is evaluated using the mean squared error between the
DeepONet prediction and the limiting infinity Laplacian solution given by~\eqref{eq:MSE_test_pinf}.

This error metric allows us to assess how well the DeepONet predictions
recover the asymptotic $p\to\infty$ behavior across different values
of $p$ and different domain configurations.
We emphasize that this metric is not a finite-$p$ error $\|\widehat u_p-u_p\|^2$. 
Instead, $\mathrm{MSE}_{\infty,p}$ measures the discrepancy between the DeepONet prediction at parameter $p$ and the limiting solution $u_\infty$. 
Thus, except for the direct unit-disc finite-$p$ validation reported in Table~\ref{tbl:deeponet-finite-p-disc}, the DeepONet errors assess accuracy for recovering the limiting distance profile rather than directly verifying the finite-$p$ solution operator.

\subsubsection{Distance-to-origin DeepONet example}
\label{subsec:deeponet-dist2orig}

In this experiment, we consider the distance-to-origin problem on both disc and square domains
and use the corresponding $p$-Laplacian solutions as training data for the DeepONet model.
The goal is to assess whether DeepONet can learn the parametric dependence of the solution on $p$
and recover the correct $p \to \infty$ limiting behavior.
The boundary value problem is given by
\begin{align*} 
-\Delta_{p} u_p &=  1, &\qquad \text{ in } \Omega,\\
u_p &= 0, &\qquad \text{ on } \Gamma_1 =  (0,0),\\
\frac{\partial u_p}{\partial n} &= 0, &\qquad \text{ on }  \Gamma_2 = \partial \Omega \setminus (0,0).
\end{align*}

We have the limiting behavior
\begin{align*}
\lim_{p\to\infty} u_p = u_{\infty} = \mathrm{dist}(\bx,(0,0)).
\end{align*}

For each testing configuration, the range of $p$ values used in the training data spans from
$p=5$ to $p=200$.
In addition, we include the PINNs-predicted $p=\infty$ solution as an auxiliary input by associating it
with $p=500$ in the training set.
This allows us to explicitly incorporate information about the limiting solution into the operator-learning process.

Numerical evidence~\cite{Potgieter} suggests that the limiting solution $u_\infty$ satisfies the
Eikonal equation
\begin{subequations}
\begin{alignat}{2}
\label{eq:eikonal}
    |\nabla u_\infty| &=  1, &\qquad \text{ in } \Omega,\\ 
  u_\infty &= 0, &\qquad \text{ on } \Gamma_1 ,
  \label{eq:eikonalbc}
\end{alignat}
\end{subequations}
Note that the Neumann condition vanishes in the $p \to \infty$ limit.

We evaluate the DeepONet predictions for $p$ values ranging from $5$ to $500$ and compute the mean
squared errors as an error metric. The results reported in Table~\ref{tbl:mse4} and Figure~\ref{fig:training3.1} for both the disc and square domains 
demonstrate that including the PINNs-predicted $p=\infty$ solution in the training data significantly improves
accuracy with respect to the limiting-solution metric $\mathrm{MSE}_{\infty,p}$.
In particular, when the $p=\infty$ information is included, the error remains more controlled over the augmented parameter range.
In contrast, when the $p=\infty$ solution is omitted from the training inputs, the MSE
continues to increase for $p>200$, which coincides with the upper limit of the available training data.

These observations indicate that incorporating information about the limiting solution can substantially improve DeepONet predictions in the large-$p$ regime. 
Since the limiting solution is included as a training label at the surrogate value $p=500$, 
predictions for $200<p<500$ should be interpreted as interpolation within an augmented parameter range,
rather than extrapolation beyond the available training data.

\begin{table}[h]
\begin{center}
\scalebox{0.9}{
\begin{tabular}{c|ccc|ccc}
\hline
\hline
&&{Disc domain} && &{Square domain}&\\
 $p$ &  $\text{MSE}_{\infty,p}^{\text{FEM}}$ & \begin{tabular}{@{}c@{}}$\text{MSE}_{\infty,p}$ \\ (no limit reference)\end{tabular} & $\text{MSE}_{\infty,p}$ & $\text{MSE}_{\infty,p}^{\text{FEM}}$ & \begin{tabular}{@{}c@{}}$\text{MSE}_{\infty,p}$ \\ (no limit reference)\end{tabular} & $\text{MSE}_{\infty,p}$\\
\hline
$5$ & $1.575$e-02 & $1.289$e-02 & $1.288$e-02 & $2.377$e-02 & $2.337$e-02 & $2.328$e-02\\
$50$ & $4.022$e-05 & $4.009$e-05 & $4.172$e-05 & $6.865$e-05 & $6.628$e-05 & $6.834$e-05\\
$100$ & $5.610$e-06 & $5.250$e-06 & $6.804$e-06 & $1.180$e-05 & $1.177$e-05 & $1.193$e-05 \\
$200$ & $1.136$e-06 & $3.251$e-06 & $5.656$e-06 & $1.523$e-06 & $2.126$e-06 & $2.216$e-06\\
$300$ & $-$ & $8.135$e-05 & $5.098$e-05 & $-$ & $2.603$e-05 & $6.411$e-06\\
$400$ & $-$ & $9.493$e-04 & $5.605$e-05 & $-$ & $1.861$e-04 & $5.638$e-06\\
$500$ & $-$ & $3.549$e-03 & $2.872$e-06 & $-$ & $5.527$e-04 & $7.148$e-07\\
\hline
\hline
\end{tabular}
}
\end{center}
\caption{Error evaluation for the distance to origin DeepONet example~\ref{subsec:deeponet-dist2orig} on the disc and square domains. For each $p$ value, we show the FEM convergence-to-limit baseline $\text{MSE}_{\infty,p}^{\text{FEM}}$. We also show the DeepONet error $\text{MSE}_{\infty,p}$ which compares against the exact $p \to \infty$ limiting solution both without including $p=\infty$ training data and with PINNs-predicted $p=\infty$ data added as a surrogate for $p=500$.}\label{tbl:mse4}
\end{table}

\begin{figure*}[h]
\centering
\begin{subfigure}[b]{0.4\textwidth}
\centering
\includegraphics[width=\textwidth]{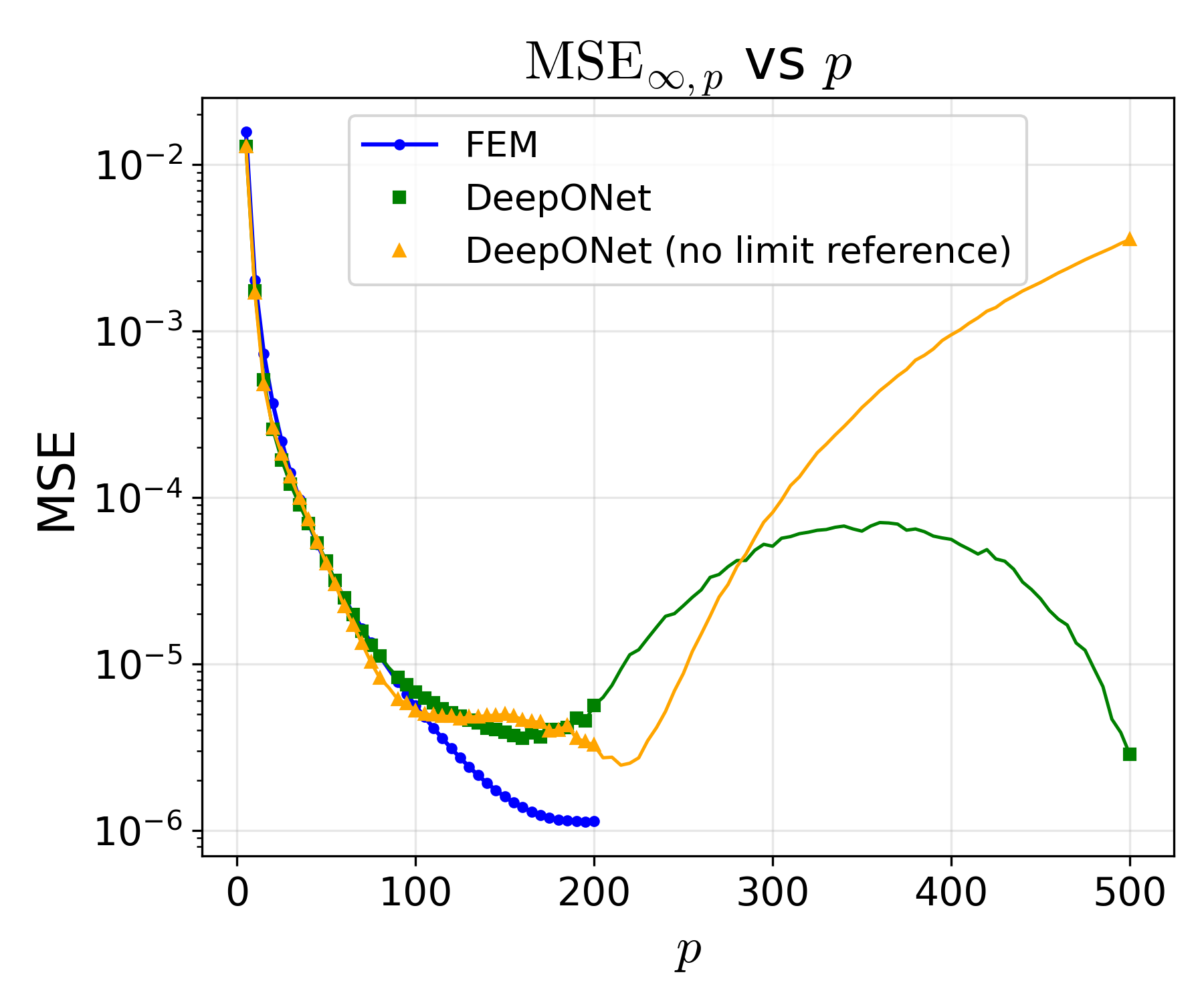}
\caption[]%
{{\small unit disc}}    
\end{subfigure}
\begin{subfigure}[b]{0.4\textwidth}
\centering
\includegraphics[width=\textwidth]{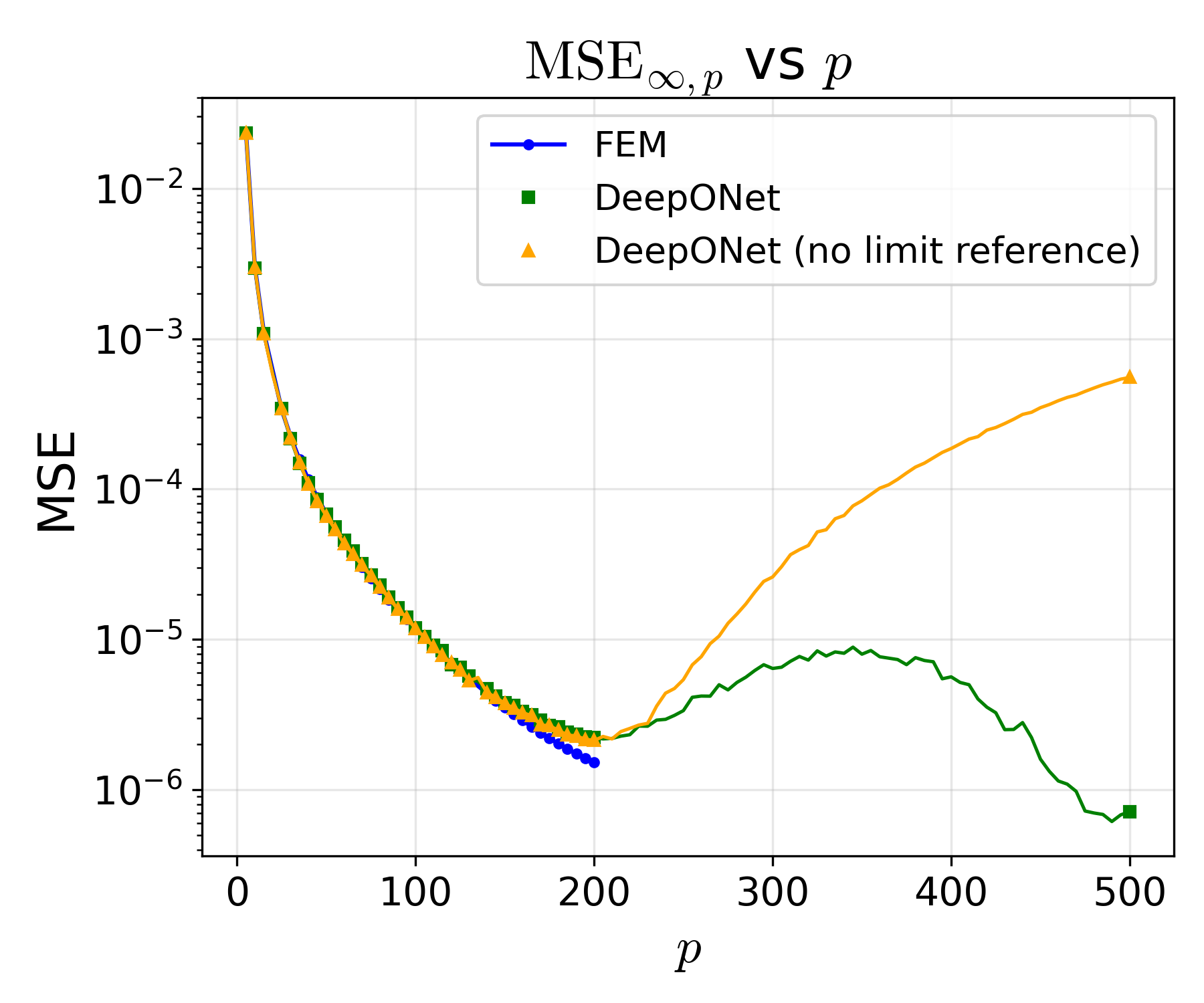}
\caption[]%
{{\small square}}    
\end{subfigure}
\caption[]%
{{\small Distance to origin DeepONet example~\ref{subsec:deeponet-dist2orig} (a) on unit disc domain (b) on square domain. 
The blue curves show the FEM convergence-to-limit baseline $\text{MSE}_{\infty,p}^{\text{FEM}}$, 
the orange curves show the DeepONet error $\text{MSE}_{\infty,p}$ against the exact $p \to \infty$ limiting solution without including $p=\infty$ training data, 
and the green curves show the DeepONet error $\text{MSE}_{\infty,p}$ with the PINNs-predicted solution at $p=\infty$ added as a surrogate for $p=500$. 
All DeepONet errors in this figure are measured against the limiting solution $u_\infty$, not against the finite-$p$ FEM solution $u_p$.
}} 
\label{fig:training3.1}
\end{figure*}

\subsubsection{Distance-to-boundary DeepONet 2D and 3D examples}
\label{subsec:deeponet-dist2bdry}

Bhattacharya et al.~\cite{bhattacharya} show that solutions of the $p$-Laplacian can be used to
approximate the distance-to-boundary function in the large $p$ limit. Unfortunately, large values of $p$ lead to increasingly ill-conditioned problems and are therefore
numerically challenging to solve.
The boundary value problem is given by
\begin{align*}
-\Delta_{p} u_p &=  1, &\qquad \text{ in } \Omega,\\
u_p &= 0, &\qquad \text{ on } \partial\Omega. 
\end{align*}

It is shown analytically in~\cite{bhattacharya} that
\begin{align*}
\lim_{p\to\infty} u_p = u_{\infty} = \mathrm{dist}(\bx,\partial \Omega).
\end{align*}

There exists a substantial body of work on numerical methods for approximating solutions of this
problem, including finite element and finite difference approaches~\cite{fayolle2018p,caliari2017quasi,toulopoulos2017numerical}.
In contrast to these methods, our goal is to assess whether DeepONet can learn the parametric
dependence of the solution on $p$ and accurately recover the limiting distance function without
directly solving highly ill-conditioned problems for large $p$.

To address the reliance on exact $p=\infty$ solutions as training data, we employ a PINN model to
approximate the limiting distance function by solving the corresponding Eikonal equation. 
For the distance-to-boundary problem, the large-$p$ limit is the distance function $u_\infty=\mathrm{dist}(\mathbf{x},\partial\Omega)$,
which satisfies the Eikonal problem (\ref{eq:eikonal})-(\ref{eq:eikonalbc}).

Table~\ref{tbl:mse5} demonstrates that the DeepONet model can recover the large-$p$ limiting profile with errors comparable to the FEM convergence-to-limit baseline.
Thus, the proposed approach is capable of accurately approximating the limiting distance profile on 2D geometries
and remains effective for substantially more challenging 3D domains. 
The DeepONet training time is approximately $17500$s ($4.86$h). This is higher than the Newton computation times in 2D, $8667$s ($2.41$h) for the disc and approximately $3.5$h for the ellipses. n three dimensions the Newton computations require up to approximately 5 days. Since the DeepONet is trained on these FEM solutions, it does not avoid this cost; rather, once trained, it can be evaluated efficiently for many values of $p$ and a family of geometries without further solves.

Figure~\ref{fig:training3.2} illustrates how the errors for the 2D examples vary with $p$ when
evaluated against the true distance function, corresponding to the $p\to\infty$ limit.
In all cases, the $\mathrm{MSE}_{\infty,p}$ curves decay comparably to 
the $\mathrm{MSE}_{\infty,p}^{\mathrm{FEM}}$ curves 
within the finite-$p$ training range.
Including the PINNs-predicted solution at $p=500$ in the training data enables interpolation between
the upper end of the training range and larger values of $p$. 
However, we note that there is some non-monotonicity for the $p$ values beyond the FEM training range particularly pronounced for ellipse~2.
Here, the $\mathrm{MSE}_{\infty,p}$ curve increases, contrary to theoretical expectations.
Recall that convergence to the distance function is expected at a rate of $\mathcal{O}(1/p)$.
For sufficiently large $p$, discretization error and network approximation error are therefore
expected to dominate the asymptotic $1/p$ convergence behavior.

For the unit-disc example, the exact finite-$p$ solution is available for every $p>1$, as described in Section~\ref{subsec:pinn-finite-p-disc}. This allows us to supplement the convergence-to-limit comparison with a direct evaluation of the DeepONet prediction against $u_p$. Table~\ref{tbl:deeponet-finite-p-disc} shows finite-$p$ error remains below $1.6\times10^{-6}$ at all $p$ values. Small errors are obtained at both the finite-$p$ training inputs and unseen interpolation points, directly validating the DeepONet predictions against the corresponding exact solutions $u_p$.

Figure~\ref{fig:training3.3} shows a similar evolution of the errors with respect to $p$ for the 3D
examples when compared against the true distance function.
The most consistent behavior is observed for the spherical domain.
In the cylindrical case, the decay of $\mathrm{MSE}_{\infty,p}$ becomes more pronounced after
the FEM training range.
For the torus, $\mathrm{MSE}_{\infty,p}$ remains larger than
$\mathrm{MSE}_{\infty,p}^{\mathrm{FEM}}$ even at $p=500$.

\begin{figure*}[h]
\centering
\begin{subfigure}[b]{0.33\textwidth}
\centering
\includegraphics[width=\textwidth]{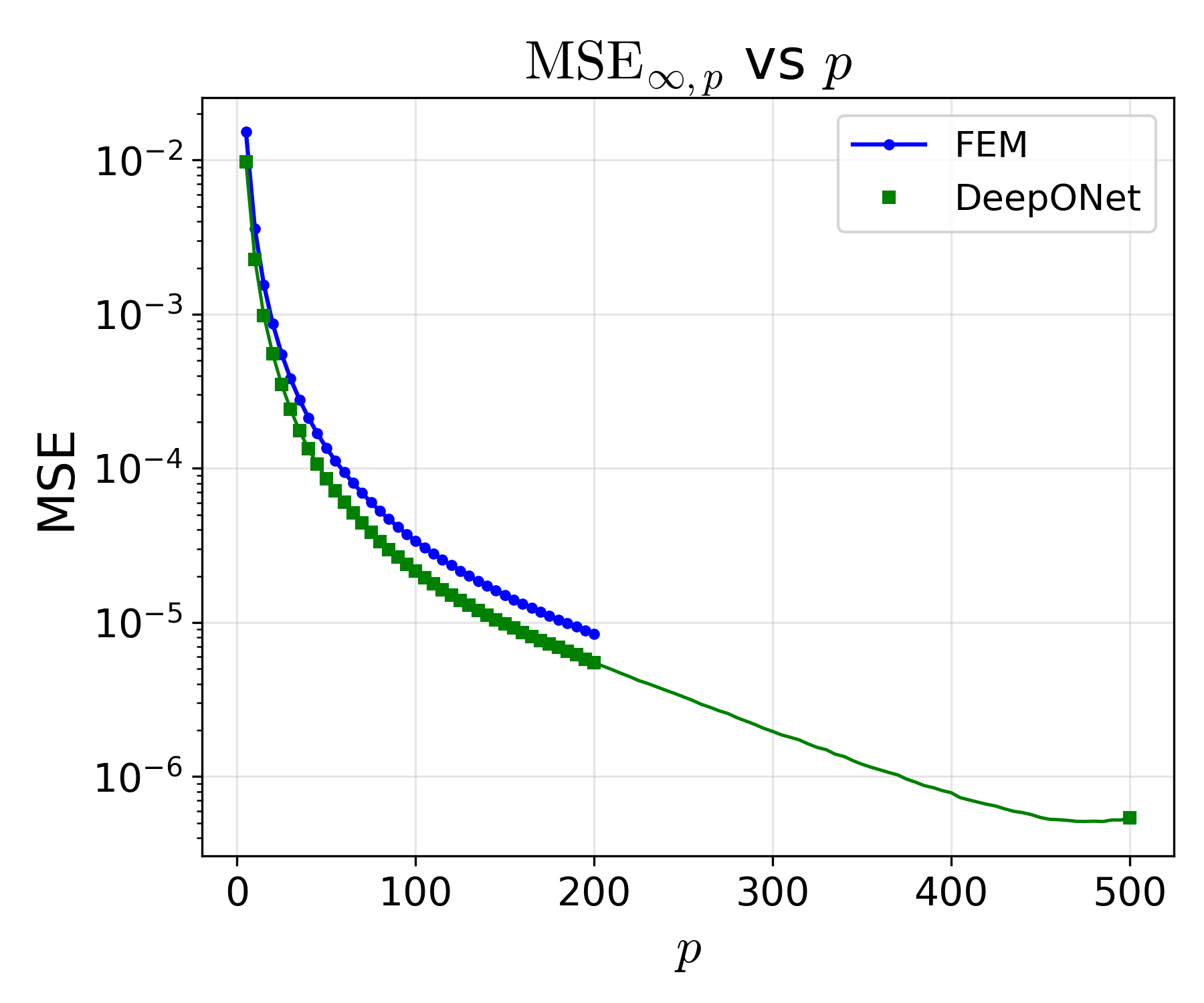}
\caption[]%
{{\small 2D disc: $x^2+y^2\leq 1$ }}    
\end{subfigure}
\begin{subfigure}[b]{0.33\textwidth}
\centering
\includegraphics[width=\textwidth]{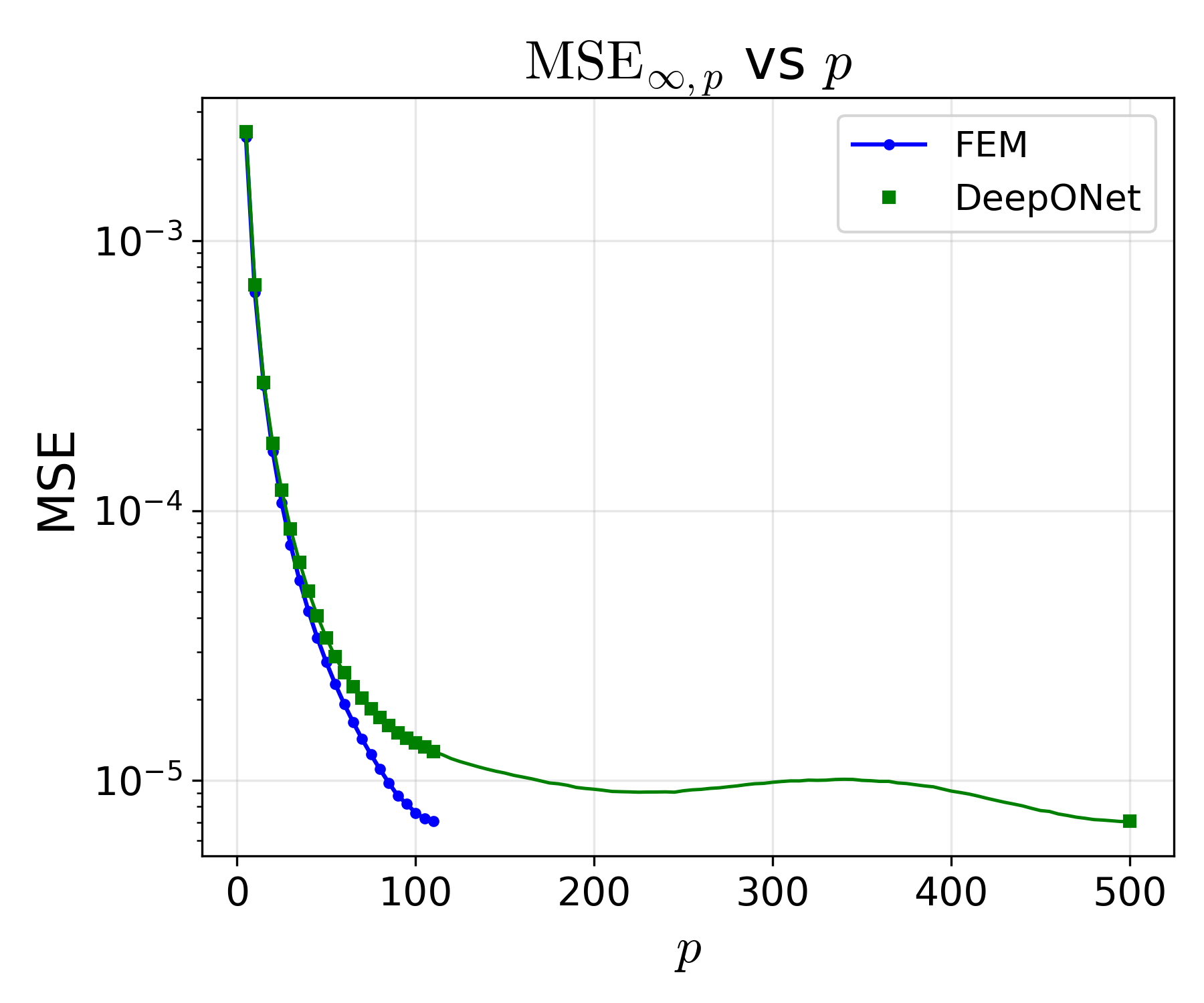}
\caption[]%
{{\small Ellipse 1: $x^2+16y^2\leq 1$ }}    
\end{subfigure}\\
\begin{subfigure}[b]{0.33\textwidth}
\centering
\includegraphics[width=\textwidth]{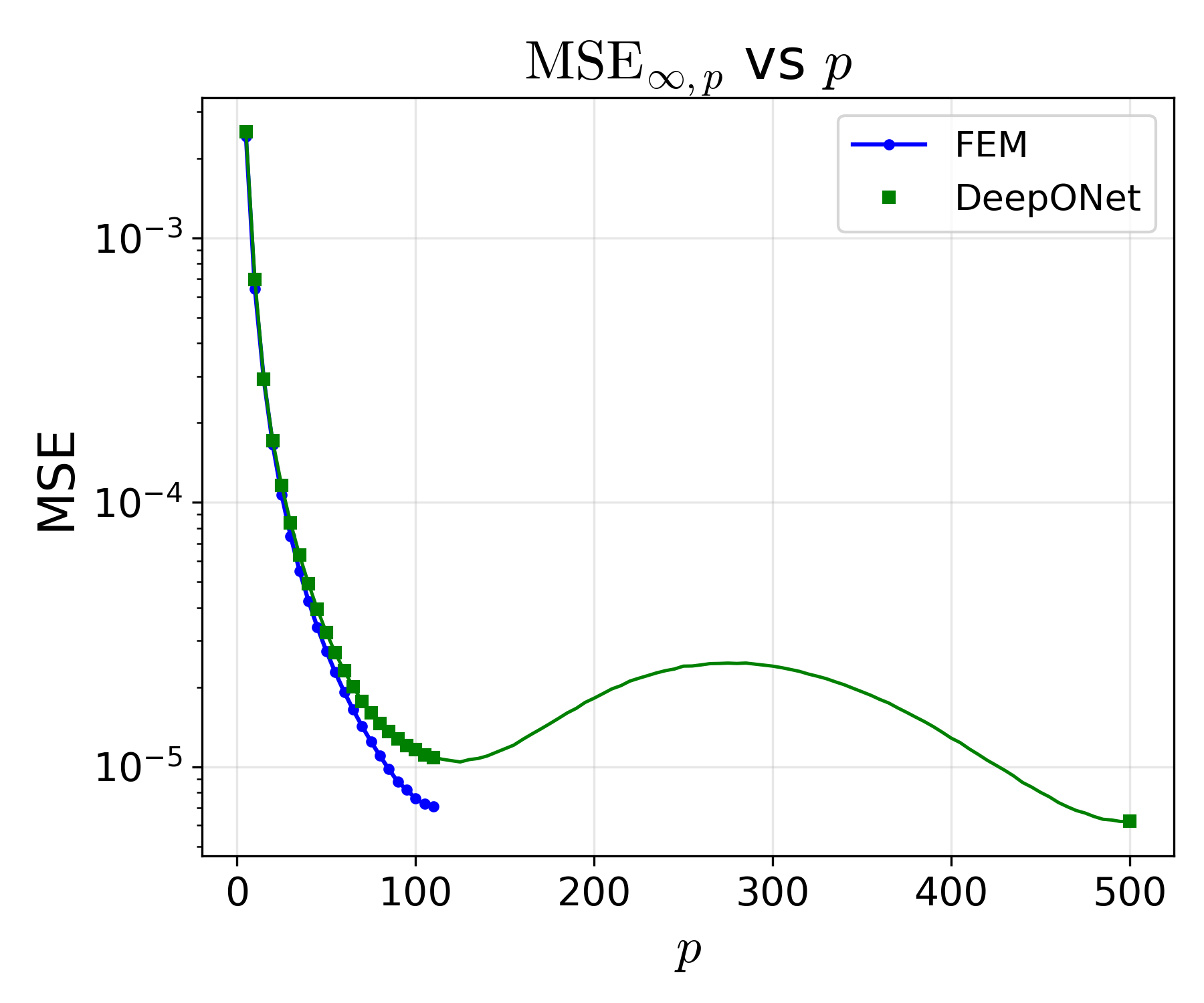}
\caption[]%
{{\small Ellipse 2: $8.5x^2+8.5y^2-15xy\leq 1$ }}    
\end{subfigure}
\begin{subfigure}[b]{0.33\textwidth}
\centering
\includegraphics[width=\textwidth]{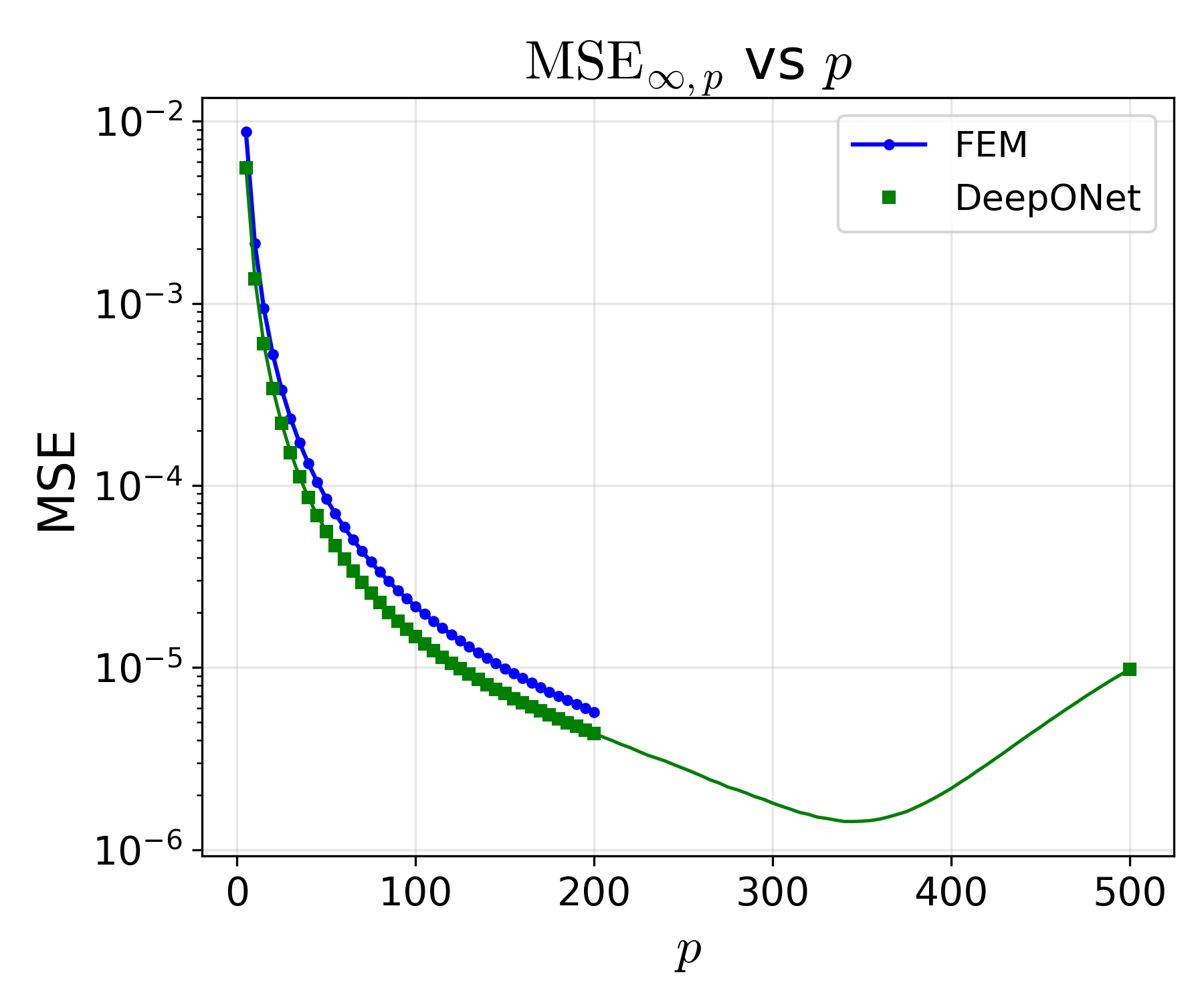}
\caption[]%
{{\small Ellipse 3: $4x^2+y^2\leq 1$ }}    
\end{subfigure}
\caption[]%
{{\small Distance to boundary DeepONet results for various 2D domains. (a)–(d) show $\text{MSE}_{\infty,p}^{\text{FEM}}$ and $\text{MSE}_{\infty,p}$ over increasing $p$ values for: (a) the unit disc, (b) an elongated ellipse, (c) a rotated ellipse, and (d) a compressed ellipse. All predictions are compared against the $p=\infty$ solution which is the true distance function.}}
\label{fig:training3.2}
\end{figure*}

\begin{figure*}[h]
\centering
\begin{subfigure}[b]{0.33\textwidth}
\centering
\includegraphics[width=\textwidth]{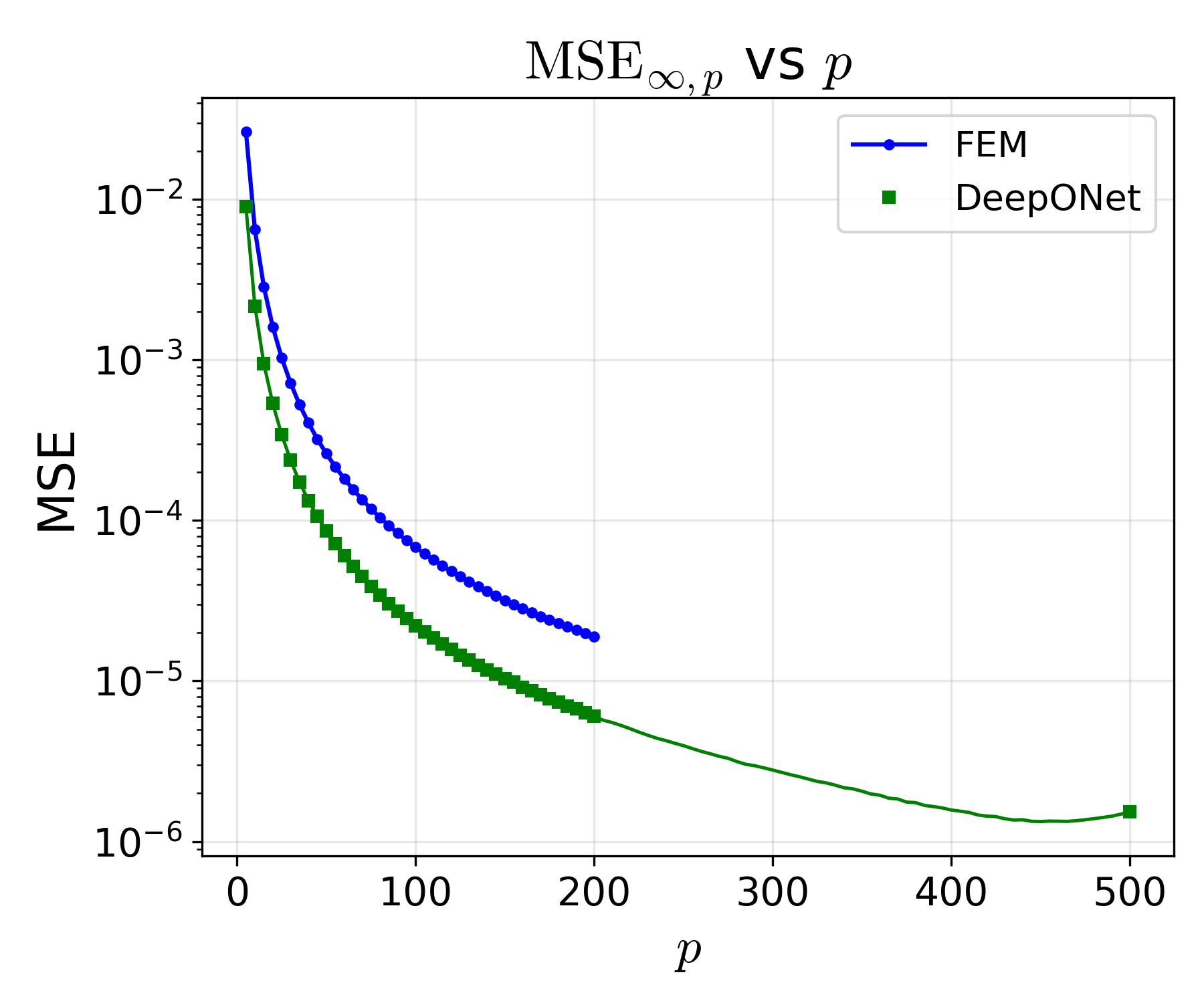}
\caption[]%
{{\small 3D ball}}    
\end{subfigure}
\begin{subfigure}[b]{0.33\textwidth}
\centering
\includegraphics[width=\textwidth]{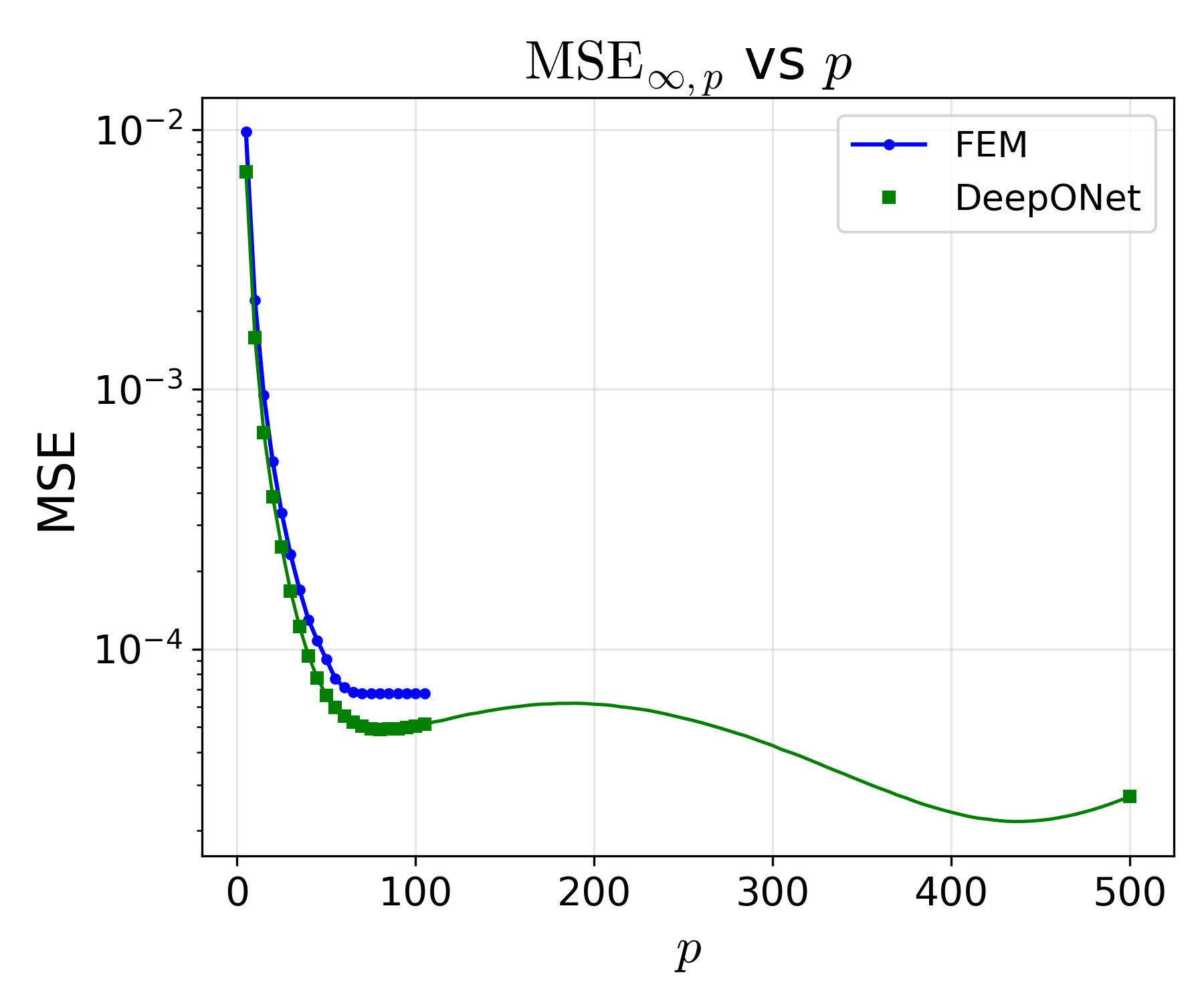}
\caption[]%
{{\small 3D cylinder}}    
\end{subfigure}
\begin{subfigure}[b]{0.33\textwidth}
\centering
\includegraphics[width=\textwidth]{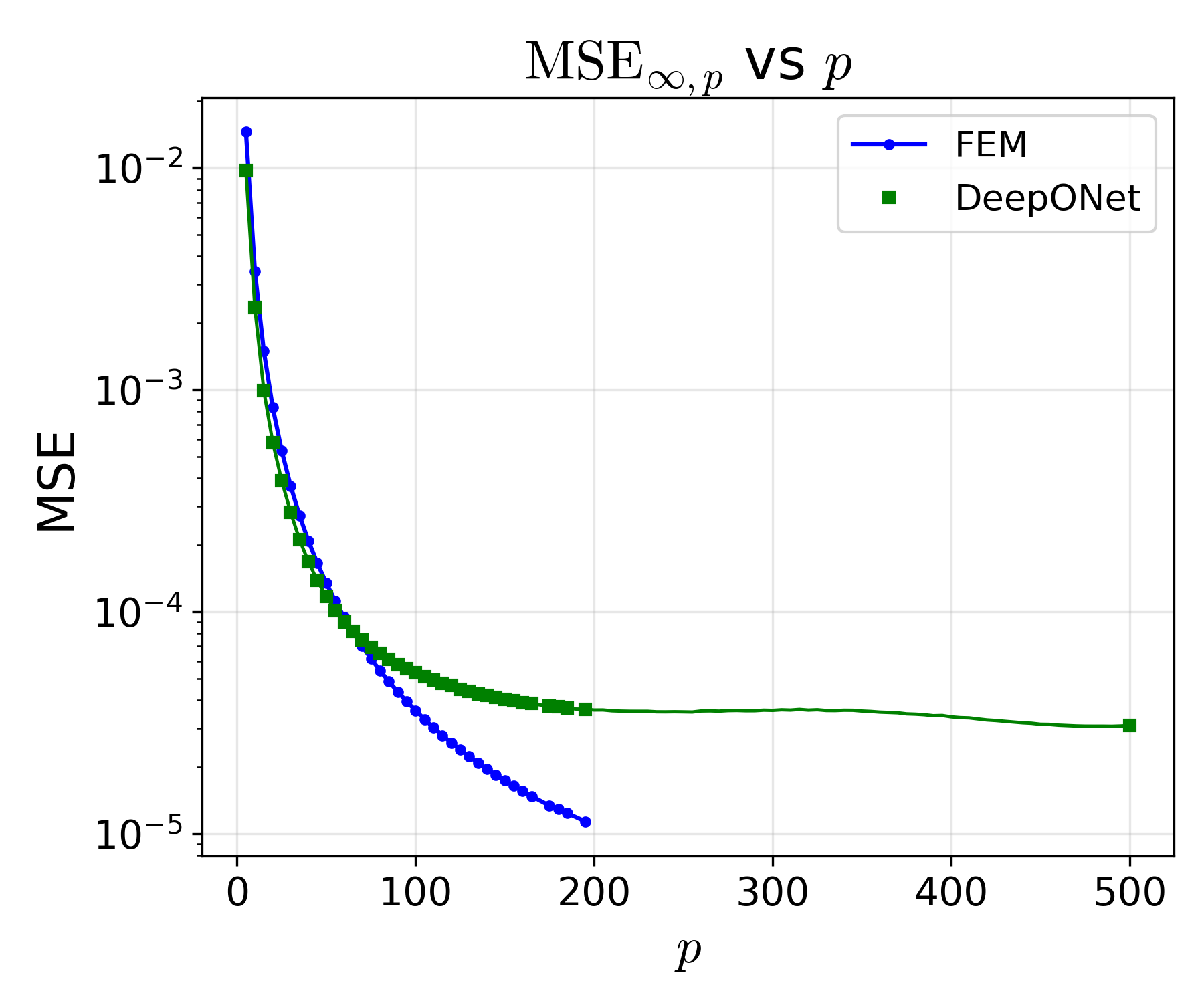}
\caption[]%
{{\small 3D torus}}    
\end{subfigure}
\caption[]%
{{\small Distance to boundary DeepONet results for various 3D domains. (a)–(c) show $\text{MSE}_{\infty,p}^{\text{FEM}}$ and $\text{MSE}_{\infty,p}$ over increasing $p$ values for: (a) the unit ball, (b) cylinder with radius $1$ and length $2$, and (c) a torus with major radius $2$ and minor radius $1$. All predictions are compared against the $p=\infty$ solution which is the true distance function.}}
\label{fig:training3.3}
\end{figure*}

\begin{table}[!h]
\begin{center}
\scalebox{0.9}{
\begin{tabular}{c|cc|cc|cc}
\hline
\hline
\multirow{3}{*}{Example} &  \multirow{3}{*}{Equation of domain} &  \multirow{3}{*}{$\|u_\infty\|_{2,N_{\mathrm{test}}}$}  & \multicolumn{2}{c|}{At largest training data $p$} &  \multicolumn{2}{c}{Runtime} \\
 & & & \multirow{2}{*}{$\text{MSE}_{\infty,p}^{\text{FEM}}$} & \multirow{2}{*}{$\text{MSE}_{\infty,p}$} & DeepONet & \multirow{2}{*}{Newton} \\
 & & & & & inference & \\
\hline
2D disc, Fig \ref{fig:training3.2}(a) & $x^2+y^2\leq 1$ & $4.085$e-01 & $8.468$e-06 & $5.248$e-06 & $2.265$e-04s & $8667$s\\
\hline
Ellipse 1, Fig \ref{fig:training3.2}(b) & $x^2+16y^2 \leq 1$ & $1.239$e-01 & $7.079$e-06 & $1.313$e-05  & $1.414$e-04s & \multirow{3}{*}{$\sim$3.5 hours}\\
Ellipse 2, Fig \ref{fig:training3.2}(c) & $8.5x^2+8.5y^2-15xy \leq 1$ & $1.237$e-01 & $7.079$e-06 & $1.066$e-05  & $1.406$e-04s & \\
Ellipse 3, Fig \ref{fig:training3.2}(d) & $4x^2+y^2 \leq 1$ & $2.398$e-01 & $5.711$e-06 & $4.435$e-06  & $1.416$e-04s & \\
\hline
3D ball, Fig \ref{fig:training3.3}(a) & $x^2+y^2+z^2 \leq 1$ & $3.167$e-01 &$1.894$e-05 & $5.995$e-06  & $2.948$e-03s & \multirow{3}{*}{$\sim$5 days} \\
3D cylinder, Fig \ref{fig:training3.3}(b) & $y^2+z^2 \leq 1, -1 \leq x \leq 1$ & $3.133$e-01 &$6.741$e-05 & $5.165$e-05  & $4.576$e-03s & \\
3D torus, Fig \ref{fig:training3.3}(c) & \begin{tabular}{@{}c@{}}$(2-\sqrt{x^2+z^2})^2+y^2\leq 1,$\\
$-3 \leq x,z \leq 3, -1 \leq y \leq 1$\end{tabular} & $4.076$e-01 &$1.136$e-05 & $3.618$e-05 & $7.401$e-04s &  \\
\hline
\hline
\end{tabular}
}
\end{center}
\caption{ Distance to boundary DeepONet example~\ref{subsec:deeponet-dist2bdry} results across various 2D and 3D domains. 
The second column includes the equation for each domain and the third column shows normalized discrete norm of $u_\infty$. 
The fourth and fifth columns show $\text{MSE}_{\infty,p}^{\text{FEM}}$ and $\text{MSE}_{\infty,p}$ errors respectively, at the largest training $p$ value. 
Lastly, the sixth and seventh columns show DeepONet inference times and total Newton FEM solver times. 
Both errors are measured against the limiting distance function, not against finite-$p$ solutions.
}
\label{tbl:mse5}
\end{table}

\begin{table}[h]
\captionsetup{font=scriptsize}
\begin{center}
\scalebox{0.76}{
\begin{tabular}{c|c|cc}
\hline
\hline
$p$ & Evaluation
& $\mathrm{MSE}_{p}$
& $\mathrm{MSE}_{\infty,p}$ \\
\hline
$5$   & Training          & $5.340$e-08 & $9.704$e-03 \\
$7$   & Interpolation     & $1.584$e-06 & $4.940$e-03 \\
$10$  & Training          & $4.621$e-08 & $2.269$e-03 \\
$12$  & Interpolation     & $9.675$e-08 & $1.540$e-03 \\
$25$  & Training          & $6.696$e-08 & $3.492$e-04 \\
$37$  & Interpolation     & $7.756$e-08 & $1.562$e-04 \\
$50$  & Training          & $8.080$e-08 & $8.571$e-05 \\
$77$  & Interpolation     & $8.697$e-08 & $3.613$e-05 \\
$100$ & Training          & $8.990$e-08 & $2.120$e-05 \\
$137$ & Interpolation     & $9.240$e-08 & $1.135$e-05 \\
$150$ & Training          & $9.299$e-08 & $9.505$e-06 \\
$177$ & Interpolation     & $9.433$e-08 & $6.820$e-06 \\
$200$ & Training          & $9.758$e-08 & $5.248$e-06 \\
$250$ & Augmented interp. & $1.299$e-07 & $3.043$e-06 \\
$300$ & Augmented interp. & $2.159$e-07 & $1.788$e-06 \\
$400$ & Augmented interp. & $5.356$e-07 & $6.456$e-07 \\
$500$ & Auxiliary input & $1.179$e-06 & $4.780$e-07 \\
\hline
\hline
\end{tabular}
}
\end{center}
\caption{Finite-$p$ validation of the unit-disc DeepONet model. The error $\mathrm{MSE}_{p}$ compares the DeepONet
prediction with the exact finite-$p$ solution $u_p$, while
$\mathrm{MSE}_{\infty,p}$ compares it with the limiting
solution $u_\infty$. }
\label{tbl:deeponet-finite-p-disc}
\end{table}

\subsubsection{Distance-to-boundary DeepONet over family of ellipses}

In this experiment, we incorporate geometric parameters defining an elliptical domain as inputs to the DeepONet model in order to assess its ability to generalize across a family of domains. The domain is defined by
\begin{align*}
    \Omega = \left\{ (x,y) \in \mathbb{R}^2 \; {\large:}\;
    \frac{(x\cos\theta + y\sin\theta)^2}{a^2}
    + \frac{(-x\sin\theta + y\cos\theta)^2}{b^2} < 1
    \right\}.
\end{align*}
We select the parameter ranges $a \in [0.9, 1.1]$, $b \in [0.2, 0.3]$, and
$\theta \in [0, \tfrac{\pi}{2}]$.
The rotation angle~$\theta$ is sampled at $11$ evenly spaced values
($10$ rotations plus the base ellipse at $\theta=0$).
The semi-axis parameters~$a$ and~$b$ are each sampled at $10$ evenly
spaced values within their respective ranges.
The model is trained for $20$ epochs with a total training time
of approximately $6{,}454$\,s.

The resulting solutions serve as the training data for the DeepONet model.
The mean squared error (MSE) across different parameter values is shown in
Figure~\ref{fig:training3.4}.

To assess the generalization of the trained DeepONet across the parameter
space, we evaluate two aggregate error measures derived from a full
three-dimensional sweep over $(\theta, a, b)$ at $p = 500$.
Let $\{\mathbf{x}_n\}_{n=1}^{N}$ denote the set of $N$ test points
obtained from the FEM mesh of the reference ellipse (after applying the
validity filter and a stride of~$10$).  For each triplet
$(\theta_i, a_j, b_k)$ on a uniform grid
$\theta_i \in [0,\,\pi/2]$, $a_j \in [0.9,\,1.1]$, $b_k \in [0.2,\,0.3]$,
we compute
\begin{equation}\label{eq:mse_ijk}
  \mathrm{MSE}(\theta_i, a_j, b_k)
  = \frac{1}{N} \sum_{n=1}^{N}
    \bigl(\hat{u}(R_{\theta_i}\mathbf{x}_n,\, p{=}500,\,\theta_i,\,a_j,\,b_k)
          - u_\infty(\mathbf{x}_n;\,a_j,\,b_k)\bigr)^{2},
\end{equation}
where $\hat{u}$ is the DeepONet prediction,
$R_{\theta}$ denotes counter-clockwise rotation by~$\theta$, and
$u_\infty(\mathbf{x};\,a,\,b) = \mathrm{dist}(\mathbf{x},\,
\partial\Omega_{a,b})$ is the exact $p \to \infty$ limit on the
(unrotated) ellipse $\Omega_{a,b} = \{(x,y) : x^2/a^2 + y^2/b^2 < 1\}$.
Note that $u_\infty$ is evaluated at the \emph{original} (unrotated)
coordinates because the distance to the boundary is invariant under
simultaneous rotation of the point and the domain, i.e.\
$\mathrm{dist}(R_\theta\mathbf{x},\,\partial R_\theta\Omega_{a,b})
= \mathrm{dist}(\mathbf{x},\,\partial\Omega_{a,b})$.
The two aggregate measures are then obtained by marginalizing
over the complementary variables:
\begin{equation}\label{eq:mse_theta}
  \overline{\mathrm{MSE}}(\theta_i)
  = \frac{1}{N_a \, N_b}\sum_{j=1}^{N_a}\sum_{k=1}^{N_b}
    \mathrm{MSE}(\theta_i,\, a_j,\, b_k),
\end{equation}
\begin{equation}\label{eq:mse_ab}
  \overline{\mathrm{MSE}}(a_j,\, b_k)
  = \frac{1}{N_\theta}\sum_{i=1}^{N_\theta}
    \mathrm{MSE}(\theta_i,\, a_j,\, b_k).
\end{equation}
Equation~\eqref{eq:mse_theta} quantifies the model's ability to maintain
rotational equivariance across the full range of ellipse shapes, while
equation~\eqref{eq:mse_ab} measures the prediction accuracy over varying
axis lengths for a given $(a,b)$ pair, averaged over all orientations.

The testing configurations are not included in the training set and therefore assess the
interpolation performance of the learned operator with respect to domain geometry.
Figure~\ref{fig:training3.4}(a) shows that the
$\overline{\mathrm{MSE}}(\theta)$ evaluated at $p=500$ is generally of order $10^{-4}$ across the full $\theta$ range, 
indicating that the DeepONet model can 
accurately approximate the limiting distance profile on previously unseen ellipse geometries.
This confirms that the model is able to learn a meaningful parametric representation of the
distance-to-boundary problem across a continuous family of domains.

\begin{figure*}[h]
\centering
\begin{subfigure}[b]{0.4\textwidth}
\centering
\includegraphics[width=\textwidth]{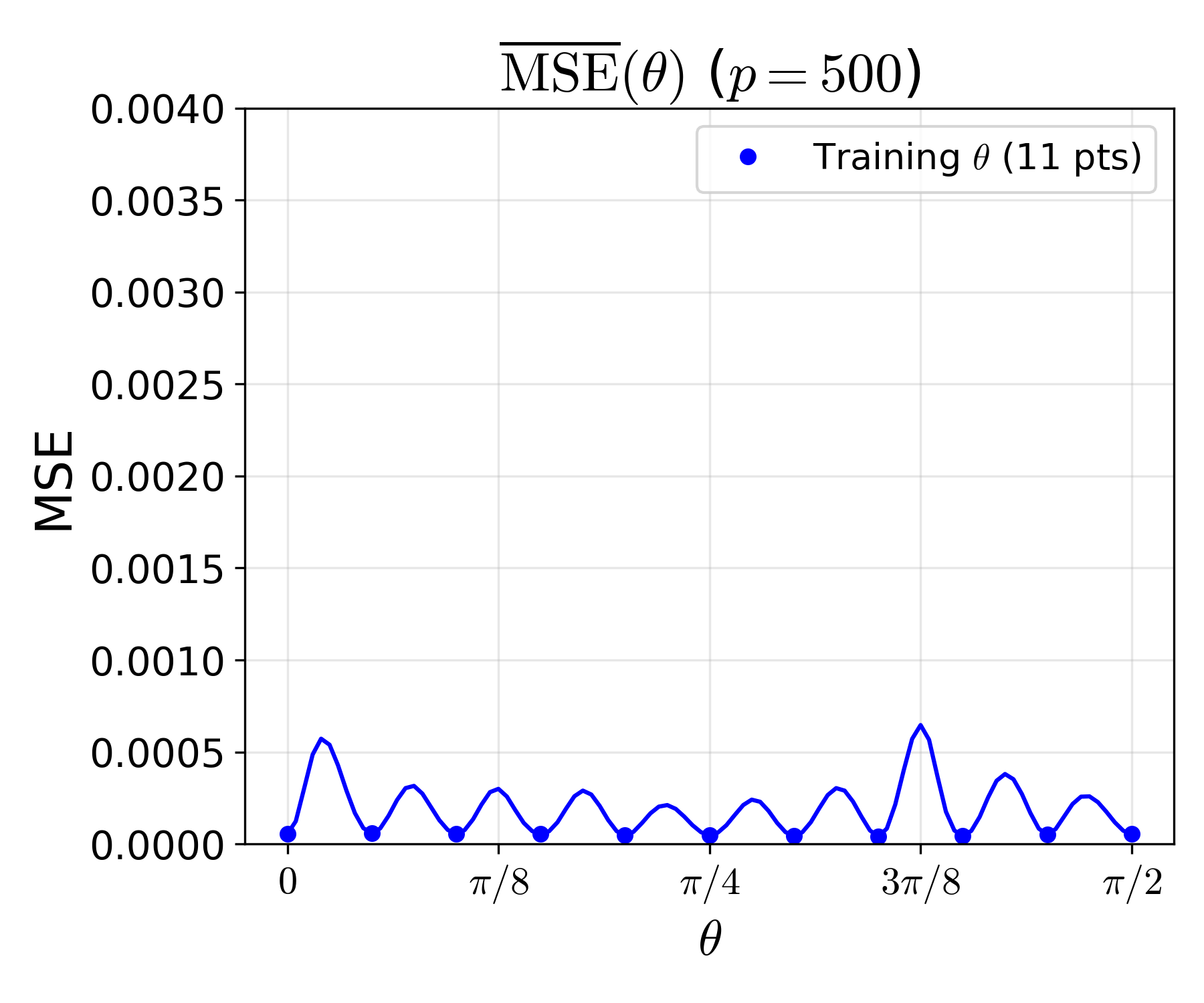}
\caption[]%
{{\small }}    
\end{subfigure}
\begin{subfigure}[b]{0.4\textwidth}  
\centering 
\includegraphics[width=\textwidth]{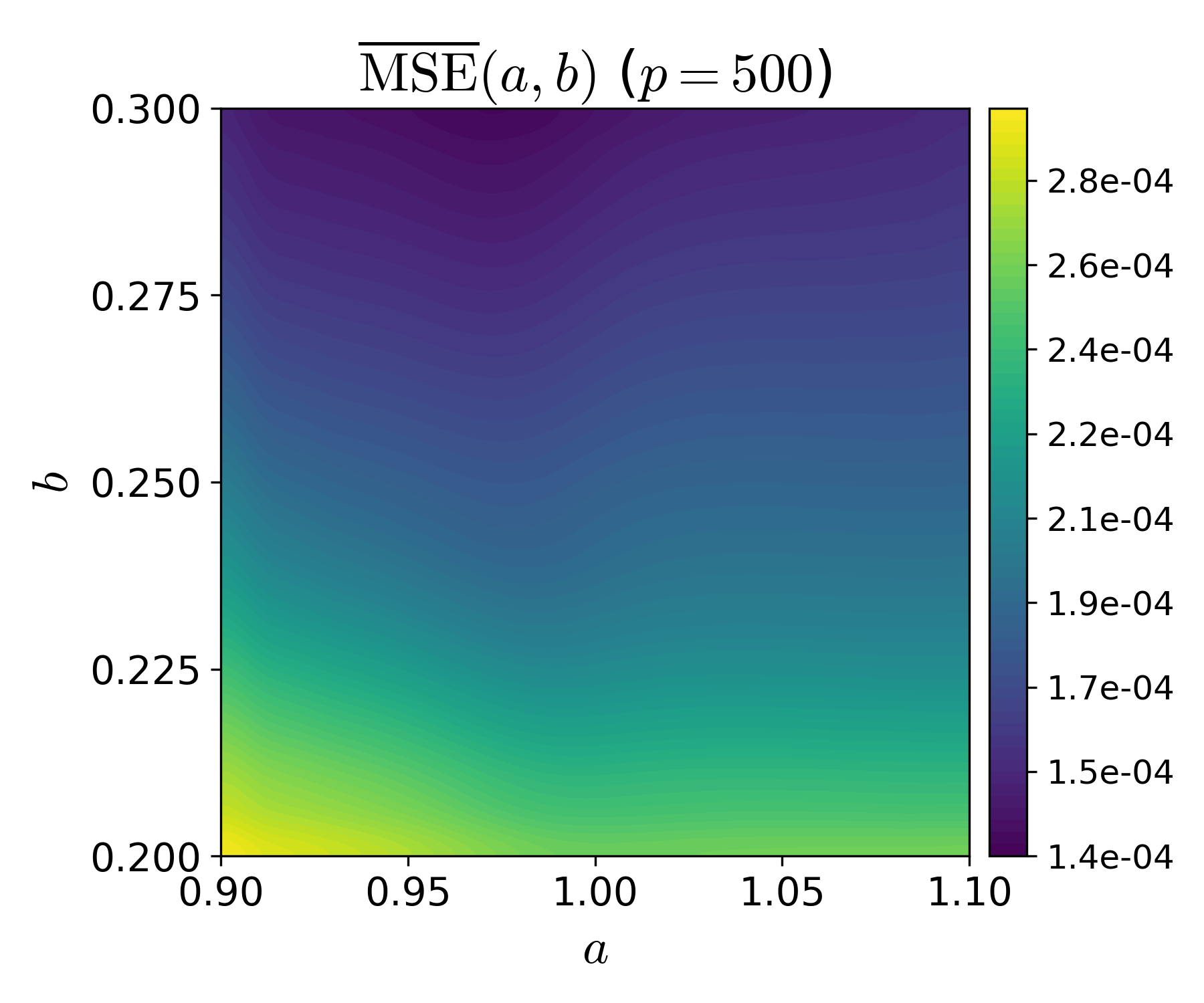}
\caption[]%
{{\small }}    
\end{subfigure}
\caption[]%
{{\small Distance to boundary DeepONet for all 2D ellipse examples.
(a) $\overline{\mathrm{MSE}}(\theta)$~\eqref{eq:mse_theta} as a function
of the rotation angle~$\theta$, evaluated at $p=500$. Blue points indicate
$\theta$ values included in the training data.
(b) $\overline{\mathrm{MSE}}(a,b)$~\eqref{eq:mse_ab} over varying semi-axis
lengths~$a$ and~$b$, evaluated at $p=500$.
The DeepONet model is trained over sampled values of $(a,b,\theta)$ and tested on unseen
configurations within the parameter range.}}
\label{fig:training3.4}
\end{figure*}

Furthermore, we investigate the effect of the number of training parameter samples on the
performance of the DeepONet model.
Specifically, we repeat the experiment using $3$, $5$, $7$, and $9$ evenly spaced values of
$\theta$ within the prescribed range while keeping the remaining parameters fixed.
This experiment isolates the influence of sampling density in the parameter space on the
interpolation quality of the learned operator.

The results in Figure~\ref{fig:training3.5} show that when the number of training points is too
small, the DeepONet model tends to overfit the available data, leading to poor predictions at
intermediate values of $\theta$.
As the number of training points increases, the approximation improves and the error profile
becomes smoother across the parameter range.
These observations highlight the importance of sufficient parameter sampling when training
operator-learning models on families of geometries.

\begin{figure*}[h]
\centering
\begin{subfigure}[b]{0.24\textwidth}
\centering
\includegraphics[width=\textwidth]{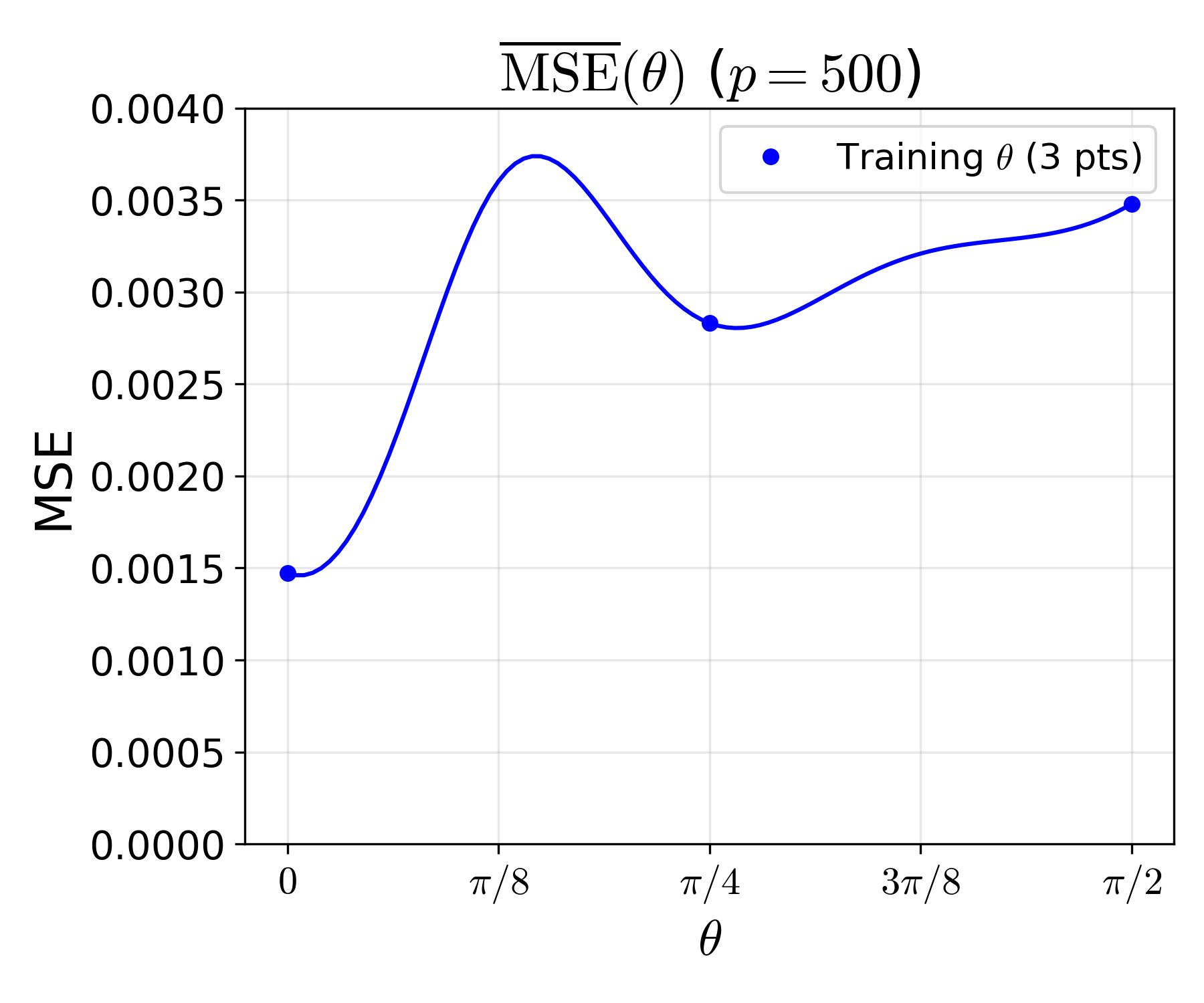}
\caption[]%
{{\small $3$ points}}    
\label{fig:ellipse_3point}
\end{subfigure}
\begin{subfigure}[b]{0.243\textwidth}  
\centering 
\includegraphics[width=\textwidth]{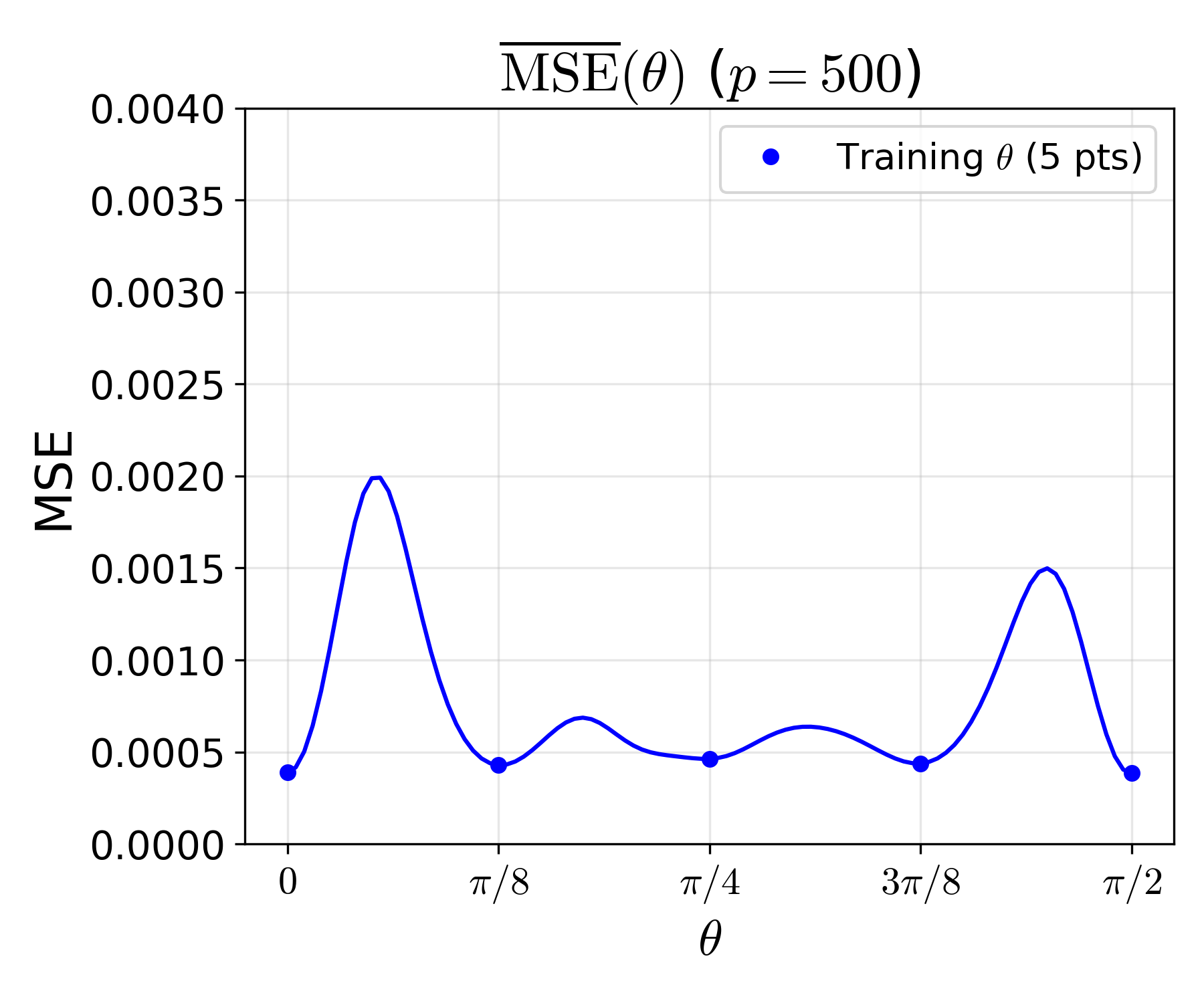}
\caption[]%
{{\small $5$ points}}    
\label{fig:ellipse_5point}
\end{subfigure}
\begin{subfigure}[b]{0.241\textwidth}
\centering
\includegraphics[width=\textwidth]{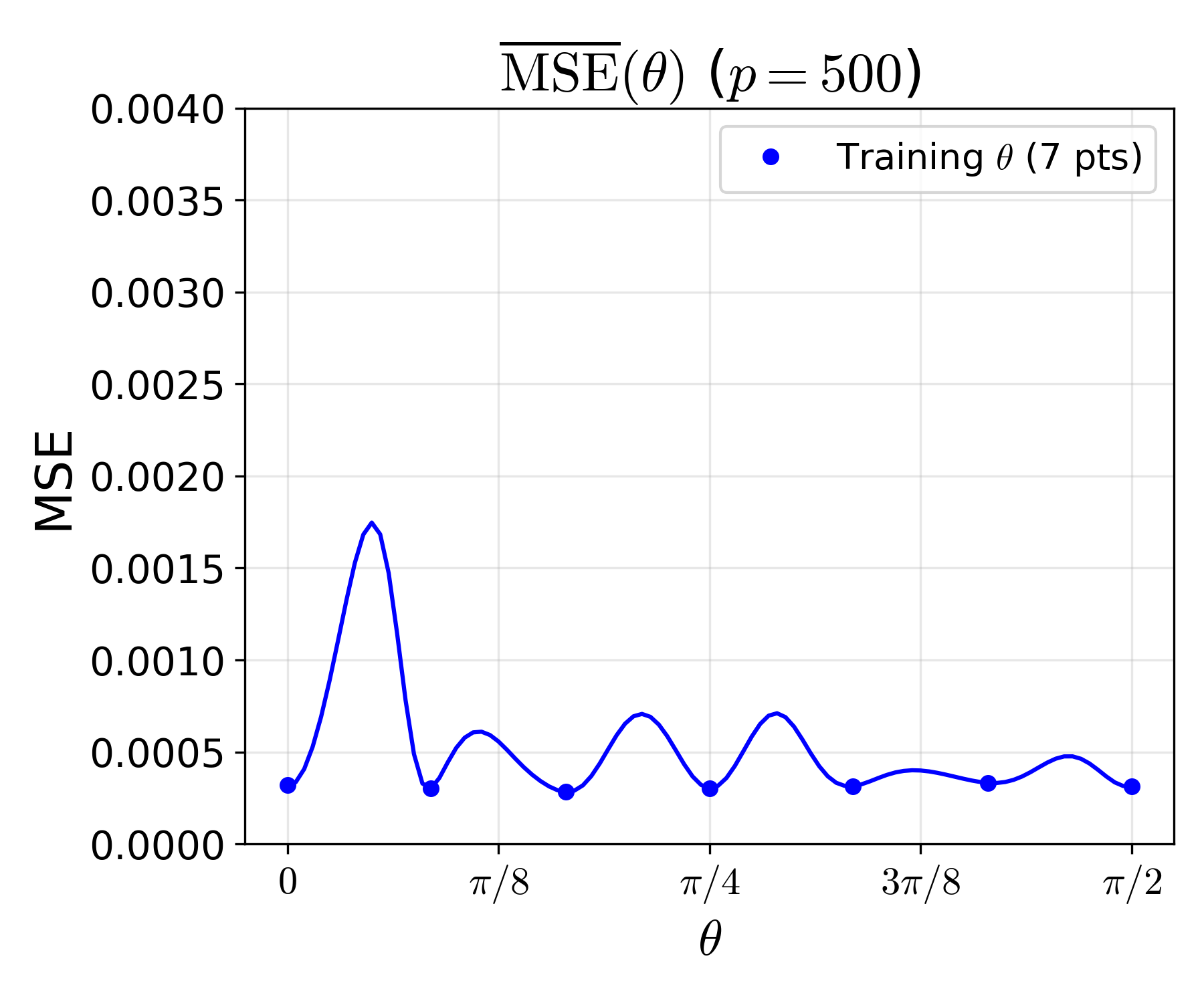}
\caption[]%
{{\small $7$ points}}    
\label{fig:ellipse_7point}
\end{subfigure}
\begin{subfigure}[b]{0.241\textwidth}  
\centering 
\includegraphics[width=\textwidth]{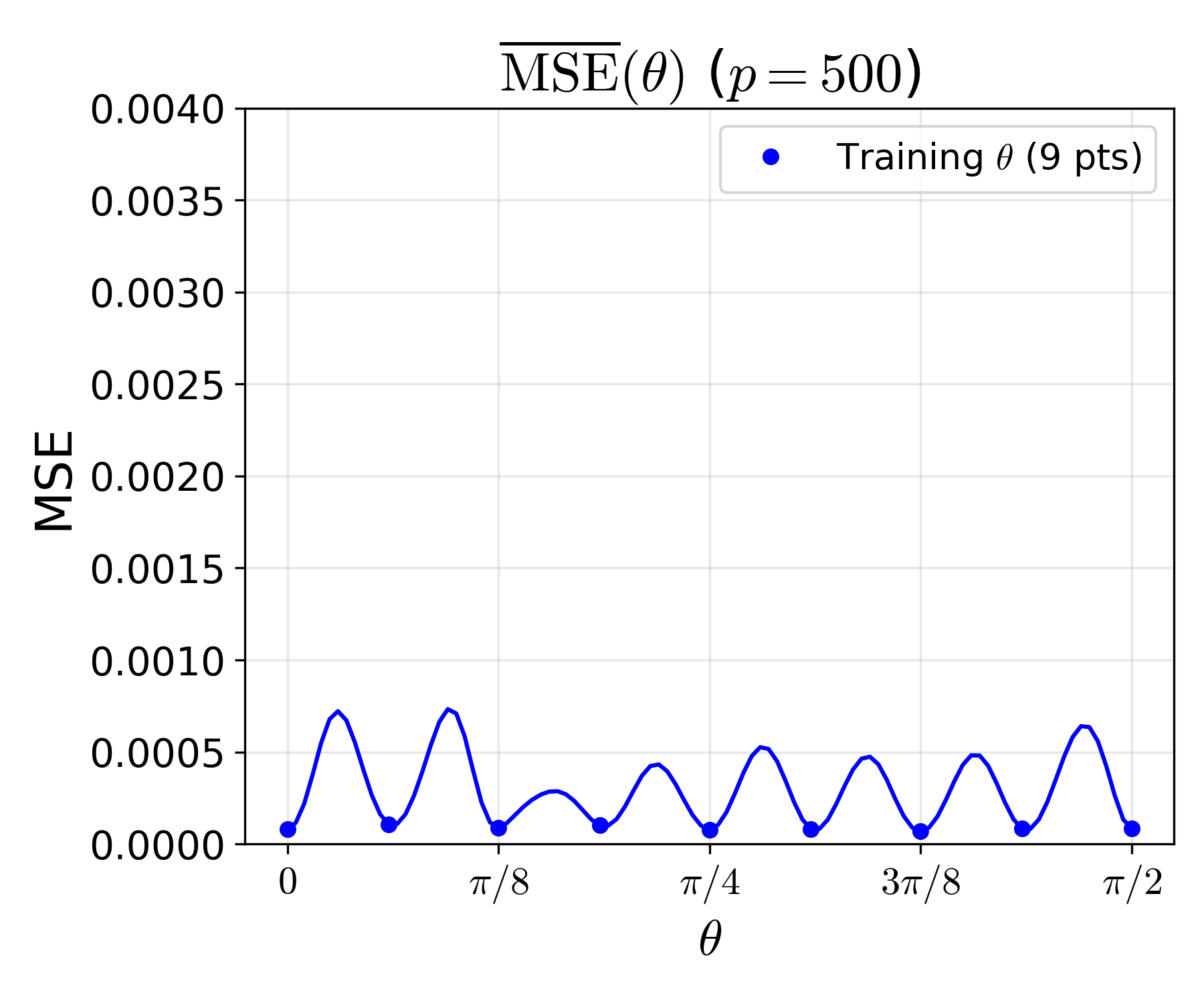}
\caption[]%
{{\small $9$ points}}    
\label{fig:ellipse_9point}
\end{subfigure}
\caption[]%
{{\small Distance to boundary DeepONet for all 2D ellipse examples.
(a)--(d) show $\overline{\mathrm{MSE}}(\theta)$~\eqref{eq:mse_theta} using
$3$, $5$, $7$, and $9$ training $\theta$~values, respectively.
Blue points indicate $\theta$ values included in the training data.}}
\label{fig:training3.5}
\end{figure*}

\section{Conclusions}
\label{section:conclusion}

Our numerical experiments demonstrate that PINNs and DeepONet can accurately recover infinity-Laplace solutions and capture the limiting behavior of $p$-Laplace solutions as $p\to\infty$. Across the finite-$p$ experiments, the neural-network predictions reproduce the limiting profiles with errors comparable to the corresponding FEM convergence-to-limit baselines, while remaining effective over a substantially wider range of $p$ values.
In particular, for the PINNs approach, we observe that an iterative training strategy in $p$,
combined with $u$-gradient norm clipping,
significantly improves stability and robustness of the optimization as $p$ increases.

DeepONet offers distinct advantages. Unlike PINNs, a single DeepONet model can incorporate
both the parameter $p$ and geometric parameters describing the domain,
allowing solutions across a family of PDEs and domains to be represented within one trained
operator.
This capability is especially well suited to distance approximation problems, where we may
wish to evaluate $p$-Poisson distances via the $p$-Laplacian for many values of $p$ and across multiple geometries.
Additionally, we find that explicitly incorporating information from the $p \to \infty$ limiting
problem can improve performance in the large $p$ regime,
leading to more reliable interpolation/extrapolation in the $p$ parameter space beyond the range of available
finite $p$ training data.

At the same time, PINNs retain the important practical advantage that precomputed solution data in the interior of the computational domain is not required and
can be trained directly from the governing equations and boundary conditions alone.
This makes PINNs particularly attractive in settings where generating high-quality training data is expensive or infeasible.

The success of the neural network solvers provides exciting avenues for future studies.  One natural extension is the transition from Euclidean domains to surface formulations involving
intrinsic differential operators, which would enable the approximation of $p$-Poisson geodesic
distances as studied in~\cite{Potgieter}.
Additional applications include image processing~\cite{Chen2006}, elastoplastic torsion~\cite{AlvarezFlores}, and non-Newtonian fluid models~\cite{Ruzicka},
where nonlinear and degenerate elliptic operators pose significant challenges for classical
numerical solvers.

\section*{Code and data availability}
The code, data-generation scripts, plotting scripts, and experiment configuration files used in this work are available at
\url{https://github.com/hannah-potgieter/infinity-and-p-Laplace-solvers}.
The repository includes fixed random-seed settings for the neural-network experiments so that the reported runs and figures can be reproduced from the same scripts and input data.

\section*{Funding/Acknowledgments}
We acknowledge the support of the Natural Sciences and Engineering Research Council of Canada (NSERC), RGPIN-2022-03302.

\section*{CRediT authorship contribution statement}

\textbf{Tak Shing Au Yeung:} Writing -- original draft, Writing -- review \& editing, Software, Formal analysis, Visualization.

\textbf{Ka Chun Cheung:} Writing -- review \& editing, Conceptualization, Methodology.

\textbf{Hannah Potgieter:} Writing -- original draft, Writing -- review \& editing, Software, Formal analysis, Data curation.

\textbf{Steven J. Ruuth:} Writing -- review \& editing, Conceptualization, Methodology, Supervision.

\textbf{Simon See:} Writing -- review \& editing, Conceptualization.

\section*{Declaration of competing interests}
The authors declare that they have no known competing financial interests or personal relationships that could have appeared to influence the work reported in this paper.

\section*{Declaration of generative AI and AI-assisted technologies in the manuscript preparation process}
During the preparation of this work, the authors used ChatGPT (OpenAI) and Claude (Anthropic) to assist with language editing, including improvements to grammar and phrasing, and to assist in checking mathematical calculations. After using these tools/services, the authors reviewed and edited the content as needed and take full responsibility for the content of the published article.

\bibliographystyle{model1-num-names}
\bibliography{refs}

\appendix

\section{Training hyperparameters}
\label{app:hyperparams}

Tables~\ref{tbl:hyperparams} and~\ref{tbl:deeponet-hyperparams} list the training hyperparameters used for the PINN experiments in Sections~\ref{subsec:2Dinflap-dbc}--\ref{subsec:plaplacian-dbc} and the DeepONet experiments in Sections~\ref{subsec:deeponet-dist2orig}--\ref{subsec:deeponet-dist2bdry} respectively.

\begin{table*}[h]
  \begin{minipage}[t]{0.49\textwidth}
  \centering
  \scalebox{0.78}{
  \begin{tabular}{l|l}
  \hline\hline
  Hyperparameter & Value \\
  \hline
  Optimizer & Adam \\
  Learning rate & $10^{-3}$ \\
  LR schedule & Cosine, $10^{-3} \to 10^{-5}$ \\
  Epochs & $100$ \\
  Validation & $20\%$ split\\
  Grad clipping & max\_norm $= 1.0$ \\
  Batch sizes & BC: $400$,\, PDE: $1000$ \\
  Initialization & Xavier normal, zero biases \\
  Hardware & Xeon Gold 5420+, RTX 5880 Ada \\
  \hline\hline
  \end{tabular}
  }
  \caption{PINNs training hyperparameters.}
  \label{tbl:hyperparams}
  \end{minipage}%
  \hfill
  \begin{minipage}[t]{0.49\textwidth}
  \centering
  \scalebox{0.78}{
  \begin{tabular}{l|l}
  \hline\hline
  Hyperparameter & Value \\
  \hline
  Trunk network & $[\text{input},512,512,512,128]$, $\tanh$ \\
  Branch network & $[\text{input},128,128,128,128]$, $\tanh$ \\
  Combination & Elementwise product $\to$ Linear$(128,1)$ \\
  Optimizer & Adam, $\text{lr} = 10^{-4}$ \\
  LR schedule & Cosine, $10^{-4} \to 10^{-6}$ \\
  Epochs & $20$ \\
  Batch size & $2048$ \\
  Grad clipping & max\_norm $= 1.0$ \\
  Hardware & Xeon Gold 5420+, RTX 5880 Ada \\
  \hline\hline
  \end{tabular}
  }
  \caption{DeepONet training hyperparameters.}
  \label{tbl:deeponet-hyperparams}
  \end{minipage}
  \end{table*}

All PINN and DeepONet inference times reported in Section~\ref{section:expts} are measured using a batched forward pass over the complete test set. Ten warm-up passes are excluded, the elapsed time is averaged over $100$ forward passes, and CUDA is synchronized immediately before and after each timed loop. Input transfer to the GPU and postprocessing are excluded. \label{app:inference-timing}

\section{Ablation study for infinity Laplacian PINNs}
\label{app:ablation}

We perform the ablation study on three test configurations: the Arctan example and the Aronsson example on the square and disc domains.
Table~\ref{tbl:ablation} reports the best-model test MSE for each configuration.
Lower values indicate better accuracy.

\begin{table}[h]
\begin{center}
\scalebox{0.78}{
\begin{tabular}{l|ccc}
\hline\hline
& \multicolumn{3}{c}{Single-network ($4\times128$ reference)}  \\
Variant & Arctan & Aronsson (square) & Aronsson (disc)  \\
\hline
normal (baseline) & $2.42$e-07 & $5.74$e-07 & $4.77$e-07  \\
\hline
\multicolumn{4}{l}{\textit{Stabilization ablations}} \\
\hline
clip0              & $1.26$e-07 & $5.74$e-07 & $4.77$e-07  \\
out0               & $1.84$e-06 & $6.17$e-06 & $1.00$e-04  \\
eta0               & $9.47$e-06 & $4.01$e-07 & $3.14$e-07  \\
clip0\_out0        & $1.90$e-06 & $5.35$e-05 & $1.74$e-04  \\
clip0\_out0\_eta0  & $2.69$e-04 & $3.51$e-06 & $6.09$e-06  \\
\hline
\multicolumn{4}{l}{\textit{Loss balancing ablations}} \\
\hline
alpha0             & $3.06$e-04 & $4.80$e-02 & $4.06$e-02  \\
alpha\_fixed\_1e-1 & $1.27$e-06 & $1.62$e-07 & $1.47$e-07  \\
alpha\_fixed\_1e-5 & $1.67$e-06 & $7.73$e-06 & $6.62$e-06  \\
relobralo          & $2.93$e-06 & $8.85$e-01 & $1.71$e+00  \\
\hline
\multicolumn{4}{l}{\textit{LR schedule ablation}} \\
\hline
nosched            & $3.19$e-07 & $1.14$e-06 & $2.36$e-06 \\
\hline
\multicolumn{4}{l}{\textit{Network width ablations}} \\
\hline
smaller (h32)  & $1.22$e-06 & $8.01$e-07 & $1.44$e-06  \\
larger (h256) & $8.31$e-07 & $5.57$e-07 & $5.70$e-07  \\
\hline
\multicolumn{4}{l}{\textit{Network depth ablations}} \\
\hline
depth2  & $1.68$e-04 & $2.12$e-06 & $1.35$e-06  \\
depth3 & $2.20$e-06 & $7.24$e-07 & $6.04$e-07  \\
depth5 & $1.80$e-07 & $4.54$e-07 & $4.60$e-07  \\
\hline\hline
\end{tabular}
}
\end{center}
\caption{Ablation study: best-model test MSE for the infinity Laplacian PINNs.}
\label{tbl:ablation}
\end{table}

\paragraph{Stabilization}
The baseline uses three stabilization techniques:
$\eta$-normalization~\eqref{eq:normalized_bound}, residual clipping to
$[-10,10]$, and top-$2\%$ outlier removal. We assess the importance of these
components by disabling them individually and in combination. Disabling
outlier removal has the largest adverse effect, increasing the error by
approximately $200\times$ for the Aronsson disc example because high-residual
outliers strongly influence the PINN loss. The effect of $\eta$-normalization
is problem-dependent: it improves the Arctan result but slightly degrades the
Aronsson results. Residual clipping has little additional effect when outlier
removal is active. Disabling all three components increases the errors in all three test configurations.

\paragraph{Loss balancing}
Setting $\alpha=0$ disables the PDE loss and causes the errors to increase
substantially in all cases because the boundary loss alone does not enforce
the governing equation. We also test a fixed high weight
($\alpha=10^{-1}$) and a fixed low weight ($\alpha=10^{-5}$). The high weight
improves some cases but degrades the Arctan result, whereas the low weight
degrades all cases. ReLoBRaLo \cite{Bischof_2025} produces substantially larger errors for all three test configurations.

\paragraph{LR schedule}
Disabling cosine annealing (\texttt{nosched}) causes a slight degradation,
although its effect is smaller than that of the other ablations.

\paragraph{Width and depth}
Network width has a limited effect: a width of $32$ performs similarly to
widths of $128$ and $256$ in most cases. Increasing the network depth improves
performance for the smooth examples, but the four- and five-layer networks
perform similarly.

\section{Ablation study for simple $p$-Laplacian PINNs}
\label{app:ablation-plaplace}

We test four stabilization techniques for simple $p$-Laplacian PINNs at $p=10$ and $p=200$, including residual clipping (\texttt{clip0}), gradient norm clamping (\texttt{gradclip0}), top-$2\%$ outlier removal (\texttt{out0}), and $\eta$-regularization (\texttt{eta0}).
Table~\ref{tbl:ablation-plaplace} reports the best-model $\text{MSE}_\infty$ for each configuration.

\begin{table}[h]
\captionsetup{font=scriptsize}
\begin{center}
\scalebox{0.78}{
\begin{tabular}{l|c|c}
\hline
\hline
Configuration & $p=10$ & $p=200$ \\
\hline
baseline (all enabled) & $1.057$e-04 & $1.652$e-02 \\
\hline
\multicolumn{3}{l}{\textit{Single component disabled}} \\
\hline
clip0 (no residual clipping) & $1.057$e-04 & $3.159$e-03 \\
out0 (no outlier removal) & $1.539$e-04 & $1.692$e-02 \\
eta0 (no $\eta$-regularization) & $1.057$e-04 & $1.652$e-02 \\
gradclip0 (no gradient norm clipping) & $1.931$e-04 & $-$ \\
\hline
\multicolumn{3}{l}{\textit{Two components disabled}} \\
\hline
clip0\_gradclip0 & $1.757$e-04 & $-$ \\
gradclip0\_out0 & $1.894$e-04 & $-$ \\
gradclip0\_eta0 & $1.931$e-04 & $-$ \\
\hline
\multicolumn{3}{l}{\textit{Three or more components disabled}} \\
\hline
clip0\_gradclip0\_out0 & $1.870$e-04 & $-$ \\
clip0\_gradclip0\_eta0 & $1.757$e-04 & $-$ \\
gradclip0\_out0\_eta0 & $1.894$e-04 & $-$ \\
clip0\_gradclip0\_out0\_eta0 & $1.870$e-04 & $-$ \\
\hline
\hline
\end{tabular}
}
\end{center}
\caption{Ablation study for the simple $p$-Laplacian PINNs (Aronsson
  example on square domain).  Each row disables one or more
  stabilization techniques.  "$-$" indicates that training diverged and therefore produced no finite MSE value.}\label{tbl:ablation-plaplace}
\end{table}

Gradient-norm clamping is the most critical component. Disabling it causes
training to diverge at $p=200$. At $p=10$, disabling gradient-norm clamping
degrades performance in every tested combination that omits it. Outlier removal has a modest effect at $p = 10$ and a negligible effect at $p = 200$.
In contrast, $\eta$-regularization has little effect at either value of $p$.

\section{Iterative training schedule for the $p$-Laplacian}
\label{app:iter-schedule}

Table~\ref{tbl:iter-schedule} details the five-band continuation strategy used for the iterative $p$-Laplacian experiments in Section~\ref{subsec:plaplacian-dbc}. 
The finite-$p$ validation on the unit disc uses the separate continuation schedule reported in Table~\ref{tbl:iter-schedule-disc}.

Each band loads the best-model checkpoint
produced by the previous band and trains on the next group of $p$ values. The
first band uses the same scheduler as the infinity-Laplacian experiments and
may terminate early when both the validation boundary loss falls below
$10^{-4}$ and the PDE residual loss falls below $10^{-3}$. Bands~2--5 use
fixed values of $\alpha$ and a constant learning rate of $10^{-5}$ for a
prescribed number of epochs at each value of $p$.

\begin{table}[h]
\begin{center}
\scalebox{0.78}{
\begin{tabular}{c|l|c|l|l}
\hline\hline
Band & $p$ values & Epochs/$p$ & LR schedule & $\alpha$ schedule \\
\hline
1 & $2,3,\dots,10,15,20$ (11 values) & 100 &
  Cosine $10^{-3}\!\to\!10^{-5}$ &
  AlphaPlateauScheduler (adaptive) \\
2 & $25,30,40,50$ & 20 &
  Constant $10^{-5}$ &
  Fixed $\alpha=10^{-2}$ \\
3 & $60,70,80,90,100$ & 20 &
  Constant $10^{-5}$ &
  Fixed $\alpha=10^{-3}$ \\
4 & $150,200,300,400,500$ & 20 &
  Constant $10^{-5}$ &
  Fixed $\alpha=10^{-4}$ \\
5 & $600,700,800,900,1000$ & 20 &
  Constant $10^{-5}$ &
  Fixed $\alpha=10^{-5}$ \\
\hline\hline
\end{tabular}
}
\end{center}
\caption{Five-band iterative training schedule for the homogeneous $p$-Laplacian
experiments.  Band~1 uses early stopping; bands~2--5 run for a fixed number
of epochs per $p$ value.}
\label{tbl:iter-schedule}
\end{table}

The unit-disc experiment in Section~\ref{subsec:pinn-finite-p-disc} uses the
same four-hidden-layer network with $128$ neurons per layer, $\tanh$
activation, Adam optimizer, and boundary and PDE batch sizes of $400$ and
$1000$, respectively. The training residual is based on the flux-regularized
$p$-Poisson operator with $\eta=10^{-5}$. The residual is clipped to
$[-100,100]$, the largest $2\%$ of the pointwise PDE losses are removed, and
both the solution-gradient cap and parameter-gradient clipping threshold are
set to $1$. The first continuation stage uses a $20\%$ boundary-validation split and may
terminate early when the boundary and PDE loss thresholds are satisfied.
Subsequent stages use all boundary points and the final checkpoint from the
preceding stage. Interior collocation points are resampled at every epoch.

\begin{table}[h]
\begin{center}
\scalebox{0.70}{
\begin{tabular}{c|l|c|r|l|l}
\hline\hline
Stage & $p$ values & Epochs/$p$ & Interior points
& LR schedule & $\alpha$ schedule \\
\hline
1 & $2,3,\ldots,10$ & up to $100$ & $783{,}764$
& Cosine $10^{-3}\!\to\!10^{-5}$
& AlphaPlateauScheduler \\

2 & $11,12,\ldots,20$ & $50$ & $195{,}496$
& Cosine $10^{-4}\!\to\!10^{-5}$
& Fixed $10^{-2}$ \\

3 & $20$ (refinement) & $50$ & $783{,}764$
& Cosine $10^{-4}\!\to\!10^{-5}$
& Fixed $10^{-2}$ \\

4 & $25,30,40,50$ & $50$ & $195{,}496$
& Cosine $10^{-4}\!\to\!10^{-5}$
& Fixed $10^{-2}$ \\

5 & $60,70,80,90,100$ & $50$ & $195{,}496$
& Cosine $10^{-4}\!\to\!10^{-5}$
& Fixed $10^{-2}$ \\

6 & $150,200,300,400,500$ & $50$ & $195{,}496$
& Constant $10^{-5}$
& Fixed $10^{-2}$ \\

7 & $600,700,800,900,1000$ & $50$ & $195{,}496$
& Constant $10^{-5}$
& Fixed $10^{-2}$ \\
\hline\hline
\end{tabular}
}
\end{center}
\caption{Continuation schedule used for the direct finite-$p$ iterative PINN
validation on the unit disc. The experiment uses $10{,}000$ boundary points.
Stage~1 uses $8{,}000$ training and $2{,}000$ validation boundary points;
subsequent stages use all $10{,}000$ boundary points.}
\label{tbl:iter-schedule-disc}
\end{table}

\section{Eikonal PINNs training for DeepONet data generation}
\label{app:eikonal-pinns}

The DeepONet examples in Sections~\ref{subsec:deeponet-dist2orig} and~\ref{subsec:deeponet-dist2bdry} 
use the limiting solution $u_\infty$ as the training data for $p=\infty$.
When the exact solution is unavailable, we train a PINN to solve the Eikonal equation
$$
|\nabla u|^2=1 \quad \text{in }\Omega.
$$
For the distance-to-origin examples on the disc and square, we impose $u=0$ at the origin.
For the distance-to-boundary examples, including the 2D and 3D domains, we impose $u=0$ on $\partial\Omega$.

\paragraph{Positivity constraint}
Since we know the distance solution must satisfy $u\geq0$, we add a soft positivity constraint
$$
r_{\mathrm{pos}} = w_{\mathrm{pos}}\operatorname{ReLU}(-u_\theta),
$$
where $w_{\mathrm{pos}}$ is listed in Table~\ref{tbl:eikonal-domains}.
Top-$2\%$ outlier removal is applied only to the Eikonal residual, 
while the positivity residual is not affected by stabilization techniques.

\paragraph{Shared hyperparameters}
The Eikonal PINNs use the same configuration as the infinity- and
$p$-Laplacian PINNs listed in~\ref{app:hyperparams}. The network has four
hidden layers with $128$ neurons each and uses the $\tanh$ activation function
and Xavier normal initialization. Training uses Adam with an initial learning
rate of $10^{-3}$, cosine-annealed to $10^{-5}$ over $100$ epochs, together
with a $20\%$ boundary-validation split. The boundary and PDE batch sizes are
$400$ and $1000$, respectively, and the gradient norm is clipped to $1.0$.
The loss weight $\alpha$ uses the same adaptive scheduler. We also clip the PDE
residual to $[-10,10]$ and remove the largest $2\%$ of the PDE losses.

Table~\ref{tbl:eikonal-domains} reports the grid sizes, positivity weight, training time, and best test MSE for each domain.

\begin{table}[h]
\captionsetup{font=scriptsize}
\begin{center}
\scalebox{0.78}{
\begin{tabular}{l|r r r|c|r|c}
\hline\hline
Domain & Interior & Boundary & Test & $w_{\text{pos}}$ & Time (s) & Test MSE \\
\hline
\multicolumn{7}{l}{\emph{Distance-to-origin (2D)}} \\
Disc    & 783{,}764   & 1{,}100  & 7{,}668  & 1.0   & 524  & $5.948\times10^{-7}$ \\
Square  & 1{,}000{,}000 & 1{,}100  & 10{,}000 & 1.0   & 691  & $2.985\times10^{-7}$ \\
\hline
\multicolumn{7}{l}{\emph{Distance-to-boundary (2D)}} \\
Disc      & 783{,}764     & 10{,}000 & 7{,}668  & 1.0   & 523  & $5.749\times10^{-7}$ \\
Square    & 1{,}000{,}000 & 40{,}000 & 10{,}000 & 1.0   & 681  & $1.713\times10^{-6}$ \\
Ellipse~1 & 998{,}912     & 10{,}000 & 1{,}588  & 100.0 & 682  & $2.395\times10^{-6}$ \\
Ellipse~2 & 998{,}852     & 10{,}000 & 1{,}588  & 100.0 & 753  & $5.838\times10^{-6}$ \\
Ellipse~3 & 998{,}511     & 10{,}000 & 3{,}172  & 100.0 & 694  & $1.077\times10^{-5}$ \\
\hline
\multicolumn{7}{l}{\emph{Distance-to-boundary (3D)}} \\
Ball   & 507{,}376 & 20{,}000 & 61{,}432 & 0.5 & 320  & $1.267\times10^{-6}$ \\
Cylinder & 766{,}800 & 19{,}998 & 90{,}048 & 0.5 & 524  & $2.865\times10^{-5}$ \\
Torus    & 453{,}208 & 20{,}000 & 54{,}856 & 0.5 & 312  & $3.974\times10^{-6}$ \\
\hline\hline
\end{tabular}
}
\end{center}
\caption{Eikonal PINNs training configuration per domain.
Interior and boundary columns give the total number of points
before the $80/20$ train/validation split.
Test MSE is the best-checkpoint mean squared error against the
exact distance function on the listed test grid.}
\label{tbl:eikonal-domains}
\end{table}

\end{document}